\documentclass[12pt]{article}

\usepackage{mathptmx,amsmath,amssymb,amsfonts}
\usepackage{amsthm} %BB03dec18 
\usepackage{graphicx,color,xcolor,pictex,epstopdf,url,algorithm,algorithmic,caption}
\usepackage{endnotes}
\usepackage{wrapfig}

\usepackage{etoolbox} % another kludge to over-ride Latex defaults
\makeatletter
\patchcmd{\@makechapterhead}{\large}{\normalsize}{}{}% for \chapter
\patchcmd{\@makechapterhead}{\large}{\normalsize}{}{}% for \chapter
\patchcmd{\@makeschapterhead}{\normalsize}{\normalsize}{}{}% for \chapter*
\makeatother
\usepackage[textwidth=6.0in,textheight=9in,left=0.75in,right=0.75in,top=0.75in,bottom=0.75in]{geometry}
\usepackage[section]{placeins}
\makeatletter
\g@addto@macro\normalsize{\setlength\abovedisplayskip{4pt}}
\g@addto@macro\normalsize{\setlength\belowdisplayskip{4pt}}
\makeatother
\newtheorem{theorem}{Theorem}[section]
\newtheorem{lemma}{Lemma}[section]
\newtheorem{corollary}{Corollary}[section]

\newtheorem{example}{Example}[section]

\newtheorem{remark}{Remark}[section]
\usepackage{parskip}
\let\oldref\ref
\renewcommand{\ref}[1]{(\oldref{#1})}  % stupid kludge for laTeX
\renewcommand{\eqref}[1]{(\oldref{#1})} %BB03dec18 
\font\titlefonrm=ptmb scaled \magstep3
\newbox\boxaddrone \newbox\boxaddrtwo
\newbox\boxcaption
\newbox\boxtemp

\def\dH#1{\dot H^{#1}(\Omega)}

\def\bdryd#1{{\small {\frac{\partial #1}{\partial\nu}}}}\def\wfunc{w}

\newcommand{\Margin}[1]{}
\newcommand{\revision}[1]{#1}

\begin{document}

\title{{\titlefonrm 
The Inverse Problem of Reconstructing  Reaction-Diffusion Systems.
}}
\author{Barbara Kaltenbacher\footnote{
Department of Mathematics,
Alpen-Adria-Universit\"at Klagenfurt.
barbara.kaltenbacher@aau.at.}
\and
William Rundell\footnote{
Department of Mathematics,
Texas A\&M University,
%College Station,
Texas 77843. % USA.
rundell@math.tamu.edu}
}
\maketitle

\begin{abstract}
This paper considers the inverse problem of recovering state-dependent
source terms in a reaction-diffusion system from overposed data
consisting of the values of the state variables either at a fixed finite
time (census-type data) or a time trace of their values at a fixed point
on the boundary of the spatial domain.
We show both uniqueness results and the convergence of an iteration
scheme designed to recover these sources.
This leads to a reconstructive method and we shall demonstrate  its
effectiveness by several illustrative examples.
\end{abstract}

%%%%%%%%%%%%%%%%%%%%%%%%%%%%%%%%%%%%%%%%%%%%%%%%%%%%
\section{Introduction}\label{sect:introduction}

Reaction diffusion equations have a rich history in the building
of mathematical models for physical processes.
They are descendants of nonlinear ordinary differential equations in time with
an added  spatial component making for a partial differential equation
of parabolic type.
These early models dating from the first decades of the twentieth century
include that of Fisher in considering the Verhulst
logistic equation together with a spatial diffusion or migration,
$\,u_t - ku_{xx} = f(u) = bu(1-cu)$ to take into account migration of species and that
of Kolmogorov, Petrovskii and Piskunov in similar models which are now
collectively referred to as the Fisher-KPP theory of population modeling,
see, \cite{Murray:2002}.
There is also work in combustion theory due to
Zeldovich and Frank-Kamenetskii that utilize higher order polynomials
in the state variable $u$ and where the diffusion term
acts as a balance to the chemical reactions, \cite{Grindrod:1996}.

The use of systems of reaction diffusion models followed quickly;
adding a spatial component to traditional population dynamic models
such as predator-prey and competitive species as well as the interaction of
multiple species or chemicals.
By the early 1950's it was recognized by Alan Turing that solutions to
such equations can, under the correct balance of terms, be used to simulate
natural pattern formations such as stripes and spots that may arise naturally
out of a homogeneous, uniform state \cite{Turing:1952}.
This theory, which can be called a reaction-diffusion theory of morphogenesis, 
has been a major recurrent theme across many application areas.

These models use the underlying physics to infer assumptions about the
specific form of the reaction term $f(u)$.
The few constants appearing, if not exactly known,
are easily  determined in a straightforward way by a least squares fit
to data measurements.
We envision a more complex situation where the function $f(u)$
(or multiple such functions as we will be considering systems of equations)
cannot be assumed to have a specific known form,
or to be analytic so that knowing it over a limited range gives a global extension,
 and therefore must be
treated as an undetermined coefficient problem for a nonlinear
partial differential equation.

Let $\Omega$ be a bounded, simply-connected region in $\mathbb{R}^d$
with smooth boundary $\partial\Omega$, and let
$\mathbb{L} := -\nabla\cdot(a(x)\nabla\cdot) + q(x)$.
be a uniformly elliptic operator of second order with $L^\infty$ coefficients
\begin{equation}\label{eqn:direct_prob_system}
\begin{aligned}
u_t(x,t) - \mathbb{L} u(x,t) &= f_1(u) + \phi_1(w) +  r_1(x,t,u,v), \\
v_t(x,t) - \mathbb{L} v(x,t) &= f_2(v) + \phi_2(w) +  r_2(x,t,u,v), \\
w &= w(u,v),
\qquad (x,t)\in\Omega\times(0,T) 
\end{aligned}
\end{equation}
for some fixed time $T$
and subject to the prescribed  initial and boundary conditions
\begin{equation}\label{eqn:initial_bdry_conds}
\begin{aligned}
\bdryd{u}(x,t)  +\gamma_1(x) u(x,t)&= \theta_1(x,t)\\
\bdryd{v}(x,t) +\gamma_2(x) v(x,t)&= \theta_2(x,t)\\
\quad (x,t)\in\partial\Omega\times(0,T) \\
u(x,0) = u_0(x),\qquad v(x,0) &= v_0(x) \quad x\in\Omega \\
\end{aligned}
\end{equation}
where 
\Margin{R2 (d)}
\revision{$\nu$ is the outer unit normal and}
$\gamma_i(x)$, $\theta_i(x,t)$  are in $C^\beta(\partial\Omega)$, 
where $C^\beta$ is the Schauder space with H\"older exponent $\beta$, $0<\beta<1$.

In equation~\eqref{eqn:direct_prob_system} we assume
$r_1(x,t,u,v)$ and $r_2(x,t,u,v)$ are known and the interaction variable
$w = w(u,v)$ is also known but
either  the pair $\{f_1,f_2\}$ or the pair $\{\phi_1,\phi_2\}$
is unknown and the inverse problems posed are to determine these quantities.

Thus, there are two distinct inverse problems.
The first is when we assume the interaction coupling $\phi_i(w)$ between
$u$ and $v$ is known, but both $f_1(u)$ and $f_2(v)$ have to be determined.
The second is when we assume the growth rate couplings $f_i$ for both
$u$ and $v$ are known, but the interaction terms $\phi_1(w)$ and
$\phi_2(w)$ have to be determined.

In order to perform these recoveries we must prescribe additional
data and we shall again consider two possibilities:
the values of $u(x,T)$ and $v(x,T)$ taken at a later, fixed time $T$
\begin{equation}\label{eqn:overposed_finalT}
u(x,T) = g_u(x),\qquad 
v(x,T) = g_v(x),\qquad 
x \in \Omega;\phantom{\partial} % \phantom for alignment purposes only
\end{equation}
or the time traces $u(x_0,t)$ and $v(x_0,t)$ measured at a fixed point
$x_0\in\partial\Omega$ and for all $t\in(0,T)$.
\begin{equation}\label{eqn:overposed_timetrace}
u(x_0,t) = h_u(t),\qquad 
v(x_0,t) = h_v(t),\qquad 
x_0 \in \partial\Omega.
\end{equation}
If the boundary conditions at the measurement point $x_0$ are of
Dirichlet type ($\gamma=\infty$) then instead we measure the flux
$a(x_0)\partial_\nu$ at $x_0$
in \eqref{eqn:overposed_timetrace}.
Note that final time data corresponds to census data taken a fixed time $T$
for the species involved.  The time trace data involves monitoring
the population (or of chemical concentrations) at a fixed spatial point
as a function of time.
Both of these data measurements are quite standard in applications.
In \eqref{eqn:overposed_finalT}, the observation domain can be restricted to a subdomain $\omega$ of $\Omega$; in view of the fact that the functions to be discovered are univariate, even an appropriately chosen curve in $\Omega$ could possibly suffice.

The interaction coupling $w$ of $u$ and $v$, which we assume known,
can take on several forms.
The near universal choice in ecological modeling is to take $w = uv$.
This is also common in other applications but other 
more complex possibilities are in use.

For example, the Gray-Scott model of reaction diffusion,
\cite{Pearson:1993}, takes
$w = u^2v$ and is a coupling term that often leads to pattern formation.
This coupling occurs, for example, in molecular realizations
where there is an activator ($u$) and an inhibitor ($v$).
The antagonistic effect here occurs from the relative depletion of
$v$ that is consumed during the production of $u$.
The so-called Brusselator equation (the Walgraef and Aifantis equation)
which also leads to the generation of sustained oscillations,
instead takes $w = b u - c u^2 v$ and occurs, amongst many other situations,
in the dislocation dynamics in materials subjected to cyclic loading,
\cite{WalgraefAifantis:1985}.

Other possibilities include $w = \sqrt{u^2+v^2}$ or the nonlocal
situation $w = \int_\Omega (u^2+v^2)\,dx$ as well as combinations
of all of the above.

The astute reader will already have noted that there are other very
similar combinations that might lead to physically-motivated inverse problems.
Indeed, this is the case and the analysis from Section \oldref{sec:vecBB3} also applies to 
other combinations of unknowns than those mentioned above.

In the case of a single equation using time trace data, uniqueness results
and the convergence of reconstruction algorithms were shown in
\cite{DuChateauRundell:1985, PilantRundell:1986, PilantRundell:1987} for the
recovery of the unknown term $f(u)$.
In the more recent paper \cite{KaltenbacherRundell:2019c}, the authors
used final time data as in \eqref{eqn:overposed_timetrace}
and showed uniqueness and contractibility results for a fixed point iteration
scheme which allowed effective recovery of $f(u)$. 
We also point to \cite{KaltenbacherRundell:2019b} where 
the reaction term is of the form $q(x) f(u)$ with known $f(u)$ and an unknown space dependent
term $q(x)$ that controls the intensity of the reaction.

\Margin{R1 A (1)}
\revision{
The reconstruction methods that we consider here are based on a projection of the PDE on the observation manifold, which naturally leads to a fixed point iteration for the unknown terms.
As compared to Newton's method, which is a more general approach, this projection principle requires an appropriate form of the given data; in particular, the observation manifold cannot be orthogonal to the dependence direction of the unknown function, i.e., if we aim at reconstructing an $x$ dependent coefficient then time trace observations cannot be used in such a projected approach (and will hardly give good reconstruction results with any other method either). Since here the unknown function does not depend on $x$ or $t$ this problem does not appear, though. 
Moreover, a convergence proof of regularized Newton type methods can only be carried out under certain conditions on the forward operator, such as the tangential cone condition (see, e.g., \cite{KNS08} and the references therein), which is most probably not satisfied here.
Also note that the fixed point approach used here entails a uniqueness result for the inverse problem.
}

\revision{
We remark that we have followed standard mathematical practice when
    writing differential equations by scaling the equation to set parameters
    not under consideration to unity.
    Thus in our reconstruction we have taken the volume of $\Omega$ to
be one
    and the elliptic operator $\mathbb{L}$ to be the Laplacian
    thereby making the implicit assumption that the diffusion
coefficient $a(x)$
    in the leading term $\mbox{div}(a\nabla u)$ is unity.
    In an actual physical situation this will not be the case; for example
    the value of $a$ for a molecule diffusing in a gas will be several
    orders of magnitude smaller.
}

%\smallskip

The outline of the paper is as follows.
Section \oldref{sec:convergence} provides an extension of the convergence analysis in \cite{KaltenbacherRundell:2019c} for a fixed point iteration as well as its consequences for uniqueness in two directions. Firstly, we ``vectorize'' this analysis to be applicable for a class of reaction-diffusion equations, characterized by certain conditions. Secondly, we allow for higher space dimensions by carrying out the estimates in Schauder spaces rather than (Hilbert) Sobolev spaces, thus largely avoiding dependence on the space dimension due to Sobolev embeddings. The focus is on the final data case \eqref{eqn:overposed_finalT} there.
In Section \oldref{sec:reconstructions} we show numerical tests with the discussed fixed point iteration. We provide reconstructions of the pair $(f_1,f_2)$ for known $\phi_i,r_i,w$, as well as of the pair $(\phi_1,\phi_2)$ for known $f_i,r_i,w$ in \eqref{eqn:direct_prob_system}. Here, both settings of final time data \eqref{eqn:overposed_finalT} and of time trace data \eqref{eqn:overposed_timetrace} are considered.

\section{Preamble}
\Margin{R1 B(1)}
\Margin{R2 (a)}
\revision{We start with a short revision of the results in \cite{KaltenbacherRundell:2019c}, since the present paper builds on these. As a matter of fact, in \cite{KaltenbacherRundell:2019c}, we  considered a scalar problem with only one unknown function $f$ but the setting there is slighty more general than the one described in the introduction (and most of the rest of this paper) in the sense that a subdiffusion equation is considered.
More precisely, in \cite{KaltenbacherRundell:2019c} we treat the
}
semilinear (sub)diffusion initial boundary value problem
\begin{equation}\label{eqn:ibvpBB3}
\begin{aligned}
D_t^\alpha u(x,t) + \mathbb{L}u(x,t) &= f(u(x,t))+r(x,t)
\quad (x,t)\in \Omega\times(0,T)\\
\frac{\partial u}{\partial\nu} + \gamma u &= 0\mbox{ on }\partial \Omega\times(0,T)\\
u(x,0) &= u_0(x),\qquad x\in\Omega
\end{aligned}
\end{equation}
where $D_t^\alpha$ denotes the Djrbashian-Caputo fractional time derivative of order $\alpha\in(0,1)$
which is defined by
\begin{equation}\label{eqn:DC_derivative}
D_t^\alpha u =
\frac{1}{\Gamma(1-\alpha)}\int_{0}^{t}\frac{u'(s)}{(t-s)^{\alpha}}ds
\end{equation}
In case $\alpha=1$, $D_t^\alpha$ is the usual first time derivative and \eqref{eqn:ibvpBB3} becomes a semilinear parabolic equation. 
\Margin{R1 B(1)}
\revision{
For details on fractional differentiation and subdiffusion equations, we refer to, e.g.,
\cite{Dzharbashyan:1964,Djrbashian:1970,Djrbashian:1993,
MainardiGorenflo:2000,SakamotoYamamoto:2011a,SamkoKilbasMarichev:1993}, see also the tutorial on inverse problems for anomalous diffusion processes \cite{JinRundell:2015};
For well-posedness of the forward problem \eqref{eqn:ibvpBB3} with Lipschitz continuous $f$ we refer to \cite[Section 3]{JinLiZhou18}, see also \cite[Section 2]{KaltenbacherRundell:2019c}.}

In \cite{KaltenbacherRundell:2019c}, the inverse problem of recovering the nonlinearity $f$ from final time data 
\begin{equation}\label{eqn:finaltimedataBB3}
g(x)\equiv u(x,T),\qquad x\in\Omega
\end{equation}
(assuming that all other coefficients in \eqref{eqn:ibvpBB3} are known) was approached by a fixed point scheme for the operator $\mathbb{T}$ defined by the identity
\[
(\mathbb{T}f)(g(x)) = D_t^\alpha u(x,T;f) - \mathbb{L}g(x) - r(x,T)
\quad x\in \Omega\,.
\]
where for given $f$, the function $u(x,t;f)$ denotes the solution to \eqref{eqn:ibvpBB3}.
This defines $\mathbb{T}f$ just on the range of $g$ and therefore it is crucial to assume that all values of $u$ that get inserted into $f$ belong to this range
\begin{equation}\label{eqn:rangeBB3}
J:=g(\Omega)\supseteq u(\Omega\times[0,T))\,.
\end{equation}
More precisely, it suffices to impose this condition at the solution $(f_{ex},u_{ex})$ of the inverse problem, and to project the iterates onto this range during the fixed point iteration.
Obviously, these considerations can be extended to the case of partial observations $g(x)\equiv u(x,T)$, $x\in\omega$ on a subset $\omega\subset\Omega$, under the adapted range condition $g(\omega)\supseteq u(\Omega\times[0,T))$.
Since $f$ is an univariate function, measurements on a higher dimensional domain $\Omega$ appear to be too much in fact and it can suffice to take observations along a curve $\omega$ connecting the points of minimal and maximal values of $u(x,T)$.

The following example illustrates the fact that this imposes a true constraint on the class of problems that can be expected to exhibit unique identifiability.

\begin{example}
Let $\varphi$ be an eigenfunction of $-\mathbb{L}$ with corresponding
eigenvalue $\lambda>0$.
Take $f(u)=c u$ as well as $u_0=\varphi$, $r=0$, then $u(x,t;f)=e^{-(\lambda-c) t}\varphi(x)$. If $\varphi$ stays positive (which is, e.g., the case is $\lambda$ is the smallest eigenvalue), and $c<\lambda$, then, with $\underline{\varphi}=\min_{x\in\bar{\Omega}}\varphi(x)\geq0$, $\overline{\varphi}=\max_{x\in\bar{\Omega}}\varphi(x)>0$, we get for the range of $u$ over all of $\Omega$,
\[ 
\min_{(x,t)\in\bar{\Omega}\times[0,T]} u(x,t;f) = \underline{\varphi} e^{-(\lambda-c) T} \,, \quad 
\max_{(x,t)\in\bar{\Omega}\times[0,T]} u(x,t;f) = \overline{\varphi} \,,
\]
whereas for the final time data we have 
\[ 
\min_{x\in\bar{\Omega}} u(x,T;f) = \underline{\varphi} e^{-(\lambda-c) T} \,, \quad 
\max_{x\in\bar{\Omega}} u(x,T;f) = \overline{\varphi} e^{-(\lambda-c) T} \,,
\]
and therefore the range condition will be violated.
\end{example}

In \cite[Section 3]{KaltenbacherRundell:2019c} we proved a self-mapping property of $\mathbb{T}$ on a sufficiently small ball in $W^{1,\infty}(J)$, as well as its contractivity in the parabolic case $\alpha=1$ for $T$ large enough provided $f$ is strictly monotonically decreasing, which implies exponential decay of the corresponding solution or actually its time derivative $D_tu$, as long as $r$ vanishes or has exponentially decaying time derivative, see also Lemma \oldref{lem:expdecay} below.
Such a dissipative setting is indeed crucial for proving that the Lipschitz constant of $\mathbb{T}$ decreases with increasing final time $T$, as the following counterexample shows.
\begin{example}
Again, let $\varphi(x)$ be an eigenfunction of $-\mathbb{L}$ with corresponding
eigenvalue $\lambda>0$.
Take $f^{(1)}(u) = c_1 u$, $f^{(2)}(u) = c_2 u$ as well as $u_0=\varphi$, $r=0$, and set $u^{(i)}(x,t)=u(x,t;f_i)$ which can be computed explicitely as
\[
u^{(i)}(x,t) = e^{(c_i-\lambda) t}\varphi(x) \,,\quad
u^{(i)}_t(x,t) = (c_i-\lambda)e^{(c_i-\lambda) t}\varphi(x)\,, 
\]
which yields
\[
\begin{aligned}
\mathbb{T}f^{(1)}(g(x))-\mathbb{T}f^{(2)}(g(x))&=
u^{(1)}_t(x,T) - u^{(2)}_t(x,T) \\&=
\bigl[ (c_1-\lambda)e^{c_1 T}- (c_2-\lambda)e^{c_2 T}\bigr]
e^{-\lambda T}\varphi(x)\,. 
\end{aligned}
\]
Thus, for any combination of norms $\|\cdot\|_X$, $\|\cdot\|_Z$ the contraction factor $\frac{\|\mathbb{T}f^{(1)}(g)-\mathbb{T}f^{(2)}(g)\|_Z}{|f^{(1)}-f^{(2)}|_X}$ is determined by the function $m(T)=\frac{|(c_1-\lambda)e^{(c_1-\lambda) T}- (c_2-\lambda)e^{(c_2-\lambda) T} |}{|c_1-c_2|}$.
Now take $c_1>c_2>\lambda$; then $m(0)=1$ and $m'(T)=\frac{(c_1-\lambda)^2e^{(c_1-\lambda) T}- (c_2-\lambda)^2e^{(c_2-\lambda) T} }{c_1-c_2}>0$.
This makes a contraction for finite time $T$ impossible unless 
\Margin{R1 D (2)}
$\frac{\|
\revision{\varphi}
\|_Z}{\|
\revision{f^0}
\|_X}$ 
\revision{(with $f^0(u)=u$)} 
is sufficiently small.
\end{example}

The two examples above clearly show that the two aims (a) range condition and (b) dissipativity (for contractivity) are conflicting, at least as long as we set $r=0$ as it was done here. 
Thus, in order to achieve (a) we need to drive the system by means of $r$ (or alternatively, by inhomogeneos boundary conditions). 
Luckily, nonvanishing $r$ does not impact the case (b)
since these inhomogenieities basically cancel out when taking differences
between fixed point iterates for establishing contractivity estimates.
Thus it is indeed possible to have range condition and contractivity together.
However, for this it is crucial to not only use found data, but to be able to
design the experiment in such a way that the data exhibits the desired
properties.

\smallskip

The analysis in \cite{KaltenbacherRundell:2019c} is restricted to the case of
a single scalar equation and to one space dimension, due to the fact that it
is carried out in (Hilbert) Sobolev spaces using embeddings into $W^{1,\infty}$.
The aim of the analysis here is therefore twofold. First of all, we intend
to generalize the approach from \cite{KaltenbacherRundell:2019c} towards
identification of nonlinearities in certain systems of reaction-diffusion
equations, see Section \oldref{sec:vecBB3}. Secondly, we provide an analysis
of the above fixed point scheme in Schauder spaces, which allows us to work
in higher space dimensions, see Section \oldref{sec:convSchauder}.
For the latter case we have to restrict ourselves to the parabolic setting
$\alpha=1$, 
%as the results from \cite{EidelmanKochubei04} indicate that the regularity theory from, e.g., \cite{Friedman1964} does not carry over to the subdiffusion setting in space dimension two and more, due to singularity of the fundamental solution.
as carrying over the regularity theory from, e.g., \cite{Friedman1964} to the
subdiffusion setting, based on results from, e.g., \cite{EidelmanKochubei04},
would be a major effort in itself that goes beyond the scope of this paper.
This also affects the extension of the analysis from
\cite{PilantRundell:1986,PilantRundell:1988} for time trace
\eqref{eqn:overposed_timetrace} instead of final time data
\eqref{eqn:overposed_finalT}, which is why we focus on the final time
observation case in the analysis section \oldref{sec:convergence}.

\smallskip

A generalization of the scalar diffusion setting to systems of an arbitrary number $N$ of  possibly interacting states can be achieved by replacing the unknown function $f$ in \eqref{eqn:ibvpBB3} by a component wise defined vector valued unknown nonlinearty $\vec{f}=(f_1,\ldots,f_N)$, whose action is not only defined on a single state but on a possible combination of these, via known functions $\vec{\wfunc}=(\wfunc_1,\ldots,\wfunc_N)$. Possible additional known interaction and reaction terms are encapsulated in a set of (now potentially also nonlinear) functions $\vec{r}=(r_1,\ldots,r_N)$.

We thus consider systems of reaction-(sub)diffusion equations of the form 
\begin{equation}\label{eqn:system_gen}
\begin{aligned}
D_t^\alpha u_i(x,t) + (\mathbb{L}\vec{u})_i(x,t) &= f_i(\wfunc_i(\vec{u}(x,t)))+r_i(x,t,\vec{u}(x,t))
\quad (x,t)\in \Omega\times(0,T)\,,\\
& \hspace*{7cm} i\in\{1,\ldots N\}\,,
\end{aligned}
\end{equation}
for $\vec{u}=(u_1,\ldots,u_N)$ subject to the boundary and initial conditions 
\begin{equation}\label{eqn:system_bdry}
\frac{\partial u_i}{\partial\nu} + \gamma_i u_i = 0\mbox{ on }\partial \Omega\times(0,T)\,, \ i\in\{1,\ldots N\}
\end{equation}
and
\begin{equation}\label{eqn:system_init}
\vec{u}(x,0) = \vec{u}_0(x),\qquad x\in\Omega\,.
\end{equation}
Well-posedness of this forward problem with Lipschitz continuous nonlinearities $f_i$, $r_i$ is a straightforward extension of the results from \cite[Section 3]{JinLiZhou18}, \cite[Section 2]{KaltenbacherRundell:2019c}.

Given the 
\Margin{R1 B (2)}
\revision{self-adjoint}
uniformly elliptic operator $-\mathbb{L}$, e.g. 
\begin{equation}\label{eqn:exampleL}
-\mathbb{L}=\mbox{diag}(-\nabla\cdot(\underline{A}\nabla\cdot),\ldots,-\nabla\cdot(\underline{A}\nabla\cdot))+Q(x) 
\end{equation}
with $Q:\Omega\to\mathbb{R}^{N\times N}$, $\underline{A}\in \mathbb{R}^{N\times N}$ 
\revision{symmetric}
(uniformly) positive definite, the functions 
\[
\vec{\wfunc}:I:=I_1\times\cdots\times I_N \to J:=J_1\times\cdots\times J_N, \qquad 
\vec{r}:\Omega\times(0,T)\times I\to \mathbb{R}^N\,,
\] 
and the data $\vec{u}_0$, $\gamma_i$, as well as measurements on a subset $\omega$ of the domain $\Omega$
\begin{equation}\label{eqn:finaltimedata}
g_i(x)\equiv u_i(x,T),\qquad x\in\omega\,,\quad i\in\{1,\ldots N\}\,,
\end{equation}
we wish to determine the unknown functions
\[
f_i:J_i\to\mathbb{R}\quad i\in\{1,\ldots N\}\,.
\]
This includes both cases of identifying the reaction terms $f_i$ (by setting $\wfunc_i(\xi_1,\xi_2)=\xi_i$) and of identifying the interaction terms $\phi_i$ in \eqref{eqn:direct_prob_system}. 
We will abbreviate the collection of unknown functions by $\vec{f}=(f_1,\ldots,f_N)$, noting that each individual $f_i$ might have a different domain of definition $J_i$.

Note that this setting allows for linear and nonlinear coupling among the individual states $u_i$ via the known differential operator $\mathbb{L}$ 
\revision{(often referred to as cross-diffusion)}
and the known functions $\wfunc_i$ as well as $r_i$.

Throughout the analysis we will impose the range condition 
\begin{equation}\label{eqn:range_sys}
I_i=[\underline{g}_i,\overline{g}_i]=[\min_{x\in\omega}g_i(x)\,,\max_{x\in\omega}g_i(x)]
= g_i(\omega)\supseteq u_{ex,i}(\Omega\times[0,T))\,,\quad i\in\{1,\ldots N\}\,,
\end{equation}
cf, \eqref{eqn:rangeBB3}, where we assume $\omega$ to be compact to guarantee (via Weierstrass' Theorem and continuity of $g_i$) that $I_i$ is indeed a compact interval.
Here $u_{ex}$ is the state part of a solution $(f_{ex},u_{ex})$ of the inverse problem.

\Margin{R1 B (3)}
\revision{
Part of the analysis is based on series expansions in terms of the eigenvalues and -functions $(\lambda_n,\phi_n)_{n\in\mathbb{N}}$ of the self-adjoint elliptic operator $\mathbb{L}$ 
as well as the induced Hilbert spaces 
\[
\dot{H}^\sigma(\Omega)=\{v\in L^2(\Omega)\ : \ \sum_{n=1}^\infty \lambda_n^{\sigma/2} \langle v,\phi_n\rangle \phi_n \, \in L^2(\Omega)\}
\]
where $\langle\cdot,\cdot\rangle$ denotes the $L^2$ inner product on $\Omega$, with the norm 
%\[
$\displaystyle{
\|v\|_{\dot{H}^\sigma(\Omega)} 
= \Bigl(\sum_{n=1}^\infty \lambda_n^\sigma \langle v,\phi_n\rangle^2 \Bigr)^{\!1/2}\!,
}$
%\]
that is equivalent to the $H^\sigma(\Omega)$ Sobolev norm provided the coefficients of $\mathbb{L}$ are sufficiently regular, which we assume to be the case here. 
}

\smallskip

Consider now the spatially one-dimensional setting of $\Omega\subseteq\mathbb{R}^1$ being an open interval $(0,L)$,
and make the invertibility and smoothness assumptions
\begin{equation}\label{eqn:invert}
\begin{aligned}
&\vec{g}\in H^2(\omega;\mathbb{R}^N)\,, \quad \vec{\wfunc} \in H^2(I_1\times\cdots\times I_N;\mathbb{R}^N)\\
&\left|\sum_{j=1}^N \frac{\partial \wfunc_i}{\partial \xi_j}(\vec{g}(x))\,g_j'(x)\right|\geq\beta>0 \mbox{ for all }
x\in\omega\,, \ i\in\{1,\ldots N\}\,,
\end{aligned}
\end{equation}
that by the Inverse Function Theorem imply that $\wfunc_i\circ\vec{g}:\omega\to J_i$ is bijective and its inverse is in $H^2(J_i;\Omega)$. 
Thus we can define the fixed point operator $\mathbb{T}:X\to X$, where
\begin{equation}\label{eqn:X}
X:=X_1\times\cdots\times X_N\,, \quad
X_i=\{f_i\in W^{1,\infty}(J_i)\, : \, f_i(\wfunc_i(\vec{u}_0))\in {\dot{H}^2(\Omega)}\}\,,
\end{equation}
by 
\begin{equation}\label{eqn:opT_fiti}
(\mathbb{T}\vec{f})_i(\wfunc_i(\vec{g}(x)) = D_t^\alpha u_i(x,T;\vec{f}) - (\mathbb{L}\vec{g})_i(x) - r_i(x,T,\vec{g}(x))
\quad x\in \omega\,, \quad i\in\{1,\ldots N\}\,,
\end{equation}
where for any $\vec{j}=(j_1,\ldots,j_N)$ with $j_i:J_i\to\mathbb{R}$, the function $\vec{u}(x,t)=\vec{u}(x,t;\vec{j})$ solves
\begin{equation}\label{eqn:system_gen_proj}
\begin{aligned}
D_t^\alpha u_i(x,t) + (\mathbb{L}\vec{u})_i(x,t) &= h_i(\wfunc_i(P\vec{u}(x,t)))+r_i(x,t,P\vec{u}(x,t))
\quad (x,t)\in \Omega\times(0,T)\,,\\
& \hspace*{7cm} i\in\{1,\ldots N\}\,,\\
\frac{\partial u_i}{\partial\nu} + \gamma_i u_i &= 0\mbox{ on }\partial \Omega\times(0,T)\,, \ i\in\{1,\ldots N\}\\
\vec{u}(x,0) &= \vec{u}_0(x),\qquad x\in\Omega\,,
\end{aligned}
\end{equation}
with the projection $P:\mathbb{R}^N\to I$ on the compact cuboid $I$, i.e., $P_i\xi=\max\{\underline{g}_i,\min\{\overline{g}_i,\xi\}\}$.

The range condition \eqref{eqn:range_sys} guarantees that any solution $(\vec{f}_{ex},\vec{u}_{ex})$ of the inverse problem 
\eqref{eqn:system_gen}, \eqref{eqn:system_bdry}, \eqref{eqn:system_init}, \eqref{eqn:finaltimedata} is a fixed point of $\mathbb{T}$.
\Margin{R1 B (6)}

\bigskip

Besides the resonstruction problem \eqref{eqn:system_gen}, \eqref{eqn:finaltimedata} with final time data, that will be discussed in detail in the convergence section \oldref{sec:convergence}, in the numercial reconstruction section we will also consider an analogous inverse problem of recovering $\vec{f}=(f_1,\ldots,f_N)$ in \eqref{eqn:system_gen} from time trace data
\begin{equation}\label{eqn:timetracedata}
\vec{h}(t)\equiv \vec{u}(x_0,t),\qquad t\in(0,T)\,,
\end{equation}
for some $x_0\in\partial\Omega$, cf. Pilant and Rundell \cite{PilantRundell:1986,PilantRundell:1988} for the scalar case.
Under the invertibility condition  
\begin{equation}\label{eqn:invert_h}
\begin{aligned}
&\vec{h}\in C^{1,1}(0,T;\mathbb{R}^N)\,, \quad \wfunc \in C^2(I_1\times\cdots\times I_N;\mathbb{R}^N)\\
&\left|\sum_{j=1}^N \frac{\partial \wfunc_i}{\partial \xi_j}(\vec{h}(t))\,h_j'(t)\right|\geq\beta>0 \mbox{ for all }
t\in(0,T)\,, \ i\in\{1,\ldots N\}\,,
\end{aligned}
\end{equation}
we can, analogously to \cite{PilantRundell:1986,PilantRundell:1988}, define a fixed point operator $\mathbb{T}:X:=X_1\times\cdots\times X_N\to X_1\times\cdots\times X_N$, where
$X_i=C^{0,1}(J_i)$,
by 
\begin{equation}\label{eqn:opT_titr}
(\mathbb{T}f)_i(\wfunc_i(\vec{h}(t)) = D_t^\alpha h_i(t) - (\mathbb{L}\vec{u})_i (x_0,t;\vec{f}) -r_i(x_0,t,\vec{h}(t))\quad t\in(0,T)\,, \quad i\in\{1,\ldots N\}\,.
\end{equation}
The crucial estimates of 
\Margin{R1 D (3)}
$\|(\mathbb{L}\vec{u})_i (x_0,t;\vec{f})\|_{\revision{C^{0,1}(\Omega)}}$ 
required for establishing self-mapping and contraction properties of $\mathbb{T}$ on a ball in $X$ as in \cite{PilantRundell:1986,PilantRundell:1988} could in principle like there be based on the implicit representation 
\[
\vec{u}(x,t)=\vec{\psi}(x,t)+\int_0^t \int_\Omega K(x,y,t-\tau)\Bigl(\vec{f}(\wfunc(\vec{u}(y,\tau))+\vec{r}(y,\tau,\vec{u}(y,\tau))\Bigr)\, dy
\]
of $\vec{u}$ by means of the Green's function $K$ and the solution
$\vec{\psi}(x,t)=\vec{u}(x,t;0)$ of the linear problem obtained by setting
$\vec{f}\equiv 0$, and replacing $\vec{r}(x,t,\vec{u}(x,t))$ by
$\vec{r}(x,t,0)$, together with regularity estimates on the Green's function
$K$. 
For Green's functions for systems and their regularity in the parabolic case
$\alpha=1$, see, e.g. \cite[Theorem 1, Chapter 9]{Friedman1964}.
In the subdiffusion case $\alpha<1$ the Green's function is defined by the
Fox H-functions, cf., e.g., \cite{EidelmanKochubei04}.

\section{Convergence of a fixed point scheme for final time data
}\label{sec:convergence}

We will now consider convergence of the fixed point scheme defined by the operator $\mathbb{T}$ defined by \eqref{eqn:opT_fiti} for reconstructing the reaction and interaction functions $f_i$ in \eqref{eqn:system_gen}.
To some extent we can here build on previous work for the case of one scalar
equation and a single function $f$ to be reconstructed in
\cite{KaltenbacherRundell:2019c}. However, it is also clear that the
interaction among several states can complicate the situation considerably and
lead to phenomena that would not be possible with single uncoupled equations.
Our aim in Section \oldref{sec:vecBB3} is to explore conditions for a scenario
that would allow to make some statements on self-mapping and contractivity of
$\mathbb{T}$. We are aware of the fact that this is far from capturing the
whole multitude of possibilities and interesting cases that can arise in
systems.
The theoretical results are illustrated by some examples of $2\times 2$ systems 
arising in systems biology.

Another extension made in this section is to get rid of the restriction on the spatially one dimensional setting that had to be imposed in \cite{KaltenbacherRundell:2019c} in order to enable certain Sobolev embeddings, in particular at the transition from $H^s(\Omega)$ for the space dependent function $f(g(x))$ to $W^{1,\infty}(J)$ for the univariate function $f(u)$. We do so by carrying out the analysis in Schauder spaces instead, which basically allows to use the same differentiability order for $f(g(x))$ and $f(u)$, independently of the dimension of $\Omega$. Since the required regularity results on PDE solutions are so far only available in the literature for $\alpha=1$, we restrict ourselves to this case in Section \oldref{sec:convSchauder}.

\subsection{Vectorization of the results from \cite{KaltenbacherRundell:2019c} in the spatially one-dimensional case}\label{sec:vecBB3}

Analogously to the proof of \cite[Theorem 3.1]{KaltenbacherRundell:2019c} we can establish $\mathbb{T}$ as a weakly * continuous self-mapping on a sufficiently large ball in $X$.
\begin{theorem}
Let $\alpha\in(\tfrac45,1]$, $\sigma\in(\frac32,2)$, $\theta\in(0,2-1/\alpha)$, $Q^*\geq \frac{2(2-\theta)}{2-\sigma}$, 
$\Omega\subseteq\mathbb{R}^1$ an open bounded interval, let \eqref{eqn:invert} hold 
and assume that $\kappa$ as well as 
$\bar{\rho}:=$ $\sup_{\zeta\in I} \|D_t r(\cdot,\cdot,\zeta)\|_{L^{Q^*}(0,T;L^2(\Omega))}$ are sufficiently small.\\
Then for large enough $\rho>0$ the operator $\mathbb{T}$ 
\Margin{R1 D (5)}
\revision{defined by \eqref{eqn:opT_fiti}}
is a self-mapping on the bounded, closed and convex set 
\[
B=\{h\in X\, : \, \|h_i\|_{W^{1,\infty}(J_i)}\leq \rho\,, \ \|h_i(\wfunc_i(\vec{u}_0))+r(\cdot,0,\vec{u}_0)-\mathbb{L}\vec{u}_0\|_{\dot{H}^\sigma(\Omega)}\leq \kappa\}
\]
and $\mathbb{T}$ is weakly* continuous in $X$ as defined in \eqref{eqn:X}.
Thus $\mathbb{T}$ has a fixed point in $B$.
\end{theorem}

\medskip

Moreover, in the parabolic case $\alpha=1$, contractivity of $\mathbb{T}$ for sufficiently large final time $T$ follows as in \cite[Theorems 3.2, 3.3]{KaltenbacherRundell:2019c} from the fact that for $\vec{f}^{(1)}$, $\vec{f}^{(2)}\in X$, the difference $\mathbb{T}(\vec{f}^{(1)})-\mathbb{T}(\vec{f}^{(2)})= (\vec{u}^{(1)}-\vec{u}^{(2)})_t=\vec{z}$, where 
\begin{equation*}
\begin{aligned}
D_t z_i - (\mathbb{L}\vec{z})_i &= \sum_{j=1}^N \Bigl(
\Bigl[{f^{(1)}_i}'(\wfunc_i(\vec{u}^{(1)}))\tfrac{\partial\wfunc_i}{\partial\xi_j}(\vec{u}^{(1)}) + \tfrac{\partial r_i}{\partial\xi_j}(\vec{u}^{(1)})\Bigr]) z_j
\\ &\qquad
+ \Bigl[ 
{f^{(1)}_i}'(\wfunc_i(\vec{u}^{(1)}))\tfrac{\partial\wfunc_i}{\partial\xi_j}(\vec{u}^{(1)}) + \tfrac{\partial r_i}{\partial\xi_j}(\vec{u}^{(1)}) 
\\ &\qquad\qquad
-{f^{(2)}_i}'(\wfunc_i(\vec{u}^{(2)}))\tfrac{\partial\wfunc_i}{\partial\xi_j}(\vec{u}^{(2)}) + \tfrac{\partial r_i}{\partial\xi_j}(\vec{u}^{(2)})
\Bigr] D_t u^{(2)}_j
\Bigr)\mbox{ in }\Omega\times(0,T)\\
\frac{\partial z_i}{\partial\nu} + \gamma_i z_i &= 0\mbox{ on }\partial \Omega\times(0,T)\\
z_i(x,0)&=f^{(1)}_i(\wfunc_i(\vec{u}_0(x)))-f^{(2)}_i(\wfunc_i(\vec{u}_0(x)))\quad x\in\Omega
\end{aligned}
\end{equation*}
$i\in\{1,\ldots N\}$. The factor $D_t\vec{u}^{(2)}$ appearing in the right hand side of this PDE decays exponentially under certain conditions.
\begin{lemma}\label{lem:expdecay}
Let $\mathbb{L}$ be of the form \eqref{eqn:exampleL} where for some $c_Q>-\underline{\lambda}$, with $\underline{\lambda}>0$ the smallest eigenvalue of $-\nabla\cdot(\underline{A}\nabla\cdot)$ with homogeneous impedance boundary conditions, 
\Margin{R1 D (6)}
\begin{equation}\label{eqn:cQ}
Q(x)-M(x,t,\xi)-c_Q\revision{\mbox{id}} \mbox{ is nonnegative definite} \mbox{ for all }x\in\Omega,\ t>0, \ \xi\in I\,,
\end{equation}
(note that here the definition of nonnegative definiteness of a matrix here does not necessarily include its symmetry)
where 
\begin{equation}\label{eqn:M}
M_{i,j}(x,t,\xi):= f_i'(\wfunc_i(\xi)) \tfrac{\partial\wfunc_i}{\partial\xi_j}(\xi) + \tfrac{\partial r_i}{\partial\xi_j}(\xi)
\end{equation}
and assume that there exist constants $C_r,c_r>0$ such that 
\[
|D_t\vec{r}(x,t,\xi)|\leq C_r e^{-c_r t} \qquad x\in\Omega\,, \ t>0\,,\xi\in I\,,
\]
and that $D_tu_i(0)= (\mathbb{L}\vec{u}_0)+f_i(\wfunc_i(\vec{u}_0))+r_i(0,\vec{u}_0)\in L^\infty(\Omega)$.\\
Then there exist $C_2,c_2>0$ (more precisely, $c_2\in(0,\min\{c_r,\frac{\underline{\lambda}+c_Q}{2}\})$) such that the solution $\vec{u}$ of \eqref{eqn:system_gen}, \eqref{eqn:system_bdry}, \eqref{eqn:system_init} satisfies the exponential decay estimate 
\[
|D_t\vec{u}(x,t)|\leq C_2 e^{-c_2 t} \qquad x\in\Omega\,, \ t>0\,.
\]
\end{lemma}
\begin{proof}
%the maximum principle for parabolic system needed for the comparison argument in \cite[Section 3.3]{KaltenbacherRundell:2019c} can be found in \cite[Theorem 1]{Weinberger1975}.
Exponential decay of $D_t\vec{u}$ follows from the maximum principle applied
to the scalar function 
$v(x,t)=\frac12|D_t\vec{u}(x,t)|^2 = \frac12\sum_{i=1}^N D_t u_i(x,t)^2$.
In the case of \eqref{eqn:exampleL} this satisfies,
with $c_1=c_Q+\underline{\lambda}$, the equations
\begin{equation}\label{eqn:v}
\begin{aligned}
D_tv-\triangle v+(c_Q-\varepsilon) v
&= - \nabla D_t\vec{u}^T\underline{A}\nabla D_t\vec{u} - D_t\vec{u}^T (Q(x)-M(x,t,P\vec{u}(x,t))-c_Q 
\revision{\mbox{id}}) D_t\vec{u}\\
&\quad  -\frac{\varepsilon}{2} |D_t\vec{u}|^2+D_t\vec{u}^T D_t\vec{r}
\qquad\qquad\qquad
\mbox{ in }\ \Omega\times(0,T), \\ 
&\leq 
- D_t\vec{u}^T (Q(x)-M(x,t,P\vec{u}(x,t))-c_Q)\revision{\mbox{id}}) D_t\vec{u} +\tfrac{1}{2\epsilon}|D_t\vec{r}|^2 \\
\frac{\partial v}{\partial \nu}+\gamma v &= 0\qquad\qquad\qquad\qquad\qquad\qquad\qquad
 \mbox{ on }\ \partial\Omega\times(0,T)\,,\\
v(x,0)&=\frac12|D_t\vec{u}(0)|^2=\frac12\sum_{i=1}^N |-(\mathbb{L}\vec{u}_0)+f_i(\wfunc_i(\vec{u}_0))+r_i(0,\vec{u}_0)|^2\qquad x\in\Omega,
\end{aligned}
\end{equation}
where we have used Young's inequality.
The right hand side in \eqref{eqn:v} can be bounded by the exponentially decaying function $\frac{1}{2\varepsilon}|D_t\vec{r}(t)|^2$, provided $Q(x)-M(x,t,P\vec{u}(x,t))-c_Q\revision{\mbox{id}}$ is positive semidefinite for all $(x,t)\in\Omega\times(0,T)$, which is guaranteed by \eqref{eqn:cQ}.
By the maximum principle, this implies $v\leq \bar{v}$ pointwise in $\Omega\times(0,T)$, for the solution $\bar{v}$ of $\bar{v}_t-\triangle \bar{v}+(c_Q-\varepsilon) \bar{v}=\frac{1}{2\varepsilon}|D_t\vec{r}(t)|^2$ with initial data $\bar{v}(x,0)=\frac12|D_t\vec{u}(0)|^2$ and therefore (cf. \cite{Pazy92}) 
\[
v(x,t)\leq \bar{v}(x,t)\leq\tfrac{1}{4\varepsilon} e^{-\min\{2c_r,\underline{\lambda}+c_Q-\varepsilon\}t} |D_t\vec{u}(0)|^2
\] 
for all $(x,t)\in\Omega\times(0,T)$.
\end{proof}

\begin{remark}
A possible way to satisfy condition \eqref{eqn:cQ} is by sufficiently strong diffusion in the elliptic operator \eqref{eqn:exampleL} 
\[
\lambda_{min}(Q(x))+\underline{\lambda}>\sup_{t\in(0,T),\xi\in I,\zeta\in\mathbb{R}^N\setminus\{0\}}\tfrac{1}{|\zeta|^2}\zeta^T \,(\vec{f}'(\vec{\wfunc}(\zeta)\tfrac{d \vec{\wfunc}}{d \xi}(\xi)+\tfrac{d \vec{r}}{d \xi}(x,t,\xi))\,\zeta
\]
for all $x\in\Omega$, $t\in(0,T)$.

We mention in passing that for the proof of Lemma \oldref{lem:expdecay} it was essential to work with homogeneous boundary conditions.
\end{remark}

\medskip

Along the lines of the proof of \cite[Theorem 3.2]{KaltenbacherRundell:2019c} we therefore obtain the following contractivity result.
\begin{theorem} \label{th:contr1}
Let the assumptions of Lemma \oldref{lem:expdecay} 
\Margin{R1 D (7)}
\revision{hold and assume that}
\begin{equation}\label{eqn:c1hat}
\begin{aligned}
&\| {f^{(1)}_i}'(\wfunc_i(\vec{u}^{(1)}))\tfrac{\partial\wfunc_i}{\partial\xi_j}(\vec{u}^{(1)})+\tfrac{\partial r_i}{\partial\xi_j}(\vec{u}^{(1)})+\hat{Q}_{ij}\|_{L^\infty(0,T;L^\infty(\Omega))}\leq \hat{c}_1\,,
\end{aligned}
\end{equation}
for some positive semidefinite matrix $\hat{Q}(x)$ and $\hat{c}_1$ sufficiently small such that 
$4\hat{c}_1<\hat{\lambda}_1^2$, where $\hat{\lambda}_1$ is the smallest eigenvalue of the operator $-\hat{\mathbb{L}}:=-\mathbb{L}+\hat{Q}$, i.e., $4\hat{c}_1<\hat{\lambda}_1^2$. 
Moreover, assume that  $f^{(1)}_i,f^{(2)}_i\in W^{2,\infty}(J_i)$ and 
\begin{equation}\label{eqn:ass_norm}
\sum_{i=1}^N\|(f^{(1)}_i-f^{(2)}_i)(\wfunc_i(\vec{u}^{(2)}(t))\|_{H^1(\Omega)}^2\leq 
\bar{C}_1(g) \sum_{i=1}^N\|(f^{(1)}_i-f^{(2)}_i)(\wfunc_i(\vec{g}))\|_{H^1(\omega)} \,.
\end{equation}
Then there exists a constant $C>0$ depending only on $\|\vec{f}^{(1)}{}''\|_{L^\infty(I)}$, $\|(\mathbb{L}\vec{u}_0)+f_i(\wfunc_i(\vec{u}_0))+r_i(0,\vec{u}_0)\|_{L^\infty(\Omega)}$ and the constant $\bar{C}_1(g)$ in \eqref{eqn:ass_norm}, such that 
\[ 
\|(\mathbb{T}(f_1)-\mathbb{T}(f_2))(g)\|_{\dH{1}}\leq 
C e^{-(\hat{\lambda}_1-4\hat{c}_1/\hat{\lambda}_1)T/2} \|(f_1-f_2)(g)\|_{\dH{1}}\,.
\]
with $\hat{c}_1$ as in \eqref{eqn:c1hat}.
\end{theorem}

Note that this result does not rely on any Sobolev embeddings and therefore remains valid for higher space dimensions, provided we can make sense of the condition \eqref{eqn:ass_norm}, which in one space dimension easily follows even in the general form
\begin{equation*}%\label{eqn:ass_norm_h}
\sum_{i=1}^N\|h_i(\wfunc_i(\vec{u}^{(2)}(t))\|_{H^1(\Omega)}^2\leq 
\bar{C}_1(g) \sum_{i=1}^N\|h_i(\wfunc_i(\vec{g}))\|_{H^1(\omega)} \quad
\forall \vec{h}\in W^{2,\infty}(J)
\end{equation*}
from uniform strict monotonicity of the functions $g_i$ and $x\mapsto u^{(2)}_i(x,t)$ for all $t\geq0$ as well as the range condition \eqref{eqn:range_sys}.

\medskip
\def\umax{u_{\max}}
\def\vmax{v_{\max}}
\subsubsection{Some Examples}
We now provide some examples of typical systems to which the analysis above applies.

\begin{example}\label{example_vec}
\begin{equation}\label{eqn:ex_vec}
D_t^\alpha u_i -\triangle u_i = f_i(u_i) \quad i\in\{1,\ldots N\}\,,
\end{equation}
where $\wfunc_i(\xi_1,\ldots,\xi_N)=\xi_i$, and \eqref{eqn:invert} is satisfied for strictly monotone and $H^2$ smooth $g_i$, with $J_i=I_i=g_i(\omega)$. 
The Jacobian of $\vec{w}$ is just the identity matrix and 
$M(x,t,\xi)= \mbox{diag}(f_1'(\xi_1) \ldots f_N'(\xi_N))$, so \eqref{eqn:cQ} is satisfied for monotonically decreasing functions $f_j$, or if the $f_j'$ have arbitrary sign but their positive values on $I_j$ are dominated by $\lambda_{\min}(Q(x))+\underline{\lambda}$. This clearly extends to 
\[
D_t^\alpha u_i -\triangle u_i = f_i(u_i) +r_i(\vec{u}) \quad i\in\{1,\ldots N\}\,,
\]
with sufficiently small interaction terms $r_i$.
A particular case of interest here is the so-called competing species interaction where 
$r_i(\vec{u})=-u_i\sum_{j\not=i}\beta_{ij}u_j$ with nonnegative coefficients $\beta_{ij}$, cf., e.g., \cite{Juengel10} and see Example \oldref{example_competing_species} below for the case $N=2$.
\end{example}

We now turn to some examples of $2\times 2$ systems, where condition \eqref{eqn:cQ} can be verified by applying the simple criterion
\begin{equation}\label{eqn:ABCD}
\left(\begin{array}{cc} A&B\\C&D\end{array}\right) \mbox{ nonnegative definite }\ \Longleftrightarrow \ A\geq0\,, \ D\geq0\,, \ 4AD\geq (B+C)^2\,.
\end{equation}

\begin{example}\label{example_competing_species}
\begin{equation}\label{eqn:ex_competing_species}
\begin{aligned}
D_t^\alpha u - \triangle u &= f_1(u)-\beta u\cdot v \\
D_t^\alpha v - \triangle v &= f_2(v)-\beta u\cdot v  
\end{aligned}
\end{equation}
with $\beta>0$.
Here $M(x,t,u,v)= \left(\begin{array}{cc}f_1'(u)-\beta v&-\beta u\\-\beta v& f_2'(v)-\beta u\end{array}\right)$, and we expect both $u$ and $v$ to be nonnegative, thus set 
\begin{equation}\label{eqn:I1I2_ex}
I_1=[0,\umax ]\,, \quad I_2=[0,\vmax ]
\end{equation}
for some upper bounds $\umax $, $\vmax $ $>0$.
Nonnegativity of $-M$ via \eqref{eqn:ABCD} is equivalent to
\[
-f_1'(u)+\beta v\geq0\,, \quad -f_2'(v)+\beta u\geq0 \,, \quad
4(-f_1'(u)+\beta v)(-f_2'(v)+\beta u)\geq \beta^2(u+v)^2\,.
\] 
A sufficient condition for this to hold is $-f_1'(u)\geq \frac{\beta}{2} u$, $-f_1'(v)\geq \frac{\beta}{2} v$ for all $u\in I_1$, $v\in I_2$. On the other hand, setting $v=0$ implies 
$-f_1'(u)\geq0$ and $-f_1'(u)\geq \frac{\beta^2u^2}{4(-f_2'(0)+\beta u)}$; likewise we get
$-f_2'(v)\geq0$ and $-f_2'(v)\geq \frac{\beta^2v^2}{4(-f_1'(0)+\beta v)}$.
This implies 
\[
\begin{aligned}
&\beta\leq\min_{u\in(0,\umax ]}\tfrac{2}{u}(-f_1'(u)+\sqrt{f_1'(u)^2-f_2'(0)f_1'(u)}) \\
&\beta\leq\min_{v\in(0,\vmax ]}\tfrac{2}{v}(-f_2'(v)+\sqrt{f_2'(v)^2-f_1'(0)f_2'(v)})\,,
\end{aligned}
\]
i.e., a restriction on the size of the interaction
$\;\beta u.v$, in terms of the reactions $\,\{f_1,f_2\}$.
\end{example}

%\begin{example}\label{example_competing_species_interaction}
%Identification of the interaction rather than the reaction terms in a (generalized) competing species model leads to
%\begin{equation}\label{eqn:ex_competing_species_interaction}
%\begin{aligned}
%D_t^\alpha u - \triangle u &= f_1(u\cdot v)+r_1(u) \\
%D_t^\alpha v - \triangle v &= f_2(u\cdot v)+r_2(v)  
%\end{aligned}
%\end{equation}
%With $M(x,t,u,v)= \left(\begin{array}{cc}f_1'(u\cdot v)v+r_1'(u)&f_1'(u\cdot v)u\\
%f_2'(u\cdot v)v&f_2'(u\cdot v)u+r_2'(v)
%\end{array}\right)$ and \eqref{eqn:I1I2_ex}+++
%\end{example}

\begin{example}\label{example_fu}
\begin{equation}\label{eqn:ex_fu}
\begin{aligned}
D_t^\alpha u - \triangle u &= f(u)+r_1(u,v) \\
D_t^\alpha v - \triangle v &= r_2(u,v) \\ 
\end{aligned}
\end{equation}
see, e.g., \cite[equations (1), (2), page 61]{Kuttler:17}, where typically $r_2(u,v)=-r_1(u,v)$ or at least the individual corresponding terms have opposite sign. More precisely, they are often of the form 
\begin{equation}\label{eqn:mak}
\begin{aligned}
r_1(u,v)&= k_{1+}u^{\nu_{11}}v^{\nu_{12}}-k_{1-}u^{\mu_{11}}v^{\mu_{12}}\\
r_2(u,v)&= k_{2+}u^{\nu_{21}}v^{\nu_{22}}-k_{2-}u^{\mu_{21}}v^{\mu_{22}}
\end{aligned}
\end{equation}
with nonnegative constants $k_{i\pm}$, $\nu_{ij}$, $\mu_{ij}$, which would be the characteristic for, e.g., mass-action kinetics.
In view of the practically relevant setting of nonnegative states, we will again consider the nonlinear functions on 
$I_1=[0,\umax ]\,, \quad I_2=[0,\vmax ]$.

In this example,
$M(x,t,u,v)= \left(\begin{array}{cc}f'(u)+r_{1,u}(u,v)&r_{1,v}(u,v)\\r_{2,u}(u,v)&r_{2,v}(u,v)\end{array}\right)$, so that nonnegativity of $-M$ via \eqref{eqn:ABCD} is equivalent to
\begin{eqnarray}
\label{eqn:ex_fu_I}
r_{2,v}(u,v)&\leq& 0 \quad \forall \ u\in [0,\umax ]\,, \ v\in [0,\vmax ] \\
\label{eqn:ex_fu_II}
f'(u)&\leq& -\sup_{v\in [0,\vmax ]} \left(r_{1,u}+\frac{(r_{1,v}+r_{2,u})^2(u,v)}{-4r_{2,v}(u,v)}\right) \quad \forall \ v\in [0,\vmax ]\,.
\end{eqnarray}
In case of the particular form \eqref{eqn:mak}, condition \eqref{eqn:ex_fu_I} is equivalent (by considering the asymptotics as $u,v\to0$) to 
\begin{equation}\label{eqn:ex_fu_makI}
%\mu_{22}k_{2-}\geq0 \ \mbox{ and } \  
\nu_{21}\geq\mu_{21} \ \mbox{ and } \ \nu_{22}\geq\mu_{22} \ \mbox{ and } \ 
\nu_{22}k_{2+}\umax ^{\nu_{21}-\mu_{21}} \vmax ^{\nu_{22}-\mu_{22}} \leq  \mu_{22}k_{2-}\,.
\end{equation}

The requirements of $f'\in L^\infty([0,\umax ])$, and $r_{1,u}\in L^\infty([0,\umax ]\times[0,\vmax ])$ lead to further restrictions on the exponents $\nu_{ij}$, $\mu_{ij}$ to avoid singularities at vanishing $u,v$.
\end{example}

\begin{example}\label{example_fv}
\begin{equation}\label{eqn:ex_fv}
\begin{aligned}
D_t^\alpha u - \triangle u &= f(v)+r_1(u,v) \\
D_t^\alpha v - \triangle v &= r_2(u,v) \\ 
\end{aligned}
\end{equation}
see, e.g., \cite[equations (18), (19), page 76]{Kuttler:17}.

Here, $M(x,t,u,v)= \left(\begin{array}{cc}r_{1,u}(u,v)&f'(v)+r_{1,v}(u,v)\\r_{2,u}(u,v)&r_{2,v}(u,v)\end{array}\right)$, so that nonnegativity of $-M$ via \eqref{eqn:ABCD} is equivalent to
\begin{eqnarray}
\label{eqn:ex_fv_I}
r_{1,u}(u,v)&\leq& 0\,, \ r_{2,v}(u,v)\leq 0 \quad \forall \ u\in [0,\umax ]\,, \ v\in [0,\vmax ] \\
\nonumber
|f'(v)+r_{1,v}(u,v)+r_{2,u}(u,v)|&\leq& 2\sqrt{r_{1,u}(u,v)\,r_{2,v}(u,v)} \quad \forall \ u\in [0,\umax ]\,, \ v\in [0,\vmax ]\,.\\
&&\label{eqn:ex_fv_II}
\end{eqnarray}
In the setting of \eqref{eqn:mak}, condition \eqref{eqn:ex_fv_I} is equivalent to
\begin{equation}\label{eqn:ex_fv_makI}
\begin{aligned}
%\mu_{22}k_{2-}>0 \ \mbox{ and } \  
\nu_{21}\geq\mu_{21} \ \mbox{ and } \ \nu_{22}\geq\mu_{22} \ \mbox{ and } \ 
\nu_{22}k_{2+}\umax ^{\nu_{21}-\mu_{21}} \vmax ^{\nu_{22}-\mu_{22}} \leq  \mu_{22}k_{2-}\\
%\mu_{11}k_{1-}>0 \ \mbox{ and } \  
\nu_{11}\geq\mu_{11} \ \mbox{ and } \ \nu_{12}\geq\mu_{12} \ \mbox{ and } \ 
\nu_{11}k_{1+}\umax ^{\nu_{11}-\mu_{11}} \vmax ^{\nu_{12}-\mu_{12}} \leq  \mu_{11}k_{1-}\,.
\end{aligned}
\end{equation}
\end{example}
Note that $f'\in L^\infty([0,\vmax ])$ together with \eqref{eqn:ex_fv_II} does not impose additional constraints on the exponents in \eqref{eqn:mak}.

\begin{example}\label{example_f1f2u}
\begin{equation}\label{eqn:ex_f1f2u}
\begin{aligned}
D_t^\alpha u - \triangle u &= f_1(u)+r_1(u,v) \\
D_t^\alpha v - \triangle v &= f_2(u)+r_2(u,v) \\ 
\end{aligned}
\end{equation}
see, e.g., \cite[page 89]{Kuttler:17}.

With $M(x,t,u,v)= \left(\begin{array}{cc}f_1'(u)+r_{1,u}(u,v)&r_{1,v}(u,v)\\f_2'(u)+r_{2,u}(u,v)&r_{2,v}(u,v)\end{array}\right)$, nonnegativity of $-M$ via \eqref{eqn:ABCD} is equivalent to
\begin{eqnarray}
\label{eqn:ex_f1f2_I}
&&f_1'(u)\leq -\sup_{v\in [0,\vmax ]} r_{1,u} \quad \forall \ v\in [0,\vmax ]
\,, \quad r_{2,v}(u,v)\leq 0 \nonumber\\ 
&& \hspace*{5cm}\forall \ u\in [0,\umax ]\,, \ v\in [0,\vmax ] \\
\label{eqn:ex_f1f2_II}
&&|f_2'(u)+r_{1,v}(u,v)+r_{2,u}(u,v)|\leq 2\sqrt{(f_1'(u)+r_{1,u}(u,v))\,r_{2,v}(u,v)} \nonumber\\ 
&& \hspace*{5cm} \forall \ u\in [0,\umax ]\,, \  v\in [0,\vmax ]
\end{eqnarray}
In the setting of \eqref{eqn:mak}, 
the right hand part of \eqref{eqn:ex_f1f2_I} coincides with \eqref{eqn:ex_fu_makI} in Example \eqref{eqn:ex_fu}.
Again, $f_i'\in L^\infty([0,\vmax ])$ together with \eqref{eqn:ex_f1f2_II} does not impose additional conditions on the exponents in \eqref{eqn:mak}.
\end{example}

\def\ug{\underline{g}}
\def\og{\overline{g}}
\subsection{Analysis in Schauder spaces in the parabolic case and higher space dimensions}\label{sec:convSchauder}

In the parabolic case $\alpha=1$, the availability of regularity results in Schauder spaces $C^{k,\beta}(\Omega\times(0,T))$, cf., e.g, \cite{Friedman1964}, allows to work in higher space dimensions $\Omega\subseteq \mathbb{R}^d$, $d>1$.
Thus these results are new as compared to those in \cite{KaltenbacherRundell:2019c} even in case of a single PDE. 
Note that the Schauder space setting has already been used in \cite{PilantRundell:1986,PilantRundell:1988} for the same nonlinearity identification problems in case of time trace (instead of final time) observations. 

For simplicity of exposition we here consider the scalar case
(abbreviating $D_tu$ by $u_t$ since there will be no further subscripts here) of recovering $f:J\to\mathbb{R}$ in 
\begin{equation}\label{eqn:PDEparabolicscalar}
\begin{aligned}
u_t(x,t)-\mathbb{L} u(x,t) &=f(u(x,t))+r(x,t) \qquad (x,t)\in\Omega\times(0,T) \\
\frac{\partial u}{\partial\nu} + \gamma u &= 0\qquad \mbox{ on }\partial \Omega\times(0,T) \\
u(x,0)&=u_0(x) \quad x\in\Omega
\end{aligned}
\end{equation}
with observations
\begin{equation}\label{eqn:obs_scalar}
g(x)=u(x,T)\quad x\in \omega\,.
\end{equation}
Note that analogously to Section \oldref{sec:vecBB3}, this can be extended to the system setting \eqref{eqn:system_gen}--\eqref{eqn:finaltimedata}, provided a higher dimensional (with respect to space) version of the condition \eqref{eqn:invert} holds and guarantees invertibility as well as smoothness of the mapping $\wfunc_i\circ\vec{g}:\omega\to J_i$. 

\subsubsection{Self-mapping fixed-point operator on spaces of Lipschitz continuous functions}

In order to work in function spaces over a fixed interval $J:= g(\omega)=[\ug,\og]$, like in \cite{KaltenbacherRundell:2019c} we project the values of $u$ onto $J$, which can as well be written by means of a superposition operator 
\[
(Pu)(x,t)= \max\{\ug ,\min\{\og ,u(x,t)\}\}=\Phi(u(x,t))
\]
with $\Phi=\mbox{id}$ on $[\ug ,\og ]$, $\Phi\equiv\ug $ on $(-\infty,\ug ]$, $\Phi\equiv
\og $ on $[\og ,\infty)$.
Note that $\Phi$ is contained in $W^{1,\infty}(\mathbb{R})=C^{0,1}(\mathbb{R})$ which will be sufficient for the proof of $\mathbb{T}$ being a self-mapping on $X=C^{0,1}(J)$.
Later on, when using higher order Schauder spaces to show contractivity of $\mathbb{T}$, the lack of additional smoothness of $\Phi$ will remain  an issue.

Moreover, we assume the range condition
\begin{equation}\label{eqn:rangecondSchauder}
J=[\ug,\og]=g(\omega)\supseteq u_{ex}(\Omega\times[0,T))
\end{equation}
to hold for any exact solution  $(f_{ex},u_{ex})$ of the inverse problem \eqref{eqn:PDEparabolicscalar}, \eqref{eqn:obs_scalar}.

Thus we define the fixed point operator $\mathbb{T}:X\to X$ by 
\begin{equation}\label{eqn:T_parabolicscalar0}
\mathbb{T} f(g(x)) = u_t(x,T;f) - (\mathbb{L}g)(x)-r(x,T)\quad x\in \omega\,,
\end{equation}
where for some $j\in C^{0,1}(J)$, the function $u(x,t;j)$ solves
\begin{equation}\label{eqn:PDEparabolicscalar_P}
\begin{aligned}
u_t-\mathbb{L} u &=j(\Phi(u))+r \qquad \mbox{ in }\Omega\times(0,T) \\
\frac{\partial u}{\partial\nu} + \gamma u &= 0\qquad \mbox{ on }\partial \Omega\times(0,T) \\
u(x,0)&=u_0(x) \quad x\in\Omega\,.
\end{aligned}
\end{equation}

Equation \eqref{eqn:T_parabolicscalar0} indeed uniquely determines $f^+:=\mathbb{T}f$ on $J$ if, e.g, $\omega$ is curve in $\Omega$ along which $g$ is strictly monotone.
Otherwise, the transition from the multivariate function $f^+(g):\omega\to\mathbb{R}$, defined on the $d$-dimensional domain $\omega$, to the real function $f:J\to\mathbb{R}$, defined on an interval $J\subseteq\mathbb{R}$, can be carried out by metric projection, cf. \cite[Lemma 3.1, Remark 3.1]{KaltenbacherRundell:2019c}.
We therefore redefine $\mathbb{T}:X\to X$ by 
\begin{equation}\label{eqn:T_parabolicscalar}
\mathbb{T} f = \mathbb{P}_g y \mbox{ where } y(x) = u_t(x,T;f) - (\mathbb{L}g)(x)-r(x,T)\quad x\in \omega
\end{equation}
and assume that the mapping $\mathbb{P}_g: C^{0,1}(\omega)\to C^{0,1}(J)$ satisfies the compatibility condition
\begin{equation}\label{eqn:compatPg}
f_{ex}=\mathbb{P}_g\left(u_{ex,t}(T) - \mathbb{L}g-r(T)\right)
\end{equation}
for an exact solution $(f_{ex},u_{ex})$ of the inverse problem \eqref{eqn:PDEparabolicscalar}, \eqref{eqn:obs_scalar},
and is continuous as a mapping from $C^{0,1}(\omega)$ to $C^{0,1}(J)$, i.e., for all $y\in C^{0,1}(\omega)$, the bound  
\begin{equation}\label{eqn:PgLip}
\|\mathbb{P}_g y\|_{C^{0,1}(J)}\leq C(g) \|y\|_{C^{0,1}(\omega)}
\end{equation}
holds. 
If, e.g, $\omega$ is a curve with a regular parametrization $\omega=\{x(\tau)\, : \tau\in[0,1]\}$ such that $\nabla g(x(\tau))\cdot \dot{x}(\tau)\geq \frac{1}{C(g)}|\dot{x}(\tau)|$ for all $\tau\in[0,1]$, then for $\mathbb{P}_g$ simply defined by $(\mathbb{P}_g w)(g(x))=w(x)$, $x\in\omega$, the estimate \eqref{eqn:PgLip} follows from 
\[
\begin{aligned}
\sup_{\xi\not=\eta\in J} \frac{|(\mathbb{P}_g w)(\xi)-(\mathbb{P}_g w)(\eta)|}{|\xi-\eta|}
&=\sup_{x\not=y\in\omega} \frac{|w(x)-w(y)|}{|g(x)-g(y)|}
=\sup_{\sigma\not=\tau\in[0,1]} \frac{|w(x(\sigma))-w(x(\tau))|}{|g(x(\sigma))-g(x(\tau))|}\\
&=\sup_{\sigma\not=\tau\in[0,1]} \frac{|w(x(\sigma))-w(x(\tau))|}{|x(\sigma)-x(\tau)|}
\frac{|x(\sigma)-x(\tau)|}{|g(x(\sigma))-g(x(\tau))|}
\end{aligned}\]
where 
\[
|g(x(\sigma))-g(x(\tau))|
=\left|\int_\tau^\sigma \nabla g(x(\rho))\cdot \dot{x}(\rho)\, d\rho\right|
\geq \frac{1}{C(g)} \int_\tau^\sigma |\dot{x}(\rho)|\, d\rho
\geq \frac{1}{C(g)} |x(\sigma)-x(\tau)|\,.
\]
\Margin{R1 B (6)}
Assumptions \eqref{eqn:rangecondSchauder} and \eqref{eqn:compatPg} imply that the $f$ part of any such solution $(f_{ex},u_{ex})$ is a fixed point of $\mathbb{T}$.

\medskip

We will now prove that $\mathbb{T}$ is a self-mapping on $X=C^{0,1}(J)$.
To do so, we use 
\Margin{R1 B (4)} 
\Margin{R1 B (5)} 
\revision{
the identity 
\begin{equation}\label{eqn:Tf-fex}
\begin{aligned}
\mathbb{T}f(g)-f_{ex}(g) &= (u_t(T;f)-\mathbb{L}g -r(T))-(u_{ex,t}(T)-\mathbb{L} u_{ex}(T)-r(T))
\\ &= u_t(T;f)-u_{ex,t}(T)
\,,
\end{aligned}
\end{equation}
to which we apply $\mathbb{P}_g$ to obtain, using \eqref{eqn:PgLip},
}
\[
\|\mathbb{T}f - f_{ex}\|_{C^{0,1}(J)}\leq  C(g) \|\hat{u}_t(T)\|_{C^{0,1}(\omega)}\,.
\]
where $\hat{u}(x,t)=u(x,t;f)-u_{ex}(x,t)$ with $u(x,t)=u(x,t;f)$ solving \eqref{eqn:PDEparabolicscalar_P} with $j=f$,
and therefore $\hat{u}=u-u_{ex}$ solves
\[
\begin{aligned}
\hat{u}_t-\mathbb{L} \hat{u} + \bar{c} \hat{u} & = 
f(\Phi(u))-f_{ex}(u) \qquad \mbox{ in }\Omega\times(0,T) \\
\frac{\partial \hat{u}}{\partial\nu} + \gamma \hat{u}&= 0\quad \mbox{ on }\partial \Omega\times(0,T) \\
\hat{u}(x,0)&=0 \quad  x\in\Omega\,.
\end{aligned}
\]
where $\bar{c}=-\frac{f_{ex}(u)-f_{ex}(u_{ex})}{u-u_{ex}} \in L^\infty((0,T)\times\Omega)$ since $f_{ex}\in C^{0,1}(J)$.
From \cite[Theorem 6, page 65]{Friedman1964} we obtain that 
\[
\begin{aligned}
\|\hat{u}_t(T)\|_{C^{0,1}(\omega)}
&\leq \|\hat{u}_t\|_{C^{0,1}((0,T)\times\Omega)}
\leq K \|f(\Phi(u))-f_{ex}(u)\|_{C^{0,1}((0,T)\times\Omega)}\\
%&\leq K (\|(f-f_{ex})(\Phi(u))\|_{C^{0,1}((0,T)\times\Omega)} + \|(f_{ex}(\Phi(u))-f_{ex}(u)\|_{C^{0,1}((0,T)\times\Omega)})
&\leq K \|f\circ\Phi-f_{ex}\|_{C^{0,1}(\mathbb{R})}
(1+\|u\|_{C^{0,1}((0,T)\times\Omega)})
\end{aligned}
\]
where $u=u_{ex}+\hat{u}$ and 
\[
\|\hat{u}\|_{C^{0,1}((0,T)\times\Omega)}\leq K \|f(\Phi(u))-f_{ex}(u)\|_{C((0,T)\times\Omega)}\leq \|f\circ\Phi-f_{ex}\|_{C(\mathbb{R})}\,.
\]
Thus, provided $f_{ex}\in C^{0,1}(J)$ and $u_{ex}\in C^{0,1}((0,T)\times\Omega)$, we can conclude from $f\in C^{0,1}(J)$ that also $\mathbb{T}f\in C^{0,1}(J)$.

\begin{theorem}
Under assumptions \eqref{eqn:rangecondSchauder}, \eqref{eqn:compatPg}, \eqref{eqn:PgLip}, the operator $\mathbb{T}$ defined by \eqref{eqn:T_parabolicscalar} is a self-mapping on $X=C^{0,1}(J)$.
\end{theorem}

\subsubsection{Contractivity in higher order Schauder spaces}

In order to prove contractivity of $\mathbb{T}$, we need to move on to higher order Schauder spaces $X\subseteq C^{2,\beta}(\mathbb{R})$ for some fixed $\beta\in(0,1]$. 
The PDE estimates we will use for this purpose will rely on the Schauder space regularity theory for parabolic equations from \cite{Friedman1964}. Note that in view of the counterexample \cite[Problem 4.9]{GilbargTrudinger}, the H\"older exponent $\beta$ needs to be strictly positive.

A first attempt to circumvent the lack of higher smoothness of the projection operator would be to replace $\Phi$ in \eqref{eqn:PDEparabolicscalar_P} by a smoothed version $\Phi^\epsilon$. However, in order to prove convergence of the resulting approximation as $\epsilon\to0$, we would need uniform boundedness of $\Phi_\epsilon'$ in $C^{0,\beta}(\mathbb{R})$, which -- as can be readily checked -- is not possible, though.
Note that this lack of smoothness of $\Phi$ was not an issue in \cite{KaltenbacherRundell:2019c}, where we worked in Sobolev spaces, since the superposition operator induced by $\Phi'$ is Lipschitz as 
\Margin{R1 D (8)}
\revision{an} 
operator from $L^q$ to $L^p$ for any $1\leq p<q<\infty$, cf. \cite{HintermullerItoKunisch02}. 

We therefore achieve confinement to the observable range $J=g(\omega)$ in a
different manner, namely by definition of the spaces in which we work as 
\begin{equation}\label{eqn:XSchauder}
X = \{j\in C^{2,\beta}(\mathbb{R})\, : \, j'=0 \mbox{ on }\mathbb{R}\setminus J\}.
\end{equation}
Thus $f_{ex}\!\in\! X$ implies the necessary conditions ${f_{ex}}'(\ug )={f_{ex}}'(\og )={f_{ex}}''(\ug )={f_{ex}}''(\og )=0$ on the exact solution $f_{ex}$.
To circumvent these conditions, one might,
in place of just enforcing constant values outside $J$,
impose $C^{2,\beta}$ smooth extrapolation by defining
$X=\{j\in C^{2,\beta}(\mathbb{R})\, : \, j\vert_{[\og,\infty)}\in \Pi([\og,\infty))\,, \  j\vert_{(-\infty,\ug]} \in \Pi((-\infty,\ug])\}$
for some low dimensional spaces $\Pi([\og,\infty))$, $\Pi((-\infty,\ug])$
of polynomials or rational functions. However, this would not allow to estimate
the global $C^{2,\beta}(\mathbb{R})$ norm of $f$ by its corresponding norm on
the observable part $J$.
We therefore remain with the space $X$ defined by \eqref{eqn:XSchauder}.
Note that for all $j\in X$ the identity 
\[
\|j\|_{C^{k,\beta}(J)}=\|j\|_{C^{k,\beta}(\mathbb{R})} \quad k\in\{0,1,2\} 
\]
holds.

In order to map from final time data defined on $\omega$ to functions defined on $J$ (or actually on all of $\mathbb{R}$ here), we will again use an operator $\mathbb{P}_g: C^{2,\beta}(\omega)\to X$ and in place of \eqref{eqn:PgLip}, assume bounds with respect to the stronger spaces to hold, i.e., for all $y\in C^{2,\beta}(\omega)$,
\begin{equation}\label{eqn:Pg}
\|\mathbb{P}_g y\|_{C^{2,\beta}(\mathbb{R})}=\|\mathbb{P}_g y\|_{C^{2,\beta}(J)}\leq C(g) \|y\|_{C^{2,\beta}(\omega)}\,.
\end{equation}

Summarizing, we define $\mathbb{T}:X\to X$ by \eqref{eqn:T_parabolicscalar} with $u(x,t,j)$ denoting the solution of
\begin{equation}\label{eqn:PDEparabolicscalar_noP}
\begin{aligned}
u_t-\mathbb{L} u &=j(u)+r \qquad \mbox{ in }\Omega\times(0,T) \\
\frac{\partial u}{\partial\nu} + \gamma u &= 0\qquad \mbox{ on }\partial \Omega\times(0,T) \\
u(x,0)&=u_0(x) \quad x\in\Omega\,,
\end{aligned}
\end{equation}
differently from the previous subsection, where we used \eqref{eqn:PDEparabolicscalar_P}.

\medskip

\paragraph{Contractivity for small $u_{ex,t}$}:\\
We will now prove that $\mathbb{T}$ is a self-mapping on a $C^{2,\beta}$ ball with appropriately chosen radius $\rho>0$, and that it contracts the error $\|f-f_{ex}\|_{C^{2,\beta}(J)}$. 
To this end, we use the definition \eqref{eqn:T_parabolicscalar} together with \eqref{eqn:Pg} and the fact that the function defined by $z(x,t):= u_t(x,t;f)-u_{ex,t}(x,t)$ on $\Theta:=\Omega\times(0,T)$ satisfies the parabolic initial boundary value problem
\begin{equation}\label{eqn:z}
\begin{aligned}
z_t-\mathbb{L} z - {f_{ex}}'(u_{ex}) z & = \bigl\{f'(u)-{f_{ex}}'(u_{ex})\bigr\}  u_t\\ 
& = \bigl\{\int_0^1 {f_{ex}}''(u_{ex}+\theta\hat{u})\, d\theta \, \hat{u}
+(f-{f_{ex}})'(u)\bigr\} u_t 
%\\&\hspace*{7cm} 
\quad\mbox{ in }\Omega\times(0,T) \\
\frac{\partial z}{\partial\nu} + \gamma z&= 0\quad \mbox{ on }\partial \Omega\times(0,T) \\
z(x,0)&=(f-f_{ex})(u_0(x)) \quad  x\in\Omega\,.
\end{aligned}
\end{equation}
Here $u=u(\cdot,\cdot;f)$ solves \eqref{eqn:PDEparabolicscalar_noP} with $j=f$,
and $\hat{u}=u-u_{ex}$ solves
\[
\begin{aligned}
\hat{u}_t-\mathbb{L} \hat{u} - \int_0^1{f_{ex}}'(u_{ex}+\theta\hat{u})\, d\theta \, \hat{u} & = 
(f-f_{ex})(u) \qquad \mbox{ in }\Omega\times(0,T) \\
\frac{\partial \hat{u}}{\partial\nu} + \gamma \hat{u}&= 0\quad \mbox{ on }\partial \Omega\times(0,T) \\
\hat{u}(x,0)&=0 \quad  x\in\Omega\,.
\end{aligned}
\]

Using \eqref{eqn:Pg} as well as the identity \eqref{eqn:Tf-fex}
we obtain that
\[
\|\mathbb{T}f - f_{ex}\|_{C^{2,\beta}(J)}\leq  C(g) \|z(T)\|_{C^{2,\beta}(\omega)}\,.
\]
\revision{
To estimate $\|z(T)\|_{C^{2,\beta}(\omega)}$, we apply \cite[Theorem 6, page 65]{Friedman1964},
together with the fact that for the space-time cylinder $\Theta=(0,T)\times\Omega$, the space $C^{0,\beta}(\Theta)$ with the norm 
\[
\|w\|_{C^{0,\beta}(\Theta)}=\|w\|_{C(\Theta)}+\sup_{(x,t)\not=(y,s)\in \Theta}\frac{|w(x,t)-w(y,s)|}{\sqrt{|x-y|^2+|t-s|}^\beta}
\]
is a Banach algebra i.e., $\|v\cdot w\|_{C^{0,\beta}(\Theta)}\leq \|v\|_{C^{0,\beta}(\Theta)} \|w\|_{C^{0,\beta}(\Theta)}$ for all $v,w\in C^{0,\beta}(\Theta)$, and, moreover
\[
\|j(w)\|_{C^{0,\beta}(M)} \leq \|j\|_{C^{0,\beta}(\mathbb{R})}(1+\|w\|_{C^{0,1}(M)})
\]
holds for all $j\in C^{0,\beta}(\mathbb{R})$, $w\in C^{0,\beta}(M)$ and either $M=\Theta=(0,T)\times\Omega$ or $M=\Omega$.
\\
Assuming $f_{ex}'(u_{ex})\in C^\beta(\Theta)$, we get, for $\Theta_t=(0,t)\times\Omega$, 
\begin{equation*}%\label{eqn:estzr0}
\begin{aligned}
&\|z\|_{C([0,t];C^{2,\beta}(\Omega))}\\
&\qquad\leq
%\|z_t\|_{C^{0,\beta}(\Theta_t)}+
\sum_{|m|\leq2}\|D^m_x z\|_{C^{0,\beta}(\Theta_t)}\\
&\qquad\leq K\Bigl(\|\int_0^1 {f_{ex}}''(u_{ex}+\theta\hat{u})\, d\theta \, \hat{u} + (f-{f_{ex}})'(u)\|_{C^{0,\beta}(\Theta)} \|u_t\|_{C^{0,\beta}(\Theta_t)} + \|(f-f_{ex})(u_0(x))\|_C^{2,\beta}(\Omega)\\
&\qquad\leq K\Bigl(\|{f_{ex}}''\|_{C^{0,\beta}(J)} (1+\|u_{ex}\|_{C^{0,1}(\Theta_t)}+\|\hat{u}\|_{C^{0,1}(\Theta_t)})\, \|\hat{u}\|_{C^{0,\beta}(\Theta_t)})\\
&\qquad\qquad+ \|(f-{f_{ex}})'\|_{C^{0,\beta}(J)} (1+\|u_{ex}\|_{C^{0,1}(\Theta_t)}+\|\hat{u}\|_{C^{0,1}(\Theta_t)})\Bigr)\,UZ
%(\|u_{ex,t}\|_{C^{0,\beta}(\Theta_t)}+\|z\|_{C^{0,\beta}(\Theta_t)})
\\ &\qquad\qquad+ \|f-{f_{ex}}\|_{C^{2,\beta}(J)} (1+\|u_0\|_{C^{0,1}(\Omega)}) 
(d^2\|u_0\|_{C^{1,\beta}(\Omega)}^2 + \|u_0\|_{C^{2,\beta}(\Omega)}+1)\Bigr)
\end{aligned}
\end{equation*}
where
\begin{equation*}
\begin{aligned}
UZ &=(\|u_{ex,t}\|_{C^{0,\beta}(\Theta_t)}+\|z\|_{C^{0,\beta}(\Theta_t)})\,.\\
%\|z\|_{C^{0,\beta}(\Theta_t)}
%&\leq \|z^r\|_{C^{0,\beta}(\Theta_t)}+\|z^0\|_{C^{0,\beta}(\Theta_t)}
%\leq \sum_{|m|\leq2}\|D^m_x z^r\|_{C^{0,\beta}(\Theta_t)} 
%+\|z^0\|_{C^{0,\beta}(\Theta_t)}\,.\\
\end{aligned}
\end{equation*}
But the latter estimate does not yield contractivity, since for this we would need an estimate by a small multiple of $\|f-{f_{ex}}\|_{C^{2,\beta}(J)}$.
}

Thus we split $z=z^r+z^0$ into a part $z^r$ satisfying the inhomogeneous PDE 
\[
z^r_t-\mathbb{L} z^r - {f_{ex}}'(u_{ex}) z^r = 
\bigl\{\int_0^1 {f_{ex}}''(u_{ex}+\theta\hat{u})\, d\theta \, \hat{u}
+(f-{f_{ex}})'(u)\bigr\} u_t \mbox{ in }\Omega\times(0,T) 
\]
with homogeneous initial conditions $z^r(x,0)=0$
and a part $z^0$ satisfying the homogeneous PDE $z^0_t-\mathbb{L} z^0 - {f_{ex}}'(u_{ex}) z^0 = 0$ with inhomogeneous initial conditions 
\[
z^0 (x,0)=(f-f_{ex})(u_0(x)) \quad  x\in\Omega\,.
\]
To estimate the $z^r$, we apply \cite[Theorem 6, page 65]{Friedman1964}
which yields
\begin{equation}\label{eqn:estzr0}
\begin{aligned}
\|z^r\|_{C([0,t];C^{2,\beta}(\Omega))}
%&\leq
%\|z^r_t\|_{C^{0,\beta}(\Theta_t)}+
%\sum_{|m|\leq2}\|D^m_x z^r\|_{C^{0,\beta}(\Theta_t)}\\
%&\leq K\|\int_0^1 {f_{ex}}''(u_{ex}+\theta\hat{u})\, d\theta \, \hat{u} + (f-{f_{ex}})'(u)\|_{C^{0,\beta}(\Theta)} \|u_t\|_{C^{0,\beta}(\Theta_t)} \\
&\leq K\Bigl(\|{f_{ex}}''\|_{C^{0,\beta}(J)} (1+\|u_{ex}\|_{C^{0,1}(\Theta_t)}+\|\hat{u}\|_{C^{0,1}(\Theta_t)})\, \|\hat{u}\|_{C^{0,\beta}(\Theta_t)})\\
&\qquad+ \|(f-{f_{ex}})'\|_{C^{0,\beta}(J)} (1+\|u_{ex}\|_{C^{0,1}(\Theta_t)}+\|\hat{u}\|_{C^{0,1}(\Theta_t)})\Bigr)\,UZ\,,
%(\|u_{ex,t}\|_{C^{0,\beta}(\Theta_t)}+\|z\|_{C^{0,\beta}(\Theta_t)})\,.
\end{aligned}
\end{equation}
where $UZ$ is as above and
\begin{equation*}
\begin{aligned}
%UZ &=(\|u_{ex,t}\|_{C^{0,\beta}(\Theta_t)}+\|z\|_{C^{0,\beta}(\Theta_t)})\,.\\
\|z\|_{C^{0,\beta}(\Theta_t)}
&\leq \|z^r\|_{C^{0,\beta}(\Theta_t)}+\|z^0\|_{C^{0,\beta}(\Theta_t)}
\leq \sum_{|m|\leq2}\|D^m_x z^r\|_{C^{0,\beta}(\Theta_t)} 
+\|z^0\|_{C^{0,\beta}(\Theta_t)}\,.
\end{aligned}
\end{equation*}
Here
\[
\begin{aligned}
\|\hat{u}\|_{C^{0,1}(\Theta_t)}&\leq 
%\|\hat{u}_t\|_{C^{0,\beta_0}(\Theta_t)}+
\sum_{|m|\leq2}\|D^m_x \hat{u}\|_{C^{0,\beta_0}(\Theta_t)}\\
&\leq K \|(f-f_{ex})(u)\|_{C^{0,\beta_0}(\Theta_t)} \leq
K \|f-{f_{ex}}\|_{C^{0,\beta_0}(J)} (1+\|u_{ex}\|_{C^{0,1}(\Theta_t)}+\|\hat{u}\|_{C^{0,1}(\Theta_t)})
\end{aligned}
\]
for any $\beta_0>0$, thus, for $\|f-{f_{ex}}\|_{C^{0,\beta_0}(J)}\leq\rho_0<\frac{1}{K}$,
\[
\|\hat{u}\|_{C^{0,\beta}(\Theta_t)}
\leq \|\hat{u}\|_{C^{0,1}(\Theta_t)}
\leq \frac{1}{1-\rho_0 K}
\|f-f_{ex}\|_{C^{0,\beta_0}(J)} 
(1+\|u_{ex}\|_{C^{0,1}(\Theta_t)})\,.
\]
So for $\|(f-{f_{ex}})'\|_{C^{0,\beta}(J)}\leq\rho_1$ with $\rho_1<\frac{1}{K}$, we end up with an estimate for $z^r$ of the form
\begin{equation}\label{eqn:estzr}
\begin{aligned}
&\|z^r\|_{C([0,t];C^{2,\beta}(\Omega))}\\
&\leq C(K,\rho_0,\rho_1) (\|f-{f_{ex}}\|_{C^{1,\beta}(J)}
(1+\|u_{ex}\|_{C^{0,1}(\Theta_t)}) (\|u_{ex,t}\|_{C^{0,\beta}(\Theta_t)}+\|z^0\|_{C^{0,\beta}(\Theta_t)})\,.
\end{aligned}
\end{equation}

\revision{Now we attempt to achieve a contractive (with respect to $f-f_{ex}$) estimate of} $z^0$. To this end, we use its series expansion in terms of the eigenvalues and -functions $(\lambda_n,\phi_n)_{n\in\mathbb{N}}$ of the 
\revision{self-adjoint}
elliptic operator \revision{$\mathbb{L}+q\cdot$, where $q\in L^2(\Omega)$, $q\approx- {f_{ex}}'(u_{ex})$, e.g., $q=-\frac{1}{T}\int_0^T{f_{ex}}'(u_{ex}(t))\, dt$.}
This yields
\[
\begin{aligned}
z^0(x,t)&=\sum_{n=1}^\infty \Bigl(e^{-\lambda_n t} \langle (f-f_{ex})(u_0),\phi_n\rangle 
\revision{+\int_0^t e^{-\lambda_n (t-s)} \langle ({f_{ex}}'(u_{ex}(s))+q)\, z^0(s),\phi_n\rangle\, ds}\Bigr)\phi_n(x)\\
&\revision{=: z^{0,1}(x,t) + z^{0,2}(x,t)}\,.
\end{aligned}
\]
Via Sobolev's embedding with $\sigma > d/2+2+\beta$, 
\revision{we get, for any $\tau\leq t$,} 
\begin{equation*}%\label{eqn:estz0C2beta}
\begin{aligned}
\|\revision{z^{0,1}(\tau)}\|_{C^{2,\beta}(\Omega)}&\leq C_{\dot{H}^\sigma,C^{2,\beta}}^\Omega
\left(\sum_{n=1}^\infty \lambda_n^\sigma e^{-2\lambda_n \tau} \langle (f-f_{ex})(u_0),\phi_n\rangle^2\right)^{1/2}\\
&\leq  C_{\dot{H}^\sigma,C^{2,\beta}}^\Omega 
\sup_{\lambda\geq\lambda_1}\lambda^{\sigma/2-1} e^{-\lambda \tau}
\|(f-f_{ex})(u_0)\|_{\dot{H}^2(\Omega)}\\
&\leq  C_{\dot{H}^\sigma,C^{2,\beta}}^\Omega 
\Psi(\tau;\sigma,\lambda_1)
C_\mathbb{L}\|f-f_{ex}\|_{C^2(J)} \|u_0\|_{\dot{H}^2(\Omega)}
\end{aligned}
\end{equation*}
with $C_{\mathbb{L}}$ such that 
$\|\mathbb{L} j(v)\|_{L^2(\Omega)}\leq C_{\mathbb{L}} \|j\|_{C^2(\mathbb{R})} \|v\|_{\dot{H}^2(\Omega)}$ for all $j\in C^2(\mathbb{R})$, $v\in \dot{H}^2(\Omega)$ and
\begin{equation}\label{eqn:Psi}
\Psi(\tau;\sigma,\lambda_1) = \begin{cases} 
(\sigma/2-1)^{\sigma/2-1} e^{1-\sigma/2}\, \tau^{1-\sigma/2}&\mbox{ for }t\leq \frac{\sigma-2}{2\lambda_1}\\
\lambda_1^{\sigma/2-1} e^{-\lambda_1 \tau}&\mbox{ for }t\geq \frac{\sigma-2}{2\lambda_1}
\,.\end{cases}
\end{equation}
\revision{
and 
\[
\begin{aligned}
&\|z^{0,2}(\tau)\|_{C^{2,\beta}(\Omega)}\leq C_{\dot{H}^\sigma,C^{2,\beta}}^\Omega
\left(\int_0^\tau\sum_{n=1}^\infty \lambda_n^\sigma e^{-2\lambda_n (\tau-s)} 
\langle ({f_{ex}}'(u_{ex}(s))+q)z^0(s),\phi_n\rangle^2\, ds\right)^{1/2}\\
&\leq  C_{\dot{H}^\sigma,C^{2,\beta}}^\Omega C(\Omega)
\left(\int_0^\tau\Psi(\tau-s;\sigma,\lambda_1)\|{f_{ex}}'(u_{ex}(s))+q\|_{H^2(\Omega)}^2\, ds \right)^{1/2}
\|z^0\|_{C([0,t];C^2(\Omega))}\,,
\end{aligned}
\]
hence, assuming 
\begin{equation}\label{eqn:fexprimeq}
C_{\dot{H}^\sigma,C^{2,\beta}}^\Omega C(\Omega)\int_0^\tau \Psi(\tau-s;\sigma,\lambda_1)^2\|{f_{ex}}'(u_{ex}(s))+q\|_{H^2(\Omega)}^2\,ds \, \leq c<1
\end{equation}
for some constant $c$ and all $\tau\in[0,T]$, we get the estimate
\[
\|z^{0,2}(\tau)\|_{C^{2,\beta}(\Omega)}\leq c(\|z^{0,1}\|_{C([0,\tau];C^2(\Omega))}+\|z^{0,2}\|_{C([0,\tau];C^2(\Omega))}),
\]
hence, taking the supremum over $\tau\in[0,t]$ on both sides 
\begin{equation}\label{eqn:estz0C2beta_02}
\|z^{0,2}\|_{C([0,t];C^{2,\beta}(\Omega))}\leq \tfrac{1}{1-c}\,\|z^{0,1}\|_{C([0,\tau];C^2(\Omega))}
\end{equation}
i.e.,
\begin{equation}\label{eqn:estz0C2beta}
\|z^0\|_{C([0,t];C^{2,\beta}(\Omega)}\leq\tfrac{2-c}{1-c}\,
C_{\dot{H}^\sigma,C^{2,\beta}}^\Omega 
\sup_{\tau\in[0,t]}\Psi(\tau;\sigma,\lambda_1)
C_\mathbb{L}\|f-f_{ex}\|_{C^2(J)} \|u_0\|_{\dot{H}^2(\Omega)}\,.
\end{equation}
}
Likewise, for $\sigma_0 > d/2+\beta$, we have
\begin{equation}\label{eqn:estz0C0beta}
\|z^0\|_{\revision{C([0,t];}C^{0,\beta}(\Omega))}\leq 
C_{\dot{H}^{\sigma_0},C^{2,\beta}}^\Omega 
\revision{\sup_{\tau\in[0,t]}\Psi(\tau;\sigma_0,\lambda_1)}
C_\mathbb{L}\|f-f_{ex}\|_{C^2(J)} \|u_0\|_{\dot{H}^2(\Omega)}\,,
\end{equation}
where we can avoid the potential singularity of $\Psi(t;\sigma_0,\lambda_1)$ at $t=0$ in \eqref{eqn:Psi} by assuming
\begin{equation}\label{eqn:sigma0}
2\geq \sigma_0 > d/2+\beta
\end{equation}
which still admits the three-dimensional space case.
(Note that this avoidance is not possible for $\Psi(t;\sigma,\lambda_1)$ due to the requirement $\sigma > d/2+2+\beta$ made above.)
\\
\revision{However, a problem occurs in \eqref{eqn:estz0C2beta}, since due to the singularity at $\tau=0$ of $\Psi(\tau;\sigma,\lambda_1)$, the factor $\sup_{\tau\in[0,t]}\Psi(\tau;\sigma,\lambda_1)$ is not finite. Recall that the time dependence of ${f_{ex}}'(u_{ex})$ led to the convolution part $z^{02}$ that forced us to take the supremum over $\tau\in[0,t]$ to arrive at \eqref{eqn:estz0C2beta_02}.
Thus, the part $z^0$ of $z$ corresponding to the initial condition cannot be controlled in this setting and we need to remove it by assuming the initial condition $(f-f_{ex})(u_0)$ to vanish.
This is in line with existing results on decay of solutions to autonomous equations, see, e.g., \cite{Lunardi:1995}, that require time periodicity of the coefficient -- in our case $f_{ex}'(u_{ex})$, which is not available here, though. 
}

Doing so,
we end up with an estimate of the form 
\[
\begin{aligned}
&\|\mathbb{T}f-f_{ex}\|_{C^{2,\beta}(J)} \\
&\leq C(K,\rho_0,\rho_1,\mathbb{L},\Omega) (1+\|u_{ex}\|_{C^{0,1}(\Theta)}) 
\|u_{ex,t}\|_{C^{0,\beta}(\Theta)}
\, \|f-{f_{ex}}\|_{C^2(J)}
\,.
\end{aligned}
\]
provided $\|f-{f_{ex}}\|_{C^{2,\beta}(J)}\leq\rho$ small enough.
This yields self-mapping and contractivity on a ball of radius $\rho$ around $f_{ex}$ 
\revision{in 
\begin{equation}\label{eqn:XSchauder0}
X = \{j\in C^{2,\beta}(\mathbb{R})\, : \, j'=0 \mbox{ on }\mathbb{R}\setminus J
\revision{\,, \ (j-f_{ex})(u_0)=0}
\}.
\end{equation}
}
if  
$\|u_{ex,t}\|_{C^{0,\beta}(\Theta)}$ 
%and $\|u_0\|_{\dot{H}^2(\Omega)}$ 
is sufficiently small.
%If $T\geq \frac{\sigma-2}{2\lambda_1}$, cf. \eqref{eqn:Psi}, then $\|u_0\|_{\dot{H}^2(\Omega)}$ comes with a factor that decays exponentially with $T$, i.e., the estimate takes the form
%\[
%\begin{aligned}
%&\|\mathbb{T}f-f_{ex}\|_{C^{2,\beta}(J)} \\
%&\leq C(K,\rho,\mathbb{L},\Omega) (1+\|u_{ex}\|_{C^{0,1}(\Theta)}) 
%(\|u_{ex,t}\|_{C^{0,\beta}(\Theta)}+e^{-\lambda_1 T}\|u_0\|_{\dot{H}^2(\Omega)})
%\, \|f-{f_{ex}}\|_{C^{2,\beta}(J)}
%\,,
%\end{aligned}
%\]
%thus contractivity can be achieved also with large initial data, provided $T$ is large enough.
\begin{theorem} \label{th:selfmapping_Schauder}
Let $2>d/2+\beta$, and \eqref{eqn:Pg} hold, and assume that ${f_{ex}}\in X$ defined as in \eqref{eqn:XSchauder0}\\
Then there exist $\rho$, $\kappa$ $>0$, such that for $\|u_{ex,t}\|_{C^{0,\beta}(\Theta)}
\leq\kappa$,
 the operator $\mathbb{T}$ is a self-mapping on $B_\rho^X(f_{ex})=\{j\in C^{2,\beta}(\mathbb{R})\, : \, j'=0 \mbox{ on }\mathbb{R}\setminus J \,, \ 
\revision{(j-f_{ex})(u_0)=0\,, \ } \|j-f_{ex}\|_{C^{2,\beta}(J)}\leq \rho\}$.

Moreover, the contraction estimates
\[
\|\mathbb{T}f_0-f_{ex}\|_{C^{2,\beta}(J)} \leq q \|f_0-f_{ex}\|_{C^2(J)}\,, \qquad
\|\mathbb{T}^nf_0-f_{ex}\|_{C^{2,\beta}(J)} \leq q^n \|f_0-f_{ex}\|_{C^2(J)}
\]
hold for some $q\in(0,1)$ and any $f_0\in B_\rho^X(f_{ex})$.
\end{theorem}
\revision{
Note that assuming $u_0=0$ we get 
$X \supset \{j\in C^{2,\beta}(\mathbb{R})\, : \, j'=0 \mbox{ on }\mathbb{R}\setminus J
\,, \ (j-f_{ex})(0)=0\}$.
}
\begin{corollary}\label{cor:uniquenessSchauder}
Under the assumptions of Theorem \oldref{th:selfmapping_Schauder}, there exists at most one $C^2$ solution $f_{ex}$ of the inverse problem within any ball of radius $\rho$ in $X$.
\end{corollary}
Note that this result is valid in space dimensions $d\in[1,4-2\beta)$, so in particular also for $d=3$, as long as $\beta<\frac12$.

\medskip

\paragraph{Contractivity for monotonically decreasing $f$}:\\
To avoid the smallness assumption on $\|u_{ex,t}\|_{C(\Theta)}$ in Theorem \oldref{th:selfmapping_Schauder}, we can make use of exponential decay of $u_t$ in case of monotonically decreasing $f$ and exponentially decaying $r_t$.

To this end, we assume existence of a nonnegative, only space dependent potential $\check{q}$, that is sufficiently close to the space and time dependent function ${f_{ex}}'(u_{ex})$
\begin{equation}\label{eqn:c1check}
\| {f_{ex}}'(u_{ex})+\check{q}\|_{L^\infty(\Omega\times(0,T))}\leq \check{c}_1\,,
\end{equation}
and rewrite \eqref{eqn:z} as
\begin{equation}\label{eqn:z1}
\begin{aligned}
z_t-\mathbb{L} z + \check{q} z & = ({f_{ex}}'(u_{ex})+\check{q}) z + y\, u_t =: r_I+r_{II}
 \quad \mbox{ in }\Omega\times(0,T) \\
\frac{\partial z}{\partial\nu} + \gamma z&= 0\quad \mbox{ on }\partial \Omega\times(0,T) \\
z(x,0)&=(f-f_{ex})(u_0) \quad  x\in\Omega\,,
\end{aligned}
\end{equation}
where
\begin{equation}\label{eqn:y}
y=\int_0^1 {f_{ex}}''(u_{ex}+\theta\hat{u})\, d\theta \, \hat{u}\, 
+(f-{f_{ex}})'(u)\,.
\end{equation}

Moreover, we take advantage of the fact that without the initial data term 
%i.e., if $(f-f_{ex})(u_0)=0$, 
we gain regularity of $f-f_{ex}$ by iterating, i.e., by \eqref{eqn:estzr}, we have 
\begin{equation}\label{eqn:contrest1}
\|\mathbb{T}^2f-f_{ex}\|_{C^{2,\beta}(J)} 
\leq 
C(g)C(K,\rho_0,\rho_1) (\|\mathbb{T}f-{f_{ex}}\|_{C^{1,\beta}(J)}
(1+\|u_{ex}\|_{C^{0,1}(\Theta)}) \|u_{ex,t}\|_{C^{0,\beta}(\Theta)}
\end{equation}
which puts us into the favorable position of only having to estimate the $C^{1,\beta}$ norm instead of the $C^{2,\beta}$ norm of $(\mathbb{T} f-f_{ex})(g)=z(T)$.
Note that the choice \eqref{eqn:sigma0} that enabled us to work in higher space dimensions required us to use the $H^2$ norm of $(f-f_{ex})(u_0)$, which we estimated by the $C^2$ norm of $f-f_{ex}$. To avoid the term $\|f-f_{ex}\|_{C^2(J)}$ in the right hand side of \eqref{eqn:estz0C2beta}, we therefore impose $(f-f_{ex})(u_0)=0$, which can, e.g., be achieved by assuming $f(0)=f_{ex}(0)$ and $u_0\equiv0$.

To achieve contractivity for large enough final time $T$, we will additionally assume that $\mathbb{L}$ is of the form 
\begin{equation}\label{eqn:LA}
\mathbb{L} = \nabla\cdot (\underline{A}\nabla\cdot) \mbox{ with }\underline{A}\in \mathbb{R}^{d\times d} \mbox{ positive definite},
\end{equation}
that $r_t$ is exponentially decaying, and that $f$ is monotonically decreasing $f'\leq0$, which altogether implies that $u_t$ appearing in the right hand side of \eqref{eqn:z} decays exponentially
\begin{equation}\label{eqn:expdecayu2}
\|u_t(t)\|_{C(\Omega)}\leq C_2 e^{-c_2t}
\end{equation}
as in \cite[Section 3.3]{KaltenbacherRundell:2019c} and Lemma \oldref{lem:expdecay} above.  

In order to make use of dissipativity in estimating the $C^{1,\beta}(\Omega)$ norm of $z(T)=(\mathbb{T}f-\mathbb{T}f_{ex})(g)$, we take a small deviation via Sobolev spaces. Namely, we use continuity of the embeddings $W^{\theta,p}(0,t)\to C(0,t)$ and $W^{2-2\theta,p}(\Omega)\to C^{1,\beta}(\Omega)$ for $\theta\in(0,1)$, $\theta>\frac{1}{p}$, $1-2\theta>\frac{d}{p}+\beta$ (which can always be achieved by choosing $p\in [1,\infty)$ sufficiently large) and apply interpolation, as well as maximal $L^p$ regularity of the operator $A=-\mathbb{L}+\check{q}$ cf. \cite[Proposition 8]{Latushkin2006} to the equation \eqref{eqn:z1}.
The latter together with causality of the equation yields, for any $\mu\in(0,\check{\lambda}_1)$, where $\check{\lambda}_1$ the smallest eigenvalue of $A$, and the functions defined by 
\[
z_\mu(x,t)=e^{\mu t} z(x,t), \quad r_{I,\mu}(x,t)=e^{\mu t} r_I(x,t), \quad r_{II,\mu}(x,t)=e^{\mu t} r_{II}(x,t)
\]
that 
\begin{equation}\label{eqn:CA}
\begin{aligned}
\|z_\mu\|_{W^{1,p}(0,t;L^p(\Omega))}+\|z_\mu\|_{L^p(0,t;W^{2,p}(\Omega))} 
\leq C^A_{\mu,p} \|r_{I,\mu}+r_{II,\mu}\|_{L^p(\Omega\times(0,t))} 
\end{aligned}
\end{equation}
with a constant $C^A_{\mu,p}$ independent of $t$.
This together with interpolation and the fact that the norm of the embedding $W^{\theta,p}(0,t)\to C(0,t)$ is independent of $t$ (by Morrey's inequality)
yields
\begin{equation}\label{eqn:expmuz}
\begin{aligned}
\|e^{\mu t} z(t)\|_{C^{1,\beta}(\Omega)}
&\leq C_{W^{\theta,p},C}^{\mathbb{R^+}} C_{W^{2-2\theta,p},C^{1,\beta}}^\Omega \|z_\mu\|_{W^{\theta,p}(0,t;W^{2-2\theta,p}(\Omega))}\\
&\leq C_{W^{\theta,p},C}^{\mathbb{R^+}} C_{W^{2-2\theta,p},C^1}^\Omega C^A_{\mu,p} 
\|r_{I,\mu}+r_{II,\mu}\|_{L^p(\Omega\times(0,t))} \,.
\end{aligned}
\end{equation}
For the two terms on the right hand side of \eqref{eqn:z1} we can estimate
\[
\|r_{I,\mu}\|_{L^p(\Omega\times(0,T))}
\leq \check{c}_1|\Omega|^{1/p}\left(\int_0^t \|e^{\mu s} z(s)\|_{C(\Omega)}^p\, ds\right)^{1/p}\, ;
\]
To estimate $r_{II,\mu}$, we first of all consider the multiplier $y$, cf. \eqref{eqn:y}, 
\[
\begin{aligned}
\|y\|_{C(\Omega\times(0,t))}
&\leq \|f-f_{ex}\|_{C^1(J))}+\|{f_{ex}}''\|_{L^\infty(J)} \|\hat{u}\|_{C(\Omega\times(0,t))}\\
&\leq \|f-f_{ex}\|_{C^1(J))}+\|{f_{ex}}''\|_{L^\infty(J)} 
K\|f-f_{ex}\|_{C(J)}\\
&\leq (1+K\|{f_{ex}}''\|_{L^\infty(J)}) \|f-f_{ex}\|_{C^1(J))}\,,
\end{aligned}
\]
hence, using \eqref{eqn:expdecayu2}, we obtain
\[
\|r_{II,\mu}\|_{L^p(\Omega\times(0,t))}
\leq (1+K\|{f_{ex}}''\|_{L^\infty(J)}) \|f-f_{ex}\|_{C^1(J))}
C_2\left(\int_0^t e^{p(\mu-c_2) s} ds\right)^{1/p}\,.
\]
We therefore choose $0<\mu<\min\{\check{\lambda}_1,c_2\}$ so that $\int_0^t e^{p(\mu-c_2) s} ds\leq \frac{1}{p(c_2-\mu)}$ and get
\[
\|r_{II,\mu}\|_{L^p(\Omega\times(0,t))}\leq C_3\|f-f_{ex}\|_{C^1(J))}
\]
with $C_3=C_2(1+K\|{f_{ex}}''\|_{L^\infty(J)})(\frac{1}{p(c_2-\mu)})^{1/p}$.

Altogether, abbreviating 
\begin{equation}\label{eqn:C}
C=2^{p-1}C_{W^{\theta,p},C}^{\mathbb{R^+}} C_{W^{2-2\theta,p},C^1}^\Omega C^A_{\mu,p}\,,
\end{equation}
we end up with the estimate 
\[
\eta(t)\leq C\Bigl(\check{c}_1^p|\Omega| \int_0^t \eta(s)\, ds + C_3^p\|f-f_{ex}\|_{C^1(J))}^p\Bigr)
\]
for $\eta(t)=e^{p\mu t}\| z(t)\|_{C^{1,\beta}(\Omega)}^p$.
Now applying Gronwalls's inequality we obtain
\[
\eta(t)\leq C C_3^p \|f-f_{ex}\|_{C^1(J))}^p\Bigl(1+C\check{c}_1^p|\Omega|\int_0^t e^{C\check{c}_1^p|\Omega|(t-s)}\, ds\Bigr)
=C C_3^p e^{C\check{c}_1^p|\Omega|t} \|f-f_{ex}\|_{C^1(J))}^p\,.
\]
Thus with $\check{c}_1$ sufficiently small so that 
\begin{equation}\label{eqn:checkc1}
C\check{c}_1^p|\Omega|<p\mu <p\min\{\check{\lambda}_1,c_2\},
\end{equation}
we get 
\[ 
\|z(T)\|_{C^{1,\beta}(\Omega)}^p\leq C C_3^p e^{-(p\mu-C\check{c}_1^p|\Omega|)T}\ \|f-f_{ex}\|_{C^1(J))}^p
% corr BK oct 25
\]
hence, via \eqref{eqn:Pg} and \eqref{eqn:contrest1}, Lipschitz continuity of $\mathbb{T}^2$ with a factor that decays exponentially with $T$.

\medskip

Here it is important to note that $K$ in \eqref{eqn:estzr0}, \eqref{eqn:contrest1} can be chosen as independent of $T$ in the dissipative setting we are considering here, due to the following lemma, whose prove can be found in the appendix. 
\begin{lemma}\label{lem:Kbar}
Let $\mathbb{L}$ be of the form \eqref{eqn:LA} and assume that $c\in C^2(\Theta)$ with 
\begin{equation}\label{eqn:c1check_lem}
\check{q}-c\leq \check{c}_1\mbox{ in }\Theta
\end{equation} 
for some $\check{q}\in C^2(\Omega)$ (depending on $x$ only), $\check{q}\geq0$, 
\begin{equation}\label{eqn:smallnessc1check}
\check{c}_1<\frac{(ep\mu)^{1/p}}{2|\Omega|^{1/p}C_{W^{\theta,p},C}^{\mathbb{R^+}} C_{W^{2-2\theta,p},C^1}^\Omega C^A_{\mu,p}}, 
\end{equation} 
for some $\mu<\check{\lambda}_1$, $p\in[1,\infty)$, where $\check{\lambda}_1>0$ is the smallest eigenvalue of $-\mathbb{L}+\check{q}$.
\\
Then there exists a constant $\bar{K}$ independent of $T$ such that for any $j\in C^{0,\beta}(\Theta)$ the solution $z$ of 
\[
z_t-\mathbb{L} z +c z = j \quad \mbox{ in }\Omega\times(0,T) 
\]
with homogeneous initial and conditions satisfies
\[
%\|z_t\|_{C^{0,\beta}(\Theta)}+
\sum_{|m|\leq2}\|D^m_x z\|_{C^{0,\beta}(\Theta)}
\leq\bar{K}\|j\|_{C^{0,\beta}(\Theta)}\,.
\]
\end{lemma}
In particular, the assumptions of this Lemma are satisfied if $c=-f_{ex}'(u_{ex})\geq0$.

\medskip

Thus we have proven the following contractivity result.
\begin{theorem} 
Let $u_0=0$, $f_{ex}(0)=0$, $\mathbb{L}$ be of the form \eqref{eqn:LA}, ${f_{ex}}'(u_{ex})\in C^2(\Theta)$, ${f_{ex}}\in X$,
and assume that $r_t$ decays exponentially  $\|r_t(t)\|_{C(\Omega)}\leq C_r e^{-c_rt}$ and that
\eqref{eqn:c1check} holds with a nonnegative potential $\check{q}\in C^2(\Omega)$ and a sufficiently small constant $\check{c}_1$, cf. \eqref{eqn:CA}, \eqref{eqn:C}, \eqref{eqn:checkc1}, \eqref{eqn:smallnessc1check}.

Then there exist $T>0$ large enough and $\rho>0$ small enough, such that the operator $\mathbb{T}^2$ is a self-mapping on $B_\rho^X(f_{ex})=\{j\in C^{2,\beta}(\mathbb{R})\, : \, j(0)=0\,, \ j'=0 \mbox{ on }\mathbb{R}\setminus J \,, \ \|j-f_{ex}\|_{C^{2,\beta}(J)}\leq \rho\}$.

Moreover, the contraction estimates 
\[
\|\mathbb{T}^2f_0-f_{ex}\|_{C^{2,\beta}(J)} \leq q \|f_0-f_{ex}\|_{C^{1,\beta}(J)}\,, \qquad
\|\mathbb{T}^{2n}f_0-f_{ex}\|_{C^{2,\beta}(J)} \leq q^n \|f_0-f_{ex}\|_{C^{1,\beta}(J)}
\]
hold for some $q\in(0,1)$ and any $f_0\in B_\rho^X(f_{ex})$.
\end{theorem}

Conclusions on uniqueness analogous to Corollary \oldref{cor:uniquenessSchauder} can be drawn.

\section{Reconstructions}\label{sec:reconstructions}

We will show the results of numerical experiments using the basic versions
of the iterative schemes defined by \eqref{eqn:opT_fiti} and \eqref{eqn:opT_titr}
for each of the two data types:
time trace data consisting of the value of $h(t) :=u(x_0,t)$ for $t\in[0,T]$;
final time data $g(x) :=u(x,T)$ for some chosen value of $T$.
The numerical results presented will be set in one space dimension
although there is no limitation in this regard
(other than computational complexity of the direct solvers) as our
unknowns are functions of a single variable.
Also, in this setting the graphical illustrations are more transparent.
Note that in one dimension the curve $\omega\subset\Omega$ becomes 
$\Omega$ itself which we take to be the unit interval.
We will also consider  only two equations as this case encompasses
most of the features of a larger system.
For notational convenience we use $u$ and $v$ for the dependent variables
in the two equations in the system.
As data we took two differing initial values $u_0(x)$ and $v_0(x)$
and as boundary conditions we used (homogeneous)
Dirichlet at the left endpoint and Neumann at the right; typically
different for each of $u$ and $v$.

\Margin{R1 C (1)}
\Margin{R1 C (4)}
\revision{
We outline below the main steps used to compile the reconstruction examples
shown throughout this section.
\begin{enumerate}
\item{} The domain and solvers used.
\begin{enumerate}
\item{} 
The domain was the rectangle $[0,L] \times [0,T]$ where we took $L=T=1$.
For a direct solver to the reaction diffusion system we used a
finite difference scheme based on the Crank-Nicolson integrator
and used this in an extrapolation mode resulting in fourth order
accuracy in space and time.
\end{enumerate}
\item{} Data assimulation and smoothing
\begin{enumerate}
\item{}
To obtain simulated data we used the direct solver at relatively low
resolution to obtain either/or a time trace $u(x_0,t)$ or final time
$u(x,T)$ values.
These values were then sampled at $S$ equally spaced points and
uniformly distributed, mean zero, random noise added to form our
simulated measurements.
 The values chosen were $S=20$ in the case of spatial data
and $S=25$ for temporal data.
\item{}
These data values were then interpolated to the entire intervals
containing $200$ in space and $300$ in time 
using an $H^2$ or $H^1$-filtering scheme
thus obtaining the working values of
$g_{\mbox{\footnotesize{meas}}}(x)$ and $h_{\mbox{\footnotesize{meas}}}(t)$.
to be used in the iterative schemes.
\item{}
For the spatial case of $g_{\mbox{\footnotesize{meas}}}(x)$
this filtering used an eigenfunction basis of the elliptic operator
to take into account the boundary conditions and projecting
onto a basis set of its eigenfunctions using $H^2$ smoothing of its
coefficients to obtain $g(x)$.
In the temporal case of $h_{\mbox{\footnotesize{meas}}}(t)$
where the only constraint is with the initial data function at $t=0$,
either a $H^1$ Tikhonov penalty term or a smoothing spline routine was used.
%Note that in this situation typically the largest error
%occurs at the endpoint $t=T$ where no constraint exists and since
%we are forcing the solution this corresponds to the largest value of $u$.
%This will be see in our reconstructions where $f_1$ and $f_2$ will
%show poorer reconstructions for larger $u$ values than for smaller ones.
\item{}
Note that the iteration schemes below themselves contains no specific
regularization (although the projection of each iteration onto the range
of the data could be considered in this light).
For more details on data smoothing and propagation of the noise through the
fixed point iteration, we refer to
\cite[Section 3.5]{KaltenbacherRundell:2019c}.
\end{enumerate}
\item{}
Algorithm for reconstructing $f_1$, $f_2$ in section~\oldref{sec:f1f2}:
\begin{enumerate}
\item{}
Set $f_1^0$, $f_2^0$ to some initial guess then for $k=0,1,2,\ldots$ 
\item{}
compute $D_tu(T)$, $D_tv(T)$ by solving
\eqref{eqn:ex_competing_species_rec} with $f_1=f_1^k$, $f_2=f_2^k$, and differentiating $u$, $v$ with respect to time
\item{} update $f_1$, $f_2$:
\begin{equation*}
\begin{aligned}
f_1^{k+1}(g_u(x))&= D_tu(x,T)-\triangle g_u(x,T) -\beta g_u(x,T) g_v(x,T)-r^u(x,T)\\
f_2^{k+1}(g_v(x))&= D_tv(x,T)-\triangle g_v(x,T) -\beta g_u(x,T) g_v(x,T)-r^v(x,T)
\end{aligned}
\end{equation*}
\end{enumerate}
The iteration for reconstructing $f_1,f_2$ from time trace data is defined analogously and involves computation of $\triangle u$, $\triangle v$ instead of $D_tu$, $D_tv$.
\item{}
Algorithm for reconstructing $\phi_1$, $\phi_2$ in section~\oldref{sec:phi1phi2}:
\begin{enumerate}
\item{}
Set $\phi_1^0$, $\phi_2^0$ to some initial guess, then for $k=0,1,2,\ldots$ 
\item{} compute $D_tu(T)$, $D_tv(T)$ by solving \eqref{phi1phi2} with $\phi_1=\phi_1^k$, $\phi_2=\phi_2^k$, and differentiating $u$, $v$ with respect to time
\item{} update $\phi_1$, $\phi_2$:
\begin{equation*}
\begin{aligned}
\phi_1^{k+1}(w(g_u(x),g_v(x)))&= \frac{1}{\beta_u}\left(D_tu(x,T)-\triangle g_u(x,T) - f_1(g_u(x,T))-r^u(x,T)\right)\nonumber \\
\phi_2^{k+1}(w(g_u(x),g_v(x)))&= \frac{1}{\beta_v}\left(D_tv(x,T)-\triangle g_v(x,T) - f_1(g_v(x,T))-r^v(x,T)\right)\nonumber
\end{aligned}
\end{equation*}
\end{enumerate}
\end{enumerate}
}

\subsection{Reconstructions of $f_1$ and $f_2$}\label{sec:f1f2}

In this first group of reconstructions  we seek the recovery of 
the reaction terms $f_1(u)$ and $f_2(v)$ and assume the interaction
terms between them are just given by a multiple of $\phi_i(w) = w =uv$.
Our equations are then
\begin{equation}\label{eqn:ex_competing_species_rec}
\begin{aligned}
D_t u - \triangle u &= f_1(u)+\beta u\cdot v +r_u(x,t,u)\\
D_t v - \triangle v &= f_2(v)+\beta u\cdot v +r_v(x,t,v)
\end{aligned}
\end{equation}
representing a ``competing species'' model if $\beta<0$ and a
``symbiotic relationship'' if $\beta>0$.
The magniture of $\beta$ represents the strength of the coupling.
The source terms $r_u(x,t,u)$ and $r_v(x,t,v)$
are assumed known if present.
We used a homogeneous Dirichlet boundary conditions at $x=0$
and Neumann conditions forcing in flux at $x=1$.
The initial conditions $u_0$ and $v_0$ were different as were
$r_u$ and $r_v$,
\[
\begin{aligned}
&u_0(x) = x(1 -2x + x^2), \qquad 
&&v_0(x)=\sin(\tfrac{\pi}{2}x)\,,
\\
&r^u(x,t)=10\sin(\tfrac{\pi}{2} x)t, \qquad 
&&r^v(x,t)=12(2x - x^2)t \,.
\end{aligned}
\]

The following sample functions to be reconstructed were used:
\begin{equation}\label{eqn:sample_f}
f_1(u) = 2u(1-u)(u-0.9),\qquad
f_2(u) =  \max\{2e^{-5(v-1)^2}-0.1v^2,-2\}
\end{equation}
The first of these is a version of the Zeldovich combustion model
with chosen parameters that are physically relevant,
the second is chosen to offer more challenge  to the reconstruction process.
Depending on the driving boundary conditions
and the strengths of the interaction terms $\phi_i$ the range over which
one must recover $f_i$ can be considerable.
This offers challenges from a computational viewpoint and of course if
one ``knew'' that the correct answer was a polynomial function
all the reconstruction process would be nothing other than a least squares
fit in some appropriate norm to obtain a small number of constants.
We are looking beyond this here and hoping to be able to detect
features in these reaction and coupling terms that might drive the
model rather than be purely derivative from it. 

The iteration schemes \eqref{eqn:opT_fiti} and \eqref{eqn:opT_titr} can be implemented pointwise
or by representing the unknowns in a set of basis functions.
The choice of the latter is important.
While many standard models use only low degree polynomials to represent
the modelling of $f$ this is clearly a severe limitation.
Using high degree polynomials
is out of the question due to the severe ill-conditioning recovering
Taylor coefficients from data far from the initial point.
Rational functions may seem to be a good choice but again
their range of accuracy is limited when used over a wide interval.
In addition, this is a nonlinear fitting problem - it is also unstable
under extrapolation as it is again analytic continuation.
There are also other negative effects.
Our initial guesses for $f_i$ may be quite distant from the actual
and in this situation during the iteration process it was frequently
found that the denominator of the rational function has zeros sufficiently
near to the real axis that reconstructed functions $f_k$ of
large amplitude resulted at a local point.
This had the effect of causing failure of the iteration scheme to converge.
Indeed basis functions designed for a fixed and relatively narrow range
tend to be suboptimal in this setting.

We found a good choice to be moving Gaussian basis functions.
These are known to effectively model many non-linear relationships
and since each basis function is non-zero over a small interval this
localization is useful in the current situation where $f_1(u)$ and $f_2(v)$ are
totally locally defined.
As a secondary consideration here
this locality property results in a sparse matrix 
that leads to much faster computation of this phase.

As noted earlier, we are initially taking
the interaction functions $\{\phi_1(w),\phi_2(w)\}$
to be a constant $\beta$ times the identity and also $w=uv$.
If $\beta=0$ then this case is just a complete decoupling of the system
and the results of \cite{KaltenbacherRundell:2019c} show a unique recovery
through the resulting contraction mappings.
For $\beta$ sufficiently small, the analysis of section~\oldref{sec:convergence}
then shows the same 
result and the question becomes if this holds true for all $\beta$.
Our analysis does not cover this case and as we will see below this answer
certainly appears to be negative and is illustrated graphically in
Figure~\oldref{fig:cgt_rates_beta} below.

\Margin{R1 C (2)}
\revision{
The norms shown in these figures are discrete $L^2$ norms of
the functions $f$ at the $100$ stored values as described previously.
}
We also show reconstructions of $f_1$ and $f_2$ achieved after
a given number of iterations in
Figures~\oldref{fig:titr_reconstructions_f_with_beta} and
\oldref{fig:fiti_reconstructions_f_with_beta}.
%
% values below are important - the x length and y height of each picture
\newdimen\xfiglen
\newdimen\yfiglen
\xfiglen=2.3 true in  
\yfiglen=1.4 true in
%
%%%%%%%%%%%%%%%%%%%%%%%%%%%%%%%%%%%%%%%%%%%%%%%%%%%%%%%%%%%%%%%%%%%%%%%%%%%%
%
% This is the file fonts.tex
%
%%%%%%%%%%%%%%%%%%%%%%%%%%%%%%%%%%%%%%%%%%%%%%%%%%%%%%%%%%%%%%%%%%%%%%%%%%%%
%
\font\tenrm=cmr10
\font\teni=cmmi10 \skewchar\teni='177
\font\tensy=cmsy10 \skewchar\tensy='60
\font\tenex=cmex10
\font\tenit=cmti10
\font\tensl=cmsl10
\font\tenbf=cmbx10
\font\tentt=cmtt10
\font\ninerm=cmr9
\font\ninei=cmmi9 \skewchar\ninei='177
\font\ninesy=cmsy9 \skewchar\ninesy='60
\font\nineit=cmti9
\font\ninesl=cmsl9
\font\ninebf=cmbx9
\font\ninett=cmtt9
\font\eightrm=cmr8
\font\eighti=cmmi8 \skewchar\eighti='177
\font\eightsy=cmsy8 \skewchar\eightsy='60
\font\eightit=cmti8
\font\eightsl=cmsl8
\font\eightbf=cmbx8
\font\eighttt=cmtt8
\font\sevenrm=cmr7
\font\seveni=cmmi7 \skewchar\seveni='177
\font\sevensy=cmsy7 \skewchar\sevensy='60
\font\sevenbf=cmbx7
\font\sevenit=cmmi7
\font\sevensl=cmmi7
\font\seventt=cmr7
\font\sixrm=cmr6
\font\sixi=cmmi6 \skewchar\sixi='177
\font\sixsy=cmsy6 \skewchar\sixsy='60
\font\sixbf=cmbx6
\font\fiverm=cmr5
\font\fivei=cmmi5 \skewchar\fivei='177
\font\fivesy=cmsy5 \skewchar\fivesy='60
\font\fivebf=cmbx5
\def\tenpoint{\def\rm{\fam0\tenrm}%
        \textfont0=\tenrm \scriptfont0=\sevenrm \scriptscriptfont0=\fiverm
        \textfont1=\teni \scriptfont1=\seveni \scriptscriptfont1=\fivei
        \textfont2=\tensy \scriptfont2=\sevensy \scriptscriptfont2=\fivesy
        \textfont3=\tenex \scriptfont3=\tenex \scriptscriptfont3=\tenex
        \def\it{\fam\itfam\tenit}%
        \textfont\itfam=\tenit
        \def\sl{\fam\slfam\tensl}%
        \textfont\slfam=\tensl
        \def\bf{\fam\bffam\tenbf}%
        \textfont\bffam=\tenbf \scriptfont\bffam=\sevenbf
                \scriptscriptfont\bffam=\fivebf
        \def\tt{\fam\ttfam\tentt}%
        \textfont\ttfam=\tentt
        \normalbaselineskip=12pt%
        \let\sc=\eightrm        % Small caps
        \setbox\strutbox=\hbox{\vrule height8.5pt depth3.5pt width0pt}%
        \normalbaselines\rm}
\def\ninepoint{\def\rm{\fam0\ninerm}%
        \textfont0=\ninerm \scriptfont0=\sixrm \scriptscriptfont0=\fiverm
        \textfont1=\ninei \scriptfont1=\sixi \scriptscriptfont1=\fivei
        \textfont2=\ninesy \scriptfont2=\sixsy \scriptscriptfont2=\fivesy
        \textfont3=\tenex \scriptfont3=\tenex \scriptscriptfont3=\tenex
        \def\it{\fam\itfam\nineit}%
        \textfont\itfam=\nineit
        \def\sl{\fam\slfam\ninesl}%
        \textfont\slfam=\ninesl
        \def\bf{\fam\bffam\ninebf}%
        \textfont\bffam=\ninebf \scriptfont\bffam=\sixbf
                \scriptscriptfont\bffam=\fivebf
        \def\tt{\fam\ttfam\ninett}%
        \textfont\ttfam=\ninett
        \normalbaselineskip=11pt%
        \let\sc=\sevenrm        % Small caps
        \setbox\strutbox=\hbox{\vrule height8pt depth3pt width0pt}%
        \normalbaselines\rm}
\def\eightpoint{\def\rm{\fam0\eightrm}%
        \textfont0=\eightrm \scriptfont0=\sixrm \scriptscriptfont0=\fiverm
        \textfont1=\eighti \scriptfont1=\sixi \scriptscriptfont1=\fivei
        \textfont2=\eightsy \scriptfont2=\sixsy \scriptscriptfont2=\fivesy
        \textfont3=\tenex \scriptfont3=\tenex \scriptscriptfont3=\tenex
        \def\it{\fam\itfam\eightit}%
        \textfont\itfam=\eightit
        \def\sl{\fam\slfam\eightsl}%
        \textfont\slfam=\eightsl
        \def\bf{\fam\bffam\eightbf}%
        \textfont\bffam=\eightbf \scriptfont\bffam=\sixbf
                \scriptscriptfont\bffam=\fivebf
        \def\tt{\fam\ttfam\eighttt}%
        \textfont\ttfam=\eighttt
        \normalbaselineskip=9pt%
        \let\sc=\sixrm  % Small caps
        \setbox\strutbox=\hbox{\vrule height7pt depth2pt width0pt}%
        \normalbaselines\rm}
\def\sevenpoint{\def\rm{\fam0\sevenrm}%
        \textfont0=\sevenrm \scriptfont0=\fiverm \scriptscriptfont0=\fiverm
        \textfont1=\seveni \scriptfont1=\fivei \scriptscriptfont1=\fivei
        \textfont2=\sevensy \scriptfont2=\fivesy \scriptscriptfont2=\fivesy
        \textfont3=\tenex \scriptfont3=\tenex \scriptscriptfont3=\tenex
        \def\it{\fam\itfam\sevenit}%
        \textfont\itfam=\sevenit
        \def\sl{\fam\slfam\sevensl}%
        \textfont\slfam=\sevensl
        \def\bf{\fam\bffam\sevenbf}%
        \textfont\bffam=\sevenbf \scriptfont\bffam=\fivebf
                \scriptscriptfont\bffam=\fivebf
        \def\tt{\fam\ttfam\seventt}%
        \textfont\ttfam=\seventt
        \normalbaselineskip=8pt%
        \let\sc=\fiverm  % Small caps
        \setbox\strutbox=\hbox{\vrule height6pt depth2pt width0pt}%
        \normalbaselines\rm}

\input pictex
\input colordvi
\newdimen\xfigdim
\newdimen\yfigdim
\newbox\figurelegendone
\newbox\figurelegendtwo
\newbox\figurelegendthree
\newbox\figurelegendfour
\newbox\figureone
\newbox\figuretwo
\newbox\figurethree 
\newbox\figurefour
\newbox\figurefive
\newbox\figuresix
\newbox\figureseven
\newbox\figureeight
%
%\definecolor{BBBlue}{rgb}{.63,.79,.95}

%%%%%%%%%%%%%%%%%%%
%
\setbox\figurelegendone=\hbox{
\beginpicture
  \setcoordinatesystem units <0.3\xfiglen,0.11\yfiglen> %point at 0 -0.7
  \setplotarea x from 0 to 1, y from 0 to 4
\footnotesize
\eightrm
  \put {\Orange{$\star$}} [lb] at 0 0.14 \put {$\beta=-1$} [lb] at 0.25 0 
  \put {\Red{$\diamond$}} [lb] at 0 1.14 \put {$\beta=0.3$} [lb] at 0.25 1
  \put {\PineGreen{$\circ$}} [lb] at 0 2.14 \put {$\beta=1$} [lb] at 0.25 2
  \put {\Cyan{$\bullet$}} [lb] at 0 3.14 \put {$\beta=1.3$} [lb] at 0.25 3
\endpicture
}
\setbox\figurelegendtwo=\hbox{
\beginpicture
  \setcoordinatesystem units <0.3\xfiglen,0.11\yfiglen> %point at 0 -0.7
  \setplotarea x from 0 to 1, y from 0 to 4
\footnotesize
  \put {\Orange{$\star$}} [lb] at 0 0.14
  \put {$\beta\!=\!-1\quad T\!=\!1$} [lb] at 0.25 0
%  \put {$T\!=\!1$} [lb] at 0.85 0
  \put {\Red{$\diamond$}} [lb] at 0 1.14
  \put {$\beta\!=\!0.3\quad T\!=\!1$} [lb] at 0.25 1
%\put {$T\!=\!1$} [lb] at 0.85 1
  \put {\PineGreen{$\circ$}} [lb] at 0 2.14
  \put {$\beta\!=\!0.5\quad T\!=\!0.75$} [lb] at 0.25 2 
  %\put {$T\!=\!0.75$} [lb] at 0.85 2
  %\put {\Cyan{$\bullet$}} [lb] at 0 3.14 \put {$\beta=1.3$} [lb] at 0.25 3
\endpicture
}
\setbox\figurelegendthree=\hbox{
\beginpicture
  \setcoordinatesystem units <0.3\xfiglen,0.11\yfiglen> %point at 0 -0.7
  \setplotarea x from 0 to 1, y from -1.1 to 4
\footnotesize
\eightrm
\setdashes <3pt> \putrule from 0.12 -1 to 0.36 -1
  \put {\Black{$\bullet$}} [lb] at 0 -1
  \put {$w=u^2v$} [lb] at 0.45 -1.1
  \put {\Orange{$\star$}} [lb] at 0 0.14 \put {$\beta=-1$} [lb] at 0.25 0 
  \put {\Red{$\diamond$}} [lb] at 0 1.14 \put {$\beta=0.1$} [lb] at 0.25 1
  \put {\PineGreen{$\circ$}} [lb] at 0 2.14 \put {$\beta=1$} [lb] at 0.25 2
  \put {\Cyan{$\bullet$}} [lb] at 0 3.14 \put {$\beta=10$} [lb] at 0.25 3
\endpicture
}
\xfigdim=0.06\xfiglen
\yfigdim=\yfiglen

\setbox\figureone=\vbox{\hsize=\xfiglen  % f_1   Time Trace
\beginpicture
\footnotesize
  \setcoordinatesystem units <\xfigdim,\yfigdim>  %point at 0 -60
  \setplotarea x from 1 to 15, y from 0 to 1
  \axis bottom shiftedto y=0 ticks numbered from 5 to 15 by 5 unlabeled short quantity 15 /
  \axis left ticks short numbered from 0 to 1 by 0.2 /
 \put {$\|f_1^{(n)}-f_1\|_2$} [lb] at 1.1 1
 \put {$n$} [lb] at 15 0.03
\put {\copy\figurelegendone} [rt] at 15 1
%
%\Orange{\relax % beta = -1
\multiput {\Orange{$\star$}} at  
1	 0.5725
2	 0.2382
3	 0.0945
4	 0.0418
5	 0.0238
6	 0.0179
7	 0.0161
8	 0.0155
9	 0.0153
10	 0.0077 /
% /}\relax
%
\Red{\relax % beta = 0.3
\multiput {$\diamond$} at  
1	0.9121
2	 0.4110
3	 0.1868
4	 0.0821
5	 0.0358
6	 0.0158
7	 0.0081
8	 0.0048
9	 0.0037
10	 0.0035
 /}\relax
\PineGreen{\relax % beta = 1
\multiput {$\circ$} at  
1	 0.5810
2	 0.3596
3	 0.2020
4	 0.1051
5	 0.0509
6	 0.0227
7	 0.0093
8	 0.0036
9	 0.0016
10	 0.0013
 /}\relax
\Cyan{\relax % beta = 1.3
\multiput {$\bullet$} at  
1	0.8539
2	0.9020
3	0.8719
4	0.7553
5	0.5815
6	0.3997
7	0.2472
8	0.1381
9	0.0695
10	0.0310
11	0.0122
12	0.0043
13	0.0016
14	0.0012
15	0.0012
 /}\relax
\endpicture
}   % end of time trace data for f1 $\beta=1$
%
%\xfiglen=0.6true in
\yfigdim=0.143\yfiglen
\setbox\figuretwo=\vbox{\hsize=\xfiglen   % f_2  Time Trace
\beginpicture
\small
\footnotesize
  \setcoordinatesystem units <\xfigdim,\yfigdim>  %point at 0 -60
  \setplotarea x from 1 to 15, y from 0 to 7
  \axis bottom shiftedto y=0 ticks numbered from 5 to 15 by 5 unlabeled short quantity 15 /
  \axis left ticks short numbered from 0 to 7 by 1 /
 \put {$\|f_2^{(n)}-f_2\|_2$} [lb] at 1.1 7
 \put {$n$} [lb] at 15 0.3
\put {\copy\figurelegendone} [rt] at 15 7
\setdashes <3pt>
\setsolid
\Orange{\relax  % beta = -1
\multiput {$\star$} at
1	  0.2509
2	  0.0636
3	  0.0267
4	  0.0161
5	  0.0131
6	  0.0127
7	  0.0129
8	  0.0130
9	  0.0131
10	  0.0132
 /}\relax
\Red{\relax % beta = 0.3
\multiput {$\diamond$} at  
1	   0.9673
2	   0.7382
3	   0.4716
4	   0.2702
5	   0.1478
6	   0.0814
7	   0.0486
8	   0.0349
9	   0.0302
10	   0.0288
 /}\relax
\PineGreen{\relax 
\multiput {$\circ$} at  
1	 1.4021
2	 1.1325
3	 0.7027
4	 0.3803
5	 0.1886
6	 0.0883
7	 0.0411
8	 0.0222
9	 0.0171
10	 0.0164
 /}\relax
\Cyan{\relax 
\multiput {$\bullet$} at  
1	  5.0254
2	  6.8811
3	  7.0678
4	  6.1377
5	  4.5934
6	  3.0148
7	  1.7712
8	  0.9464
9	  0.4644
10	  0.2110
11	  0.0901
12	  0.0381
13	  0.0206
14	  0.0178
15	  0.0181
 /}\relax
\endpicture
}
\xfigdim=0.06\xfiglen
\yfigdim=\yfiglen
%\yfigdim=0.25\yfiglen
%
\setbox\figurethree=\vbox{\hsize=\xfiglen   % f_1 Final Time
\beginpicture
\small
\footnotesize
  \setcoordinatesystem units <\xfigdim,\yfigdim>  %point at 0 -60
  \setplotarea x from 1 to 15, y from 0 to 1
  \axis bottom shiftedto y=0 ticks numbered from 5 to 15 by 5 unlabeled short quantity 15 /
  \axis left ticks short numbered from 0 to 1 by 0.2 /
 \put {$\|f_1^{(n)}-f_1\|_2$} [lb] at 1.1 1
 \put {$n$} [lb] at 15 0.04
\put {\copy\figurelegendtwo} [rt] at 15 1
\setdashes <3pt>
\setsolid
\Orange{\relax  % beta = -1
\multiput {$\star$} at
1	 0.4829
2	 0.0961
3	 0.0308
4	 0.0208
5	 0.0159
6	 0.0141
7	 0.0135
8	 0.0134
 /}\relax
\Red{\relax % beta = 0.3  T = 1
\multiput {$\diamond$} at  
1	 0.9121
2	 0.4110
3	 0.1868
4	 0.0821
5	 0.0358
6	 0.0158
7	 0.0081
8	 0.0048
9	 0.0037
10	 0.0035
 /}\relax
\PineGreen{\relax   % beta = 0.5  T = 0.75
\multiput {$\circ$} at  
1	 0.7128
2	 0.6724
3	 0.4311
4	 0.2853
5	 0.1988
6	 0.1327
7	 0.0847
8	 0.0536
9	 0.0339
10	 0.0214
11	 0.0140
12	 0.0106
13	 0.0098
14	 0.0099
15	 0.0103
 /}\relax
%
%\Cyan{\relax 
%\multiput {$\bullet$} at  
% /}\relax
%
\endpicture
}
\setbox\figurefour=\vbox{\hsize=\xfiglen   % f_2  Time Trace
\beginpicture
\small
\footnotesize
  \setcoordinatesystem units <\xfigdim,\yfigdim>  %point at 0 -60
  \setplotarea x from 1 to 15, y from 0 to 1
  \axis bottom shiftedto y=0 ticks numbered from 5 to 15 by 5 unlabeled short quantity 15 /
  \axis left ticks short numbered from 0 to 1 by 0.2 /
 \put {$\|f_2^{(n)}-f_2\|_2$} [lb] at 1.1 1
 \put {$n$} [lb] at 15 0.04
\put {\copy\figurelegendtwo} [rt] at 15 1
\setdashes <3pt>
\setsolid
\Orange{\relax  % beta = -1
\multiput {$\star$} at
1	 0.3002
2	 0.0246
3	 0.0146
4	 0.0147
5	 0.0148
6	 0.0148
7	 0.0149
8	 0.0149
 /}\relax
\Red{\relax % beta = 0.3
\multiput {$\diamond$} at  
1	 0.9673
2	 0.7382
3	 0.4716
4	 0.2702
5	 0.1478
6	 0.0814
7	 0.0486
8	 0.0349
9	 0.0302
10	 0.0288
 /}\relax
\PineGreen{\relax 
\multiput {$\circ$} at  
1	 0.5067
2	 0.2879
3	 0.2108
4	 0.1506
5	 0.1042
6	 0.0720
7	 0.0498
8	 0.0346
9	 0.0249
10	 0.0191
11	 0.0160
12	 0.0145
13	 0.0139
14	 0.0137
15	 0.0136
 /}\relax
%
%\Cyan{\relax 
%\multiput {$\bullet$} at  
% /}\relax
%
\endpicture
}
\xfigdim=0.167\xfiglen
\yfigdim=4\yfiglen
\setbox\figurefive=\vbox{\hsize=\xfiglen   % f_2  Time Trace
\beginpicture
\small
\footnotesize
  \setcoordinatesystem units <\xfigdim,\yfigdim>  %point at 0 -60
  \setplotarea x from 1 to 6, y from 0 to 0.25
  \axis bottom shiftedto y=0 ticks numbered from 1 to 6 by 1 /
  \axis left ticks short numbered from 0 to 0.25 by 0.05 /
 \put {$\|\phi_1^{(n)}-\phi_1\|$} [l] at 1.05 0.25
 \put {$n$} [lb] at 6.04 0.002
\put {\copy\figurelegendthree} [rt] at 6 0.25
\setdashes <3pt>
\setsolid
\setlinear
\linethickness=0.3pt
\Orange{\relax  % beta = -1,  -1  fiti  T=1
\multiput {$\star$} at
1	 0.2544
2	 0.0435
3	 0.0093
4	 0.0063
5	 0.0036
6	 0.0033 /
\plot
1	 0.2544
2	 0.0435
3	 0.0093
4	 0.0063
5	 0.0036
6	 0.0033 /
 }\relax
\Red{\relax % beta = 0.1  fiti T=1
\multiput {$\diamond$} at  
1	 0.1900
2	 0.0294
3	 0.0200
4	 0.0203
5	 0.0201
6	 0.0201 /
\plot
1	 0.1900
2	 0.0294
3	 0.0200
4	 0.0203
5	 0.0201
6	 0.0201 /
 }\relax
\PineGreen{\relax  %beta = 1 fiti T=1
\multiput {$\circ$} at  
1	 0.1370
2	 0.0164
3	 0.0133
4	 0.0132
5	 0.0132
6	 0.0132 /
\plot
1	 0.1370
2	 0.0164
3	 0.0133
4	 0.0132
5	 0.0132
6	 0.0132 /
}\relax
\Cyan{\relax 
\multiput {$\bullet$} at  
1	 0.0936
2	 0.0809
3	 0.0808
4	 0.0808
5	 0.0808
6	 0.0808 /
\plot
1	 0.0936
2	 0.0809
3	 0.0808
4	 0.0808
5	 0.0808
6	 0.0808 /
 }\relax
\multiput {$\bullet$} at  
1      0.1092
2      0.0254
3      0.0229
4      0.0228
5      0.0228 
6      0.0228  /
\setdashes
\plot
1      0.1092
2      0.0254
3      0.0229
4      0.0228
5      0.0228 
6      0.0228  /
\endpicture
}
\setbox\figuresix=\vbox{\hsize=\xfiglen   % \phi_2  Time Trace
\beginpicture
\small
\footnotesize
  \setcoordinatesystem units <\xfigdim,\yfigdim>  %point at 0 -60
  \setplotarea x from 1 to 6, y from 0 to 0.25
  \axis bottom shiftedto y=0 ticks numbered from 1 to 6 by 1 /
  \axis left ticks short numbered from 0 to 0.25 by 0.05 /
  \axis bottom shiftedto y=0 ticks numbered from 1 to 6 by 1 /
  \axis left ticks short numbered from 0 to 0.25 by 0.05 /
 \put {$\|\phi_2^{(n)}-\phi_2\|$} [l] at 1.05 0.25
 \put {$n$} [lb] at 6.04 0.002
\put {\copy\figurelegendthree} [rt] at 6 0.25
\setdashes <3pt>
\setsolid
\setlinear
\linethickness=0.3pt
\Orange{\relax  % beta = -1,  -1  fiti  T=1
\multiput {$\star$} at
1	 0.2208
2	 0.0132
3	 0.0113
4	 0.0043
5	 0.0033
6	 0.0032 /
\plot
1	 0.2208
2	 0.0132
3	 0.0113
4	 0.0043
5	 0.0033
6	 0.0032 /
}\relax
\Red{\relax % beta = 0.3    both beta = 0.1  T=1
\multiput {$\diamond$} at  
1	 0.0526
2	 0.0354
3	 0.0247
4	 0.0232
5	 0.0231
6	 0.0230 /
\plot
1	 0.0526
2	 0.0354
3	 0.0247
4	 0.0232
5	 0.0231
6	 0.0230 /
}\relax
\PineGreen{\relax  %beta = 1 fiti T=1
\multiput {$\circ$} at  
1	 0.0248
2	 0.0112
3	 0.0040
4	 0.0038
5	 0.0038
6	 0.0038 /
\plot
1	 0.0248
2	 0.0112
3	 0.0040
4	 0.0038
5	 0.0038
6	 0.0038 /
 }\relax
\Cyan{\relax     %beta = 10  T=1  fiti
\multiput {$\bullet$} at  
1	 0.0392
2	 0.0156
3	 0.0156
4	 0.0156
5	 0.0156
6	 0.0156 /
\plot
1	 0.0392
2	 0.0156
3	 0.0156
4	 0.0156
5	 0.0156
6	 0.0156 /
 }\relax
\multiput {$\bullet$} at  
1      0.0394
2      0.0096
3      0.0049
4      0.0051
5      0.0051 
6      0.0051  /
\setdashes
\plot
1      0.0394
2      0.0096
3      0.0049
4      0.0051
5      0.0051 
6      0.0051  /
\endpicture
}
\yfigdim=3\yfiglen
\setbox\figureseven=\vbox{\hsize=\xfiglen   % phi_1  Time Trace
\beginpicture
\small
\footnotesize
  \setcoordinatesystem units <\xfigdim,\yfigdim>  %point at 0 -70
  \setplotarea x from 1 to 6, y from 0 to 0.33
  \axis bottom shiftedto y=0 ticks numbered from 1 to 6 by 1 /
  \axis left ticks short numbered from 0 to 0.3  by 0.1 /
 \put {$\|\phi_1^{(n)}-\phi_1\|$} [lt] at 1.08 0.34
 \put {$n$} [lb] at 6 0.04
\put {\copy\figurelegendthree} [rt] at 6 0.32
\setdashes <3pt>
\setsolid
\linethickness=0.2pt
\Orange{\relax  % beta = -1,  -1  titr  T=1
\multiput {$\star$} at
1	 0.3214
2	 0.1365
3	 0.0633
4	 0.0296
5	 0.0135
6	 0.0062
%7        0.0036 /
/
\plot
1	 0.3214
2	 0.1365
3	 0.0633
4	 0.0296
5	 0.0135
6	 0.0062 /
 }\relax
\Red{\relax % beta = 0.1  titr T=1
\multiput {$\diamond$} at  
1	 0.2641
2	 0.0697
3	 0.0248
4	 0.0197
5	 0.0203
6	 0.0204 /
%7	 0.0203 /
\plot
1	 0.2641
2	 0.0697
3	 0.0248
4	 0.0197
5	 0.0203
6	 0.0204 /
 }\relax
\PineGreen{\relax  %beta = 1 titr T=1
\multiput {$\circ$} at  
1	 0.2447
2	 0.0569
3	 0.0179
4	 0.0089
5	 0.0067
6	 0.0061 /
%7	 0.0060
\plot
1	 0.2447
2	 0.0569
3	 0.0179
4	 0.0089
5	 0.0067
6	 0.0061 /
 }\relax
\Cyan{\relax 
\multiput {$\bullet$} at  
1	 0.1590
2	 0.0305
3	 0.0218
4	 0.0203
5	 0.0199
6	 0.0198 /
%7	 0.0198 /
\plot
1	 0.1590
2	 0.0305
3	 0.0218
4	 0.0203
5	 0.0199
6	 0.0198 /
 }\relax
\multiput {$\bullet$} at  
1      0.2391
2      0.0533
3      0.0171
4      0.0116
5      0.0110 
6      0.0109  /
\setdashes
\plot
1      0.2391
2      0.0533
3      0.0171
4      0.0116
5      0.0110 
6      0.0109  /
\endpicture
}
\setbox\figureeight=\vbox{\hsize=\xfiglen   % \phi_2  Time Trace
\beginpicture
\small
\footnotesize
  \setcoordinatesystem units <\xfigdim,\yfigdim>  %point at 0 -60
  \setplotarea x from 1 to 6, y from 0 to 0.33
  \axis bottom shiftedto y=0 ticks numbered from 1 to 6 by 1 /
  \axis left ticks short numbered from 0 to 0.3 by 0.1 /
 \put {$\|\phi_2^{(n)}-\phi_2\|$} [lt] at 1.08 0.34
 \put {$n$} [lb] at 6 0.04
\put {\copy\figurelegendthree} [rt] at 6 0.32
\setdashes <3pt>
\setsolid
\Orange{\relax  % beta = -1,  -1  titr  T=1
\multiput {$\star$} at
1	 0.3211
2	 0.1001
3	 0.0346
4	 0.0138
5	 0.0064
6	 0.0039 /
%7	 0.0034
\plot
1	 0.3211
2	 0.1001
3	 0.0346
4	 0.0138
5	 0.0064
6	 0.0039 /
 }\relax
\Red{\relax %    both beta = 0.1  T=1
\multiput {$\diamond$} at  
1	 0.2693
2	 0.0731
3	 0.0296
4	 0.0248
5	 0.0241
6	 0.0238  /
%7	 0.0236
\plot
1	 0.2693
2	 0.0731
3	 0.0296
4	 0.0248
5	 0.0241
6	 0.0238 /
 }\relax
\PineGreen{\relax % beta = 1 titr T=1
\multiput {$\circ$} at  
1	 0.2427
2	 0.0556
3	 0.0138
4	 0.0060
5	 0.0043
6	 0.0036 /
%7	 0.0034 
\plot
1	 0.2427
2	 0.0556
3	 0.0138
4	 0.0060
5	 0.0043
6	 0.0036 /
 }\relax
\Cyan{\relax   %  beta = 10  T=1  titr
\multiput {$\bullet$} at  
1	 0.1220
2	 0.0177
3	 0.0092
4	 0.0067
5	 0.0061
6	 0.0060 /
%7	 0.0060
\plot
1	 0.1220
2	 0.0177
3	 0.0092
4	 0.0067
5	 0.0061
6	 0.0060 /
 }\relax
\multiput {$\bullet$} at  
1      0.2827
2      0.0849
3      0.0269
4      0.0096
5      0.0047 
6      0.0035  /
\setdashes
\plot
1      0.2827
2      0.0849
3      0.0269
4      0.0096
5      0.0047 
6      0.0035  /
\endpicture
}

\begin{figure}[h]
\hbox to\hsize{\hss\copy\figureone\hss\hss\copy\figuretwo\hss}
\medskip
\hbox to\hsize{\hss\copy\figurethree\hss\hss\copy\figurefour\hss}
\smallskip
\caption{\small
Convergence rates of iterations as a function of $\beta$.
Top: time trace, Bottom: final time.
}
\label{fig:cgt_rates_beta}
\end{figure}
%
%  \PiCTeX file for BB5 - time trace data for $f_1$ and $f_2$
\input colordvi

\input pictex
\font\smallsymbol = cmmi8
%\newdimen\xfiglen \newdimen\yfiglen
\newdimen\xfigdim \newdimen\yfigdim
\newdimen\xmindim \newdimen\ymindim
\newdimen\xmaxdim \newdimen\ymaxdim
\newbox\figurelegendone
\newbox\figureone
\newbox\figuretwo
\newbox\figurethree 
\newbox\figurefour
\newbox\figurefive
\newbox\figuresix
\newbox\figureseven
\newbox\figureeight
\newbox\figurenine
\newbox\figureten
%%%%%%%%%%%%%%%%%%%
%\xfiglen=2.5true in
%\yfiglen=1.6true in
%%%%%%%%%%%%%%%%%%%%

\setbox\figurelegendone=\hbox{
\small
%\eightpoint
\beginpicture
  \setcoordinatesystem units <0.3\xfiglen,0.5\yfiglen> %point at 0 -0.7
  \setplotarea x from 0 to 0.8, y from 0 to 0.7
\linethickness=0.7pt
\ninerm
\setsolid
  \Goldenrod{\relax
  \putrule from 0 0.0 to 0.25 0.0 }\relax
  \put {iteration 10} [l] at 0.35 0.0
  \putrule from 0 0.2 to 0.25 0.2  %\relax
  \Orange{\relax
  \putrule from 0 0.2 to 0.25 0.2 }\relax
  \put {iteration 3} [l] at 0.35 0.2
  \LimeGreen{\relax
  \putrule from 0 0.4 to 0.25 0.4 }\relax
  \putrule from 0 0.4 to 0.2 0.4
  \put {iteration 1} [l] at 0.35 0.4
%\setquadratic
  \setdashes <3pt>
  \Black
  \putrule from 0 0.6 to 0.25 0.6
  \put {actual $q$} [l] at 0.35 0.6
\endpicture
\relax
}
\newdimen\xxfiglen \newdimen\yyfiglen
\xxfiglen=0.5true in
\yyfiglen=0.1true in
\setbox\figurelegendtwo=\hbox{
\beginpicture
  \setcoordinatesystem units <\xxfiglen,\yyfiglen> %point at 0 -0.7
  \setplotarea x from 0 to 2, y from 0 to 5
\footnotesize
\linethickness=0.7pt
\eightrm
   \setdashes <3pt>  \putrule from 0 5 to 0.5 5
   \setsolid
   \Goldenrod{\relax\putrule from 0 4 to 0.5 4 \relax}\relax
   \Orange{\relax\putrule from 0 3 to 0.5 3 \relax}\relax
   \Red{\relax\putrule from 0 2 to 0.5 2 \relax}\relax
   \LimeGreen{\relax\putrule from 0 1 to 0.5 1 \relax}\relax
   \PineGreen{\relax\putrule from 0 0 to 0.5 0 \relax}\relax
  \setsolid
  \put {$f_{\hbox{act}}$}  [lb] at 0.8 5
  \put {${\hbox{iter}\,1}$}  [l] at 0.8 4 
  \put {$_{\hbox{iter}\,2}$}  [l] at 0.8 3
  \put {$_{\hbox{iter}\,3}$} [l] at 0.8 2
  \put {$_{\hbox{iter}\,5}$} [l] at 0.8 1
  \put {$_{\hbox{iter}\,10}$} [l] at 0.8 0
\endpicture
}
\xfigdim=0.571\xfiglen
\yfigdim=0.67\yfiglen
\setbox\figureone=\vbox{\hsize=\xfiglen
% This is for: time trace data, beta=-1, T=1, f_1(u)
\beginpicture
\footnotesize
  \setcoordinatesystem units <\xfigdim,\yfigdim>  point at 0 -1.5
  \setplotarea x from 0 to 1.75, y from -1.5 to 0
  \axis bottom shiftedto y=-1.5 ticks short numbered from 0 to 1.5 by 0.5 /
% unlabeled short quantity 8 / /
  \axis left ticks short numbered from -1.5 to 0 by 0.5 /
% unlabeled short quantity 16 / /
 \put {{\sevenrm $f_1(u)$}} [lb] at 0.02 0
 \put {{\sevenrm $u$}} [rt] at 1.75 -1.54
\put {\copy\figurelegendtwo} [l] at 0.3 -1
\setquadratic
\setdashes <3pt>
\Black{
\plot
         0         0
    0.0484   -0.0784
    0.0968   -0.1404
    0.1452   -0.1873
    0.1935   -0.2205
    0.2419   -0.2414
    0.2903   -0.2512
    0.3387   -0.2514
    0.3871   -0.2434
    0.4355   -0.2284
    0.4838   -0.2079
    0.5322   -0.1831
    0.5806   -0.1555
    0.6290   -0.1265
    0.6774   -0.0973
    0.7258   -0.0694
    0.7742   -0.0440
    0.8225   -0.0226
    0.8709   -0.0065
    0.9193    0.0029
    0.9677    0.0042
    1.0161   -0.0038
    1.0645   -0.0226
    1.1129   -0.0535
    1.1612   -0.0978
    1.2096   -0.1570
    1.2580   -0.2324
    1.3064   -0.3253
    1.3548   -0.4372
    1.4032   -0.5693
    1.4515   -0.7230
    1.4999   -0.8997
    1.5483   -1.1008 %
    1.5967   -1.3276
    1.6451   -1.5814
/
}\relax
\setsolid
\Goldenrod{\relax 
\plot
         0    0.0093
    0.0484   -0.0725
    0.0968   -0.1226
    0.1452   -0.1586
    0.1935   -0.1814
    0.2419   -0.1931
    0.2903   -0.1954
    0.3387   -0.1902
    0.3871   -0.1787
    0.4355   -0.1616
    0.4839   -0.1397
    0.5322   -0.1136
    0.5806   -0.0842
    0.6290   -0.0529
    0.6774   -0.0214
    0.7258    0.0088
    0.7742    0.0364
    0.8226    0.0606
    0.8709    0.0805
    0.9193    0.0957
    0.9677    0.1055
    1.0161    0.1092
    1.0645    0.1054
    1.1129    0.0929
    1.1612    0.0711
    1.2096    0.0397
    1.2580   -0.0018
    1.3064   -0.0542
    1.3548   -0.1178
    1.4032   -0.1933
    1.4516   -0.2813
    1.4999   -0.3817
    1.5483   -0.4948
    1.5967   -0.6221
    1.6451   -0.7664 /
 }\relax
\Orange{\relax 
\plot
         0    0.0095
    0.0484   -0.0784
    0.0968   -0.1368
    0.1452   -0.1815
    0.1935   -0.2122
    0.2419   -0.2304
    0.2903   -0.2377
    0.3387   -0.2360
    0.3871   -0.2265
    0.4355   -0.2104
    0.4839   -0.1881
    0.5322   -0.1610
    0.5806   -0.1303
    0.6290   -0.0974
    0.6774   -0.0643
    0.7258   -0.0330
    0.7742   -0.0046
    0.8226    0.0195
    0.8709    0.0385
    0.9193    0.0516
    0.9677    0.0577
    1.0161    0.0558
    1.0645    0.0439
    1.1129    0.0203
    1.1612   -0.0161
    1.2096   -0.0657
    1.2580   -0.1296
    1.3064   -0.2089
    1.3548   -0.3045
    1.4032   -0.4177
    1.4516   -0.5500
    1.4999   -0.7029
    1.5483   -0.8757
    1.5967   -1.0433
    1.6451   -1.1848 /
 }\relax
\Red{\relax 
\plot
         0    0.0095
    0.0484   -0.0789
    0.0968   -0.1386
    0.1452   -0.1853
    0.1935   -0.2183
    0.2419   -0.2387
    0.2903   -0.2483
    0.3387   -0.2488
    0.3871   -0.2414
    0.4355   -0.2272
    0.4839   -0.2071
    0.5322   -0.1822
    0.5806   -0.1537
    0.6290   -0.1231
    0.6774   -0.0926
    0.7258   -0.0635
    0.7742   -0.0375
    0.8226   -0.0158
    0.8709    0.0008
    0.9193    0.0113
    0.9677    0.0145
    1.0161    0.0090
    1.0645   -0.0073
    1.1129   -0.0365
    1.1612   -0.0797
    1.2096   -0.1379
    1.2580   -0.2121
    1.3064   -0.3039
    1.3548   -0.4142
    1.4032   -0.5448
    1.4516   -0.6976
    1.4999   -0.8727
    1.5483   -1.0683
    1.5967   -1.2787
    1.6451   -1.4996
/
 }\relax
\LimeGreen{\relax 
\plot
    0.0484   -0.0789
    0.0968   -0.1386
    0.1452   -0.1854
    0.1935   -0.2185
    0.2419   -0.2390
    0.2903   -0.2488
    0.3387   -0.2494
    0.3871   -0.2422
    0.4355   -0.2283
    0.4839   -0.2086
    0.5322   -0.1841
    0.5806   -0.1561
    0.6290   -0.1261
    0.6774   -0.0960
    0.7258   -0.0674
    0.7742   -0.0418
    0.8226   -0.0205
    0.8709   -0.0042
    0.9193    0.0059
    0.9677    0.0088
    1.0161    0.0031
    1.0645   -0.0135
    1.1129   -0.0430
    1.1612   -0.0868
    1.2096   -0.1456
    1.2580   -0.2206
    1.3064   -0.3133
    1.3548   -0.4249
    1.4032   -0.5567
    1.4516   -0.7107
    1.4999   -0.8870
    %1.5483   -1.0863
    1.5967   -1.3104
    1.6451   -1.5620
/
 }\relax
\PineGreen{\relax 
\plot
         0    0.0095
    0.0484   -0.0789
    0.0968   -0.1386
    0.1452   -0.1854
    0.1935   -0.2185
    0.2419   -0.2390
    0.2903   -0.2488
    0.3387   -0.2494
    0.3871   -0.2423
    0.4355   -0.2284
    0.4839   -0.2087
    0.5322   -0.1842
    0.5806   -0.1563
    0.6290   -0.1263
    0.6774   -0.0963
    0.7258   -0.0678
    0.7742   -0.0423
    0.8226   -0.0209
    0.8709   -0.0048
    0.9193    0.0054
    0.9677    0.0083
    1.0161    0.0025
    1.0645   -0.0141
    1.1129   -0.0436
    1.1612   -0.0874
    1.2096   -0.1462
    1.2580   -0.2213
    1.3064   -0.3141
    1.3548   -0.4257
    1.4032   -0.5576
    1.4516   -0.7116
    1.4999   -0.8879
    1.5483   -1.0881
    1.5967   -1.3154
    1.6451   -1.5738
/
 }\relax
\Blue{\relax 
\plot
         0    0.0095
    0.0484   -0.0789
    0.0968   -0.1386
    0.1452   -0.1854
    0.1935   -0.2185
    0.2419   -0.2391
    0.2903   -0.2488
    0.3387   -0.2494
    0.3871   -0.2423
    0.4355   -0.2284
    0.4839   -0.2087
    0.5322   -0.1842
    0.5806   -0.1563
    0.6290   -0.1264
    0.6774   -0.0963
    0.7258   -0.0678
    0.7742   -0.0423
    0.8226   -0.0210
    0.8709   -0.0048
    0.9193    0.0053
    0.9677    0.0082
    1.0161    0.0025
    1.0645   -0.0141
    1.1129   -0.0436
    1.1612   -0.0874
    1.2096   -0.1463
    1.2580   -0.2213
    1.3064   -0.3142
    1.3548   -0.4258
    1.4032   -0.5577
    1.4516   -0.7115
    1.4999   -0.8880
    1.5483   -1.0885
    1.5967   -1.3166
    1.6451   -1.5761
/
 }\relax
\endpicture
}   % end of time trace data for f1:  $\beta=-1$, $T=1$
\xfigdim=0.40\xfiglen
\yfigdim=0.37\yfiglen
\setbox\figuretwo=\vbox{\hsize=\xfiglen
\beginpicture
\footnotesize
  \setcoordinatesystem units <\xfigdim,\yfigdim>  point at 0 -0.70
  \setplotarea x from 0.0 to 2.5, y from -0.7 to 2
  \axis bottom shiftedto y=-0.7 ticks short numbered from 0 to 2 by 0.5 /
  \axis left ticks short numbered from -0.5 to 2 by 0.5 /
 \put {{\sevenrm $f_2(u)$}} [lb] at 0 2
 \put {{\sevenrm $u$}} [rt] at 2.5 -0.74
\setquadratic
\setlinear
\setdashes <3pt>
%\Black{
%\input ./pics_Nov10/titr_files/f2_files/fex2_beta-1_T1_titr.tex
\plot
        0    0.0135
    0.03      0.02
    0.0736    0.0268
    0.1472    0.0505
    0.2208    0.0912
    0.2944    0.1572
    0.3680    0.2579
    0.4416    0.4011
    0.5152    0.5909
    0.5888    0.8240
    0.6624    1.0872
    0.7360    1.3572
    0.8096    1.6028
    0.8832    1.7901
    0.9568    1.8899
    1.0304    1.8846
    1.1040    1.7729
    1.1776    1.5697
    1.2512    1.3025
    1.3247    1.0049
    1.3983    0.7091
    1.4719    0.4401
    1.5455    0.2127
    1.6191    0.0320
    1.6927   -0.1050
    1.7663   -0.2059
    1.8399   -0.2798
    1.9135   -0.3353
    1.9871   -0.3796
    2.0607   -0.4174
    2.1343   -0.4523
    2.2079   -0.4861
    2.2815   -0.5200
    2.3551   -0.5544
    2.4287   -0.5898
    2.5023   -0.6261
/
%}
%
\setsolid
\Goldenrod{\relax 
\plot
         0    0.0137
    0.0736    0.0246
    0.1472    0.0433
    0.2208    0.0754
    0.2944    0.1269
    0.3680    0.2046
    0.4416    0.3153
    0.5152    0.4633
    0.5888    0.6439
    0.6624    0.8424
    0.7360    1.0387
    0.8096    1.2089
    0.8832    1.3317
    0.9568    1.3902
    1.0304    1.3765
    1.1040    1.2907
    1.1776    1.1417
    1.2511    0.9462
    1.3247    0.7267
    1.3983    0.5075
    1.4719    0.3077
    1.5455    0.1397
    1.6191    0.0069
    1.6927   -0.0980
    1.7663   -0.1828
    1.8399   -0.2508
    1.9135   -0.3054
    1.9871   -0.3506
    2.0607   -0.3895
    2.1343   -0.4241
    2.2079   -0.4564
    2.2815   -0.4877
    2.3551   -0.5191
    2.4287   -0.5517
    2.5023   -0.5866
/
 }\relax
\Orange{\relax 
\plot
         0    0.0137
    0.0736    0.0264
    0.1472    0.0491
    0.2208    0.0877
    0.2944    0.1500
    0.3680    0.2446
    0.4416    0.3805
    0.5152    0.5629
    0.5888    0.7865
    0.6624    1.0336
    0.7360    1.2799
    0.8096    1.4965
    0.8832    1.6580
    0.9568    1.7439
    1.0304    1.7437
    1.1040    1.6559
    1.1776    1.4884
    1.2511    1.2585
    1.3247    0.9915
    1.3983    0.7159
    1.4719    0.4573
    1.5455    0.2352
    1.6191    0.0581
    1.6927   -0.0796
    1.7663   -0.1876
    1.8399   -0.2717
    1.9135   -0.3377
    1.9871   -0.3911
    2.0607   -0.4368
    2.1343   -0.4779
    2.2079   -0.5167
    2.2815   -0.5551
    2.3551   -0.5938
    2.4287   -0.6343
    2.5023   -0.6781
/
 }\relax
\Red{\relax 
\plot
         0    0.0137
    0.0736    0.0266
    0.1472    0.0498
    0.2208    0.0897
    0.2944    0.1545
    0.3680    0.2538
    0.4416    0.3973
    0.5152    0.5914
    0.5888    0.8305
    0.6624    1.0954
    0.7360    1.3596
    0.8096    1.5918
    0.8832    1.7635
    0.9568    1.8536
    1.0304    1.8501
    1.1040    1.7516
    1.1776    1.5673
    1.2511    1.3163
    1.3247    1.0269
    1.3983    0.7313
    1.4719    0.4579
    1.5455    0.2271
    1.6191    0.0470
    1.6927   -0.0899
    1.7663   -0.1947
    1.8399   -0.2739
    1.9135   -0.3340
    1.9871   -0.3819
    2.0607   -0.4227
    2.1343   -0.4602
    2.2079   -0.4962
    2.2815   -0.5324
    2.3551   -0.5691
    2.4287   -0.6075
    2.5023   -0.6486
/
 }\relax
\LimeGreen{\relax 
\plot
         0    0.0137
    0.0736    0.0266
    0.1472    0.0498
    0.2208    0.0898
    0.2944    0.1546
    0.3680    0.2541
    0.4416    0.3980
    0.5152    0.5929
    0.5888    0.8331
    0.6624    1.0995
    0.7360    1.3651
    0.8096    1.5982
    0.8832    1.7704
    0.9568    1.8599
    1.0304    1.8551
    1.1040    1.7545
    1.1776    1.5677
    1.2511    1.3141
    1.3247    1.0226
    1.3983    0.7257
    1.4719    0.4521
    1.5455    0.2222
    1.6191    0.0438
    1.6927   -0.0908
    1.7663   -0.1932
    1.8399   -0.2698
    1.9135   -0.3276
    1.9871   -0.3732
    2.0607   -0.4121
    2.1343   -0.4478
    2.2079   -0.4823
    2.2815   -0.5168
    2.3551   -0.5520
    2.4287   -0.5884
    2.5023   -0.6270
/
 }\relax
\PineGreen{\relax 
%/\input ./pics_Nov10/titr_files/f2_files/f2_beta-1_T1_it8_titr.tex
\plot
         0    0.0137
    0.0736    0.0266
    0.1472    0.0498
    0.2208    0.0898
    0.2944    0.1546
    0.3680    0.2541
    0.4416    0.3980
    0.5152    0.5930
    0.5888    0.8333
    0.6624    1.0998
    0.7360    1.3655
    0.8096    1.5986
    0.8832    1.7708
    0.9568    1.8601
    1.0304    1.8553
    1.1040    1.7545
    1.1776    1.5674
    1.2511    1.3136
    1.3247    1.0218
    1.3983    0.7250
    1.4719    0.4514
    1.5455    0.2218
    1.6191    0.0437
    1.6927   -0.0906
    1.7663   -0.1925
    1.8399   -0.2687
    1.9135   -0.3259
    1.9871   -0.3711
    2.0607   -0.4095
    2.1343   -0.4448
    2.2079   -0.4789
    2.2815   -0.5131
    2.3551   -0.5479
    2.4287   -0.5836
    2.5023   -0.6213
/
 }\relax
\Blue{\relax 
\plot
         0    0.0137
    0.0736    0.0266
    0.1472    0.0498
    0.2208    0.0898
    0.2944    0.1546
    0.3680    0.2541
    0.4416    0.3980
    0.5152    0.5930
    0.5888    0.8333
    0.6624    1.0998
    0.7360    1.3655
    0.8096    1.5987
    0.8832    1.7708
    0.9568    1.8603
    1.0304    1.8553
    1.1040    1.7545
    1.1776    1.5674
    1.2511    1.3135
    1.3247    1.0219
    1.3983    0.7249
    1.4719    0.4514
    1.5455    0.2218
    1.6191    0.0437
    1.6927   -0.0905
    1.7663   -0.1923
    1.8399   -0.2685
    1.9135   -0.3256
    1.9871   -0.3707
    2.0607   -0.4090
    2.1343   -0.4442
    2.2079   -0.4783
    2.2815   -0.5124
    2.3551   -0.5470
    2.4287   -0.5825
    2.5023   -0.6200
/
 }\relax
\endpicture
}   % end of time trace data for f2:  $\beta=-1$, $T=1$
\xfigdim=0.334\xfiglen
\yfigdim=0.052\yfiglen
\setbox\figurethree=\vbox{\hsize=\xfiglen
% This is for: time trace data, beta=1, T=1, f_1(u)
\beginpicture
\footnotesize
  \setcoordinatesystem units <\xfigdim,\yfigdim>  point at 0 -19
  \setplotarea x from 0 to 3, y from -19 to 0
  \axis bottom shiftedto y=-19 ticks short numbered from 0 to 3 by 1 /
  \axis left ticks short numbered from -15 to 0 by 5 /
 \put {{\sevenrm $f_1(u)$}} [lb] at 0.02 0
 \put {{\sevenrm $u$}} [rt] at 3 -17.8
\setquadratic
\setdashes <3pt>
\Black{
\plot
         0         0
    0.0799   -0.1206
    0.1598   -0.1988
    0.2398   -0.2407
    0.3197   -0.2524
    0.3996   -0.2401
    0.4795   -0.2099
    0.5595   -0.1679
    0.6394   -0.1202
    0.7193   -0.0730
    0.7992   -0.0323
    0.8792   -0.0044
    0.9591    0.0046
    1.0390   -0.0113
    1.1189   -0.0583
    1.1989   -0.1425
    1.2788   -0.2701
    1.3587   -0.4471
    1.4386   -0.6798
    1.5186   -0.9742
    1.5985   -1.3364
    1.6784   -1.7726
    1.7583   -2.2890
    1.8383   -2.8915
    1.9182   -3.5865
    1.9981   -4.3799
    2.0780   -5.2779
    2.1579   -6.2867
    2.2379   -7.4124
    2.3178   -8.6610
    2.3977  -10.0388
    2.4776  -11.5518
    2.5576  -13.2062
    2.6375  -15.0080
    2.7174  -16.9636
/
}
\setsolid
\Goldenrod{\relax 
\plot
         0    0.0028
    0.0799   -0.1048
    0.1598   -0.1674
    0.2398   -0.1954
    0.3197   -0.1958
    0.3996   -0.1765
    0.4795   -0.1447
    0.5595   -0.1070
    0.6394   -0.0691
    0.7193   -0.0350
    0.7992   -0.0077
    0.8792    0.0097
    0.9591    0.0129
    1.0390   -0.0025
    1.1189   -0.0407
    1.1989   -0.1060
    1.2788   -0.2017
    1.3587   -0.3314
    1.4386   -0.4980
    1.5185   -0.7039
    1.5985   -0.9515
    1.6784   -1.2418
    1.7583   -1.5757
    1.8382   -1.9534
    1.9182   -2.3726
    1.9981   -2.8309
    2.0780   -3.3212
    2.1579   -3.8325
    2.2379   -4.3469
    2.3178   -4.8357
    2.3977   -5.2518
    2.4776   -5.5180
    2.5576   -5.4883
    2.6375   -4.9193
    2.7174   -3.7804
/
 }\relax
\Orange{\relax 
\plot
         0    0.0030
    0.0799   -0.1161
    0.1598   -0.1932
    0.2398   -0.2333
    0.3197   -0.2415
    0.3996   -0.2249
    0.4795   -0.1911
    0.5595   -0.1477
    0.6394   -0.1019
    0.7193   -0.0589
    0.7992   -0.0229
    0.8792    0.0012
    0.9591    0.0077
    1.0390   -0.0093
    1.1189   -0.0559
    1.1989   -0.1374
    1.2788   -0.2590
    1.3587   -0.4258
    1.4386   -0.6426
    1.5185   -0.9141
    1.5985   -1.2449
    1.6784   -1.6389
    1.7583   -2.1006
    1.8382   -2.6342
    1.9182   -3.2433
    1.9981   -3.9324
    2.0780   -4.7013
    2.1579   -5.5374
    2.2379   -6.3880
    2.3178   -7.1182
    2.3977   -7.5553
    2.4776   -7.8519
    2.5576   -8.1989
    2.6375   -8.3332
    2.7174   -8.1808
/
 }\relax
\Red{\relax 
\plot
         0    0.0030
    0.0799   -0.1173
    0.1598   -0.1978
    0.2398   -0.2418
    0.3197   -0.2537
    0.3996   -0.2398
    0.4795   -0.2072
    0.5595   -0.1633
    0.6394   -0.1152
    0.7193   -0.0690
    0.7992   -0.0294
    0.8792   -0.0020
    0.9591    0.0069
    1.0390   -0.0093
    1.1189   -0.0572
    1.1989   -0.1429
    1.2788   -0.2724
    1.3587   -0.4516
    1.4386   -0.6858
    1.5185   -0.9808
    1.5985   -1.3424
    1.6784   -1.7759
    1.7583   -2.2872
    1.8382   -2.8833
    1.9182   -3.5695
    1.9981   -4.3538
    2.0780   -5.2395
    2.1579   -6.2259
    2.2379   -7.3045
    2.3178   -8.4597
    2.3977   -9.6663
    2.4776  -10.8833
    2.5576  -12.0275
    2.6375  -12.9623
    2.7174  -13.7292
/
 }\relax
\LimeGreen{\relax 
\plot
         0    0.0030
    0.0799   -0.1173
    0.1598   -0.1979
    0.2398   -0.2421
    0.3197   -0.2543
    0.3996   -0.2407
    0.4795   -0.2082
    0.5595   -0.1642
    0.6394   -0.1162
    0.7193   -0.0697
    0.7992   -0.0299
    0.8792   -0.0021
    0.9591    0.0071
    1.0390   -0.0090
    1.1189   -0.0568
    1.1989   -0.1428
    1.2788   -0.2726
    1.3587   -0.4527
    1.4386   -0.6881
    1.5185   -0.9850
    1.5985   -1.3490
    1.6784   -1.7856
    1.7583   -2.3004
    1.8382   -2.9004
    1.9182   -3.5907
    1.9981   -4.3792
    2.0780   -5.2720
    2.1579   -6.2748
    2.2379   -7.3929
    2.3178   -8.6312
    2.3977   -9.9895
    2.4776  -11.4657
    2.5576  -13.0294
    2.6375  -14.6021
    2.7174  -16.1329
/
 }\relax
\PineGreen{\relax 
\plot
         0    0.0030
    0.0799   -0.1173
    0.1598   -0.1979
    0.2398   -0.2421
    0.3197   -0.2543
    0.3996   -0.2407
    0.4795   -0.2083
    0.5595   -0.1644
    0.6394   -0.1162
    0.7193   -0.0698
    0.7992   -0.0299
    0.8792   -0.0021
    0.9591    0.0071
    1.0390   -0.0089
    1.1189   -0.0568
    1.1989   -0.1427
    1.2788   -0.2726
    1.3587   -0.4528
    1.4386   -0.6883
    1.5185   -0.9853
    1.5985   -1.3496
    1.6784   -1.7863
    1.7583   -2.3013
    1.8382   -2.9014
    1.9182   -3.5916
    1.9981   -4.3802
    2.0780   -5.2735
    2.1579   -6.2782
    2.2379   -7.4009
    2.3178   -8.6489
    2.3977  -10.0254
    2.4776  -11.5362
    2.5576  -13.1760
    2.6375  -14.9394
    2.7174  -16.8605
/
 }\relax
\Blue{\relax 
\plot
         0    0.0030
    0.0799   -0.1173
    0.1598   -0.1979
    0.2398   -0.2421
    0.3197   -0.2543
    0.3996   -0.2407
    0.4795   -0.2083
    0.5595   -0.1644
    0.6394   -0.1162
    0.7193   -0.0698
    0.7992   -0.0299
    0.8792   -0.0021
    0.9591    0.0071
    1.0390   -0.0089
    1.1189   -0.0568
    1.1989   -0.1427
    1.2788   -0.2726
    1.3587   -0.4528
    1.4386   -0.6883
    1.5185   -0.9854
    1.5985   -1.3496
    1.6784   -1.7863
    1.7583   -2.3013
    1.8382   -2.9014
    1.9182   -3.5916
    1.9981   -4.3802
    2.0780   -5.2736
    2.1579   -6.2785
    2.2379   -7.4018
    2.3178   -8.6507
    2.3977  -10.0289
    2.4776  -11.5435
    2.5576  -13.1955
    2.6375  -14.9947
    2.7174  -16.9871
/
 }\relax
\endpicture
}   % end of time trace data for f1:  $\beta=0.3$, $T=1$
\xfigdim=0.189\xfiglen
\yfigdim=0.222\yfiglen
\setbox\figurefour=\vbox{\hsize=\xfiglen
\beginpicture
\footnotesize
  \setcoordinatesystem units <\xfigdim,\yfigdim>  point at 0 -2.5
  \setplotarea x from 0.0 to 5.3, y from -2.5 to 2
  \axis bottom shiftedto y=-2.5 ticks short numbered from 1 to 5 by 1 /
  \axis left ticks short numbered from -2 to 2 by 1 /
 \put {{\sevenrm $f_2(u)$}} [lb] at 0.1 2
 \put {{\sevenrm $u$}} [rb] at 5.3 -2.4
\setquadratic
\setdashes <3pt>
%\Black{
%\input ./pics_Nov10/titr_files/f2_files/fex2_beta1_T1_titr.tex
\plot
         0    0.0135
    0.1455    0.0498
    0.2911    0.1536
    0.4366    0.3901
    0.5822    0.8016
    0.7277    1.3276
    0.8733    1.7694
    1.0188    1.8927
    1.1644    1.6117
    1.3099    1.0657
    1.4555    0.4970
    1.6010    0.0723
    1.7466   -0.1818
    1.8921   -0.3206
    2.0376   -0.4060
    2.1832   -0.4748
    2.3287   -0.5420
    2.4743   -0.6122
    2.6198   -0.6863
    2.7654   -0.7647
    2.9109   -0.8473
    3.0565   -0.9342
    3.2020   -1.0253
    3.3476   -1.1206
    3.4931   -1.2202
    3.6386   -1.3240
    3.7842   -1.4320
    3.9297   -1.5443
    4.0753   -1.6608
    4.2208   -1.7815
    4.3664   -1.9065
    4.5119   -2.0000
    4.6575   -2.0000
    4.8030   -2.0000
    4.9486   -2.0000
/
%}
%
\setsolid
\Goldenrod{\relax 
\plot
         0    0.0134
    0.1453    0.0424
    0.2905    0.1227
    0.4358    0.3095
    0.5810    0.6370
    0.7263    1.0309
    0.8715    1.3284
    1.0168    1.3812
    1.1620    1.1609
    1.3073    0.7554
    1.4525    0.3134
    1.5978   -0.0341
    1.7430   -0.2586
    1.8883   -0.3975
    2.0335   -0.4930
    2.1787   -0.5714
    2.3240   -0.6400
    2.4693   -0.6969
    2.6145   -0.7416
    2.7598   -0.7716
    2.9050   -0.7809
    3.0503   -0.7688
    3.1955   -0.7330
    3.3408   -0.6674
    3.4860   -0.5702
    3.6313   -0.4385
    3.7765   -0.2652
    3.9218   -0.0444
    4.0670    0.2360
    4.2123    0.5942
    4.3575    1.0483
    4.5027    1.6179
    4.5370    1.7705
    4.5772    1.9507
    4.6173    2.1420
    % 4.6480    2.2924
    %4.7933    3.0237
    %4.9385    3.7640
/
 }\relax
\Orange{\relax 
\plot
         0    0.0133
    0.1453    0.0478
    0.2905    0.1434
    0.4358    0.3695
    0.5810    0.7716
    0.7263    1.2615
    0.8715    1.6424
    1.0168    1.7371
    1.1620    1.5049
    1.3073    1.0378
    1.4525    0.5053
    1.5978    0.0739
    1.7430   -0.2037
    1.8883   -0.3669
    2.0335   -0.4704
    2.1787   -0.5526
    2.3240   -0.6281
    2.4693   -0.6984
    2.6145   -0.7645
    2.7598   -0.8248
    2.9050   -0.8741
    3.0503   -0.9122
    3.1955   -0.9378
    3.3408   -0.9430
    3.4860   -0.9229
    3.6313   -0.8688
    3.7765   -0.7619
    3.9218   -0.5844
    4.0670   -0.3166
    4.2123    0.0597
    4.3575    0.5369
    4.5027    1.0799
    4.6480    1.6288
    %4.7933    2.1440
    %4.9385    2.6308
/
 }\relax
\Red{\relax 
\plot
         0    0.0133
    0.1453    0.0484
    0.2905    0.1473
    0.4358    0.3839
    0.5810    0.8108
    0.7263    1.3338
    0.8715    1.7390
    1.0168    1.8378
    1.1620    1.5894
    1.3073    1.0948
    1.4525    0.5380
    1.5978    0.0965
    1.7430   -0.1773
    1.8883   -0.3289
    2.0335   -0.4194
    2.1787   -0.4904
    2.3240   -0.5591
    2.4693   -0.6289
    2.6145   -0.7023
    2.7598   -0.7793
    2.9050   -0.8567
    3.0503   -0.9358
    3.1955   -1.0170
    3.3408   -1.0949
    3.4860   -1.1681
    3.6313   -1.2345
    3.7765   -1.2830
    3.9218   -1.3049
    4.0670   -1.2935
    4.2123   -1.2347
    4.3575   -1.1218
    4.5027   -0.9598
    4.6480   -0.7632
    4.7933   -0.5602
    4.9385   -0.3729
/
 }\relax
\LimeGreen{\relax 
\plot
         0    0.0133
    0.1453    0.0484
    0.2905    0.1474
    0.4358    0.3846
    0.5810    0.8130
    0.7263    1.3382
    0.8715    1.7447
    1.0168    1.8426
    1.1620    1.5916
    1.3073    1.0944
    1.4525    0.5365
    1.5978    0.0962
    1.7430   -0.1749
    1.8883   -0.3232
    2.0335   -0.4103
    2.1787   -0.4786
    2.3240   -0.5453
    2.4693   -0.6142
    2.6145   -0.6881
    2.7598   -0.7672
    2.9050   -0.8486
    3.0503   -0.9339
    3.1955   -1.0243
    3.3408   -1.1163
    3.4860   -1.2108
    3.6313   -1.3104
    3.7765   -1.4095
    3.9218   -1.5056
    4.0670   -1.5976
    4.2123   -1.6723
    4.3575   -1.7170
    4.5027   -1.7287
    4.6480   -1.7085
    4.7933   -1.6721
    4.9385   -1.6428
/
 }\relax
\PineGreen{\relax 
\plot
         0    0.0133
    0.1453    0.0484
    0.2905    0.1474
    0.4358    0.3846
    0.5810    0.8132
    0.7263    1.3385
    0.8715    1.7450
    1.0168    1.8428
    1.1620    1.5916
    1.3073    1.0942
    1.4525    0.5362
    1.5978    0.0961
    1.7430   -0.1747
    1.8883   -0.3225
    2.0335   -0.4092
    2.1787   -0.4771
    2.3240   -0.5434
    2.4693   -0.6121
    2.6145   -0.6861
    2.7598   -0.7653
    2.9050   -0.8469
    3.0503   -0.9329
    3.1955   -1.0245
    3.3408   -1.1185
    3.4860   -1.2166
    3.6313   -1.3223
    3.7765   -1.4316
    3.9218   -1.5431
    4.0670   -1.6576
    4.2123   -1.7628
    4.3575   -1.8464
    4.5027   -1.9047
    4.6480   -1.9369
    4.7933   -1.9550
    4.9385   -1.9805
/
 }\relax
\Blue{\relax 
\plot
         0    0.0133
    0.1453    0.0484
    0.2905    0.1474
    0.4358    0.3846
    0.5810    0.8132
    0.7263    1.3385
    0.8715    1.7450
    1.0168    1.8428
    1.1620    1.5916
    1.3073    1.0942
    1.4525    0.5362
    1.5978    0.0960
    1.7430   -0.1747
    1.8883   -0.3225
    2.0335   -0.4091
    2.1787   -0.4769
    2.3240   -0.5432
    2.4693   -0.6119
    2.6145   -0.6858
    2.7598   -0.7650
    2.9050   -0.8467
    3.0503   -0.9326
    3.1955   -1.0243
    3.3408   -1.1186
    3.4860   -1.2172
    3.6313   -1.3239
    3.7765   -1.4348
    3.9218   -1.5491
    4.0670   -1.6677
    4.2123   -1.7789
    4.3575   -1.8704
    4.5027   -1.9388
    4.6480   -1.9825
    4.7933   -2.0131
    4.9385   -2.0512
/
}\relax
\endpicture
}   % end of time trace data for f2:  $\beta=0.3$, $T=1$
\xfigdim=0.312\xfiglen
\yfigdim=0.033\yfiglen
\setbox\figurefive=\vbox{\hsize=\xfiglen
% This is for: time trace data, beta=1.3, T=1, f_1(u)
\beginpicture
\footnotesize
  \setcoordinatesystem units <\xfigdim,\yfigdim>  point at 0 -30
  \setplotarea x from 0 to 3.2, y from -30 to 0
  \axis bottom shiftedto y=-30 ticks short numbered from 0 to 3 by 1 /
  \axis left ticks short numbered from -30 to 0 by 5 /
 \put {{\sevenrm $f_1(u)$}} [lb] at 0.02 0
 \put {{\sevenrm $u$}} [rt] at 3.2 -28.4
\setquadratic
\setdashes <3pt>
\Black{
\plot
         0         0
    0.0895   -0.1322
    0.1791   -0.2120
    0.2686   -0.2481
    0.3582   -0.2491
    0.4477   -0.2237
    0.5373   -0.1803
    0.6268   -0.1278
    0.7164   -0.0746
    0.8059   -0.0294
    0.8955   -0.0008
    0.9850    0.0025
    1.0746   -0.0280
    1.1641   -0.1009
    1.2537   -0.2250
    1.3432   -0.4087
    1.4328   -0.6608
    1.5223   -0.9897
    1.6119   -1.4042
    1.7014   -1.9129
    1.7910   -2.5244
    1.8805   -3.2473
    1.9701   -4.0901
    2.0596   -5.0617
    2.1492   -6.1704
    2.2387   -7.4251
    2.3283   -8.8342
    2.4178  -10.4064
    2.5074  -12.1503
    2.5969  -14.0746
    2.6865  -16.1878
    2.7760  -18.4986
    2.8656  -21.0156
    2.9551  -23.7474
    3.0447  -26.7026
/
}
\setsolid
\Goldenrod{\relax 
\plot
        0    0.0011
    0.0896   -0.1139
    0.1791   -0.1769
    0.2686   -0.1986
    0.3582   -0.1889
    0.4478   -0.1589
    0.5373   -0.1183
    0.6269   -0.0763
    0.7164   -0.0394
    0.8060   -0.0111
    0.8955    0.0039
    0.9851    0.0002
    1.0746   -0.0290
    1.1642   -0.0898
    1.2537   -0.1882
    1.3433   -0.3294
    1.4328   -0.5178
    1.5224   -0.7576
    1.6119   -1.0523
    1.7015   -1.4040
    1.7910   -1.8146
    1.8806   -2.2848
    1.9701   -2.8134
    2.0597   -3.3963
    2.1492   -4.0267
    2.2388   -4.6911
    2.3283   -5.3663
    2.4179   -6.0156
    2.5074   -6.5776
    2.5970   -6.9559
    2.6865   -6.9800
    2.7761   -6.2855
    2.8656   -3.9768
%    2.9552    1.4306
%    3.0447    9.7614
/
 }\relax
\Orange{\relax 
\plot
         0    0.0014
    0.0896   -0.1284
    0.1791   -0.2107
    0.2686   -0.2483
    0.3582   -0.2478
    0.4478   -0.2192
    0.5373   -0.1726
    0.6269   -0.1194
    0.7164   -0.0683
    0.8060   -0.0255
    0.8955    0.0014
    0.9851    0.0040
    1.0746   -0.0272
    1.1642   -0.1014
    1.2537   -0.2271
    1.3433   -0.4126
    1.4328   -0.6653
    1.5224   -0.9928
    1.6119   -1.4024
    1.7015   -1.9008
    1.7910   -2.4958
    1.8806   -3.1972
    1.9701   -4.0176
    2.0597   -4.9722
    2.1492   -6.0779
    2.2388   -7.3303
    2.3283   -8.6497
    2.4179   -9.7829
    2.5074  -10.2069
    2.5970   -9.2325
    2.6865   -7.0876
    2.7761   -5.7532
    2.8656   -3.5751
%    2.9552    2.5558
%    3.0447   12.6771
/
 }\relax
\Red{\relax 
\plot
         0    0.0014
    0.0896   -0.1286
    0.1791   -0.2117
    0.2686   -0.2504
    0.3582   -0.2511
    0.4478   -0.2232
    0.5373   -0.1768
    0.6269   -0.1229
    0.7164   -0.0704
    0.8060   -0.0261
    0.8955    0.0026
    0.9851    0.0067
    1.0746   -0.0237
    1.1642   -0.0976
    1.2537   -0.2242
    1.3433   -0.4121
    1.4328   -0.6688
    1.5224   -1.0027
    1.6119   -1.4214
    1.7015   -1.9320
    1.7910   -2.5429
    1.8806   -3.2626
    1.9701   -4.1011
    2.0597   -5.0690
    2.1492   -6.1815
    2.2388   -7.4537
    2.3283   -8.8905
    2.4179  -10.4592
    2.5074  -11.9992
    2.5970  -13.1217
    2.6865  -13.3197
    2.7761  -12.6612
    2.8656  -12.0981
    2.9552  -11.2053
    3.0447   -9.6368
/
 }\relax
\LimeGreen{\relax 
\plot
         0    0.0014
    0.0896   -0.1286
    0.1791   -0.2117
    0.2686   -0.2504
    0.3582   -0.2511
    0.4478   -0.2233
    0.5373   -0.1769
    0.6269   -0.1229
    0.7164   -0.0704
    0.8060   -0.0260
    0.8955    0.0027
    0.9851    0.0068
    1.0746   -0.0235
    1.1642   -0.0975
    1.2537   -0.2240
    1.3433   -0.4119
    1.4328   -0.6688
    1.5224   -1.0026
    1.6119   -1.4217
    1.7015   -1.9330
    1.7910   -2.5447
    1.8806   -3.2654
    1.9701   -4.1040
    2.0597   -5.0691
    2.1492   -6.1720
    2.2388   -7.4230
    2.3283   -8.8305
    2.4179  -10.4025
    2.5074  -12.1329
    2.5970  -14.0043
    2.6865  -15.9691
    2.7761  -17.9590
    2.8656  -19.8738
    2.9552  -21.5504
    3.0447  -23.0640
/
 }\relax
\PineGreen{\relax 
\plot
         0    0.0014
    0.0896   -0.1286
    0.1791   -0.2117
    0.2686   -0.2504
    0.3582   -0.2511
    0.4478   -0.2233
    0.5373   -0.1769
    0.6269   -0.1229
    0.7164   -0.0704
    0.8060   -0.0260
    0.8955    0.0027
    0.9851    0.0068
    1.0746   -0.0235
    1.1642   -0.0975
    1.2537   -0.2240
    1.3433   -0.4118
    1.4328   -0.6687
    1.5224   -1.0026
    1.6119   -1.4218
    1.7015   -1.9332
    1.7910   -2.5451
    1.8806   -3.2658
    1.9701   -4.1043
    2.0597   -5.0688
    2.1492   -6.1702
    2.2388   -7.4181
    2.3283   -8.8216
    2.4179  -10.3925
    2.5074  -12.1369
    2.5970  -14.0645
    2.6865  -16.1750
    2.7761  -18.4687
    2.8656  -20.9525
    2.9552  -23.6160
    3.0447  -26.4857
/
 }\relax
%
%\Blue{\relax 
%\input ./pics_Nov10/titr_files/f1_files/f1_beta1dot3_T1_it10_titr.tex
% }\relax
%
\endpicture
}   % end of time trace data for f1:  $\beta=1.3$, $T=1$
\xfigdim=0.154\xfiglen
\yfigdim=0.222\yfiglen
\setbox\figuresix=\vbox{\hsize=\xfiglen
\beginpicture
\footnotesize
  \setcoordinatesystem units <\xfigdim,\yfigdim>  point at 0 -2.5
  \setplotarea x from 0.0 to 6.5, y from -2.5 to 2
  \axis bottom shiftedto y=-2.5 ticks short numbered from 0 to 6 by 1 /
  \axis left ticks short numbered from -2 to 2 by 1 /
 \put {{\sevenrm $f_2(u)$}} [lb] at 0.1 2
 \put {{\sevenrm $u$}} [rb] at 6.5 -2.35
\setquadratic
\setdashes <3pt>
%\Black{
%\input ./pics_Nov10/titr_files/f2_files/fex2_beta1dot3_T1_titr.tex
\plot
         0    0.0135
    0.1815    0.0669
    0.3631    0.2500
    0.5446    0.6796
    0.7262    1.3220
    0.9077    1.8343
    1.0893    1.8032
    1.2708    1.2244
    1.4524    0.5079
    1.6339    0.0012
    1.8155   -0.2577
    1.9970   -0.3849
    2.1786   -0.4727
    2.3601   -0.5568
    2.5417   -0.6460
    2.7232   -0.7416
    2.9048   -0.8438
    3.0863   -0.9525
    3.2679   -1.0679
    3.4494   -1.1898
    3.6310   -1.3184
    3.8125   -1.4535
    3.9941   -1.5953
    4.1756   -1.7436
    4.3572   -1.8985
    4.5387   -2.0000
    4.7203   -2.0000
    4.9018   -2.0000
    5.0834   -2.0000
    5.2649   -2.0000
    5.4464   -2.0000
    5.6280   -2.0000
    5.8095   -2.0000
    5.9911   -2.0000
    6.1726   -2.0000
/
%}
%
\setsolid
\Goldenrod{\relax 
\plot
         0    0.0136
    0.1804    0.0545
    0.3607    0.1952
    0.5411    0.5444
  0.7215   1.0256
  0.7816   1.1860
  0.8417   1.2820
  0.9018   1.3428
  0.9620   1.3970
  1.0221   1.3963
  1.0822   1.3082
  1.1423   1.1842
  1.2025   1.0595
    1.2626    0.9006
    1.4430    0.3368
    1.6233   -0.0947
    1.8037   -0.3472
    1.9841   -0.4866
    2.1644   -0.5822
    2.3448   -0.6817
    2.5252   -0.7596
    2.7056   -0.7940
    2.8859   -0.8200
    3.0663   -0.8249
    3.2467   -0.7614
    3.4270   -0.6498
    3.6074   -0.5029
    3.7878   -0.2817
    3.9682    0.0225
    4.1485    0.4241
    4.3289    0.9549
    4.5093    1.6159
    4.6294   1.9220
 %  4.6896    2.4970
 %   4.8700    3.7135
 %   5.0504    5.1645
 %   5.2307    6.9280
 %   5.4111    9.3986
 %   5.5915   12.3431
 %   5.7719   15.2802
 %   5.9522   18.7643
 %   6.1326   22.5747
/
 }\relax
\setquadratic
\Orange{\relax 
\plot
         0    0.0135
    0.1804    0.0630
    0.3607    0.2368
    0.4209    0.3435
    0.5411    0.6817
    0.7215    1.3095
  0.7816   1.5217
  0.8417   1.6531
  0.9018   1.7409
  0.9620   1.8231
  1.0221   1.8382
  1.0822   1.7418
  1.1423   1.5986
  1.2025   1.4559
    1.2626    1.2696
    1.4430    0.5732
    1.6233    0.0329
    1.8037   -0.2664
    1.9841   -0.4130
    2.1644   -0.5043
    2.3448   -0.6041
    2.5252   -0.6995
    2.7056   -0.7746
    2.8859   -0.8677
    3.0663   -0.9754
    3.2467   -1.0554
    3.4270   -1.1440
  3.5473  -1.2439
  3.6074  -1.2718
  %3.6675  -1.2769
  3.7276  -1.3008
  3.7878  -1.3389
  %3.8479  -1.3398
  3.9080  -1.2947
  3.9681  -1.2467
  4.0282  -1.1970
  4.0884  -1.0900
    4.1485   -0.9109
    4.3289   -0.1117
    4.5093    1.3134
    4.5694   1.9815
   % 4.6896    3.4172
   % 4.8700    6.1890
   % 5.0504    9.1401
   % 5.2307   12.0407
   % 5.4111   15.4632
   % 5.5915   19.1378
   % 5.7719   22.5124
   % 5.9522   26.5826
   % 6.1326   31.1171
/
 }\relax
\setlinear
\Red{\relax 
\plot
         0    0.0135
    0.1804    0.0632
    0.3607    0.2381
    0.5411    0.6892
    0.7215    1.3280
  0.7816   1.5436
  0.8417   1.6766
  0.9018   1.7650
  0.9620   1.8473
  1.0221   1.8611
  1.0822   1.7618
  1.1423   1.6152
  1.2025   1.4692
    1.2626    1.2794
    1.4430    0.5761
    1.6233    0.0384
    1.8037   -0.2507
    1.9841   -0.3859
    2.1644   -0.4674
    2.3448   -0.5587
    2.5252   -0.6502
    2.7056   -0.7285
    2.8859   -0.8311
    3.0663   -0.9565
    3.2467   -1.0642
    3.4270   -1.1925
    3.6074   -1.3834
    3.7878   -1.5579
    3.9682   -1.6625
    4.1485   -1.7098
    4.3289   -1.5612
    4.5093   -1.0473
    4.6896   -0.1676
    4.8700    1.0732
    4.9301   1.5298
    4.9902   2.0034
  %  5.0503   2.5614
  %  5.0504    2.5621
    %5.2307    4.1424
    %5.4111    5.8620
    %5.5915    7.4787
    %5.7719    8.7356
    %5.9522   10.0780
    %6.1326   11.5085
    /
 }\relax
\setquadratic
\LimeGreen{\relax 
\plot
         0    0.0135
    0.1804    0.0632
    0.3607    0.2381
    0.5411    0.6893
    0.7215    1.3283
  0.7816   1.5439
  0.8417   1.6770
  0.9018   1.7653
  0.9620   1.8476
  1.0221   1.8613
  1.0822   1.7620
  1.1423   1.6152
  1.2025   1.4692
    1.2626    1.2793
    1.4430    0.5758
    1.6233    0.0385
    1.8037   -0.2502
    1.9841   -0.3847
    2.1644   -0.4655
    2.3448   -0.5562
    2.5252   -0.6476
    2.7056   -0.7268
    2.8859   -0.8309
    3.0663   -0.9563
    3.2467   -1.0585
    3.4270   -1.1715
    3.6074   -1.3359
    3.7878   -1.4860
    3.9682   -1.6009
    4.1485   -1.7475
    4.3289   -1.8668
    4.5093   -1.8548
    4.6896   -1.7804
    4.8700   -1.6818
    5.0504   -1.4750
    5.2307   -1.1921
    5.4111   -0.9131
    5.5915   -0.6148
    5.7719   -0.3005
    5.9522   -0.0159
    6.1326    0.2281
/
 }\relax
\PineGreen{\relax 
\plot
         0    0.0135
    0.1804    0.0632
    0.3607    0.2381
    0.5411    0.6893
  0.7215   1.3284
  0.7816   1.5439
  0.8417   1.6770
  0.9018   1.7653
  0.9620   1.8476
  1.0221   1.8613
  1.0822   1.7620
  1.1423   1.6152
  1.2025   1.4692
  1.2626   1.2793
    1.4430    0.5758
    1.6233    0.0385
    1.8037   -0.2501
    1.9841   -0.3847
    2.1644   -0.4654
    2.3448   -0.5561
    2.5252   -0.6476
    2.7056   -0.7271
    2.8859   -0.8315
    3.0663   -0.9566
    3.2467   -1.0571
    3.4270   -1.1662
    3.6074   -1.3246
    3.7878   -1.4687
    3.9682   -1.5821
    4.1485   -1.7379
    4.3289   -1.8865
    4.5093   -1.9332
    4.6896   -1.9569
    4.8700   -2.0062
    5.0504   -1.9834
    5.2307   -1.9153
    5.4111   -1.9074
    5.5915   -1.9002
    5.7719   -1.8426
    5.9522   -1.8427
    6.1326   -1.9037 /
 }\relax
%
%\Blue{\relax 
%\input ./pics_Nov10/titr_files/f2_files/f2_beta1dot3_T1_it12_titr.tex
% }\relax
%
\endpicture
}   % end of time trace data for f2:  $\beta=1.3$, $T=1$

\begin{figure}[h]
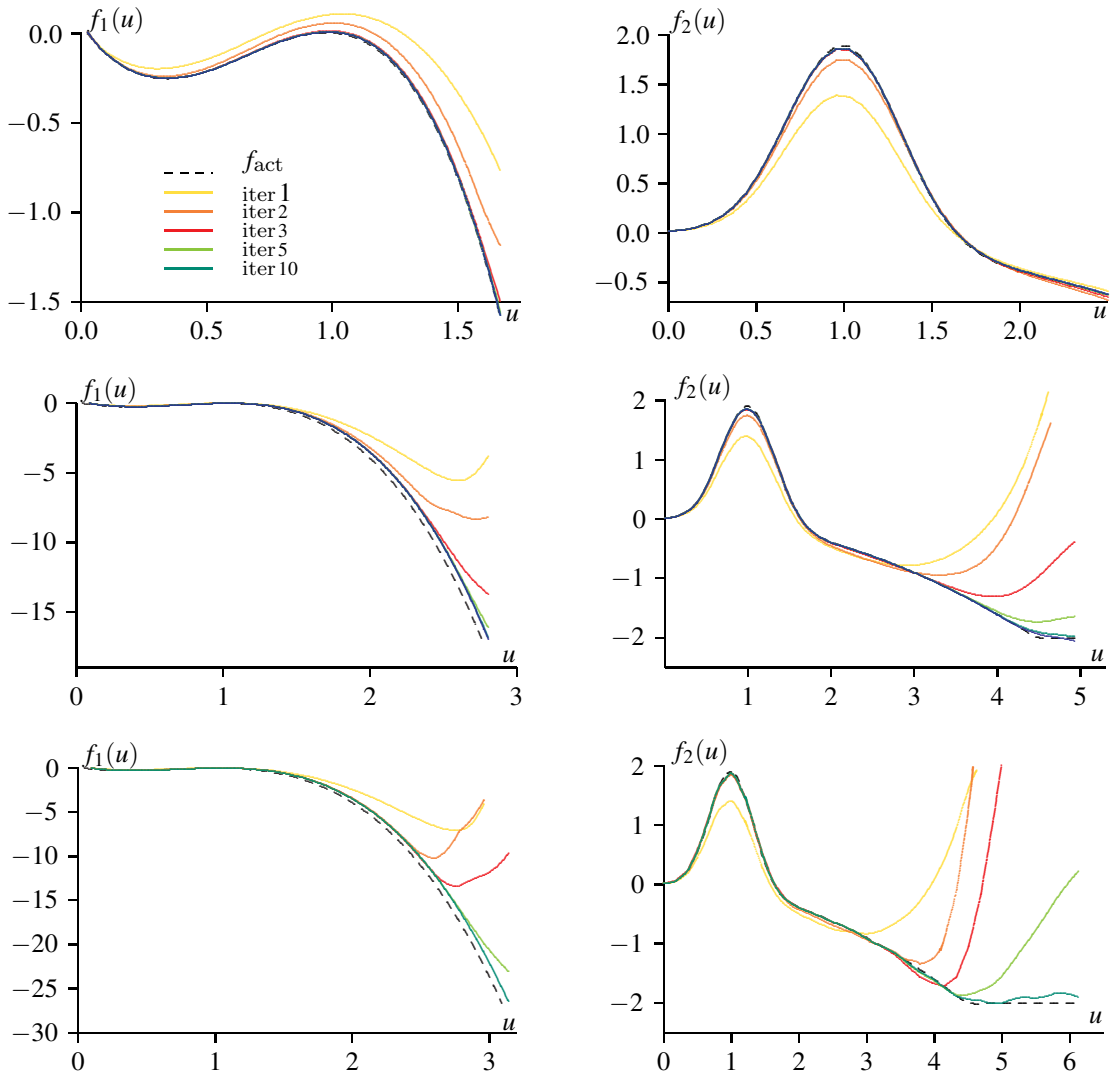

\hbox to\hsize{\hss\copy\figureone\hss\hss\hss\copy\figuretwo\hss}
\bigskip
\hbox to\hsize{\hss\copy\figurethree\hss\hss\hss\copy\figurefour\hss}
\bigskip
\hbox to\hsize{\hss\copy\figurefive\hss\hss\hss\copy\figuresix\hss}
%\centerline{\small {\bf Plots of iterations:}
%Time trace data with $T=1$ and $\beta=-1$ (top), $\beta=1$ (middle),
%$\beta=1.3$ (bottom)}
\caption{\small
Time trace data.\quad
Reconstructions of $f_1$ and $f_2$ at selected iterations as a function of
$\beta$:
$\beta=-1$ (top), $\beta=1$ (middle), $\beta=1.3$ (bottom).
$T =1$.
}
\label{fig:titr_reconstructions_f_with_beta}
\end{figure}
In these figures we show the exact function as a black dashed line
and the iterations in bold lines with the ordering:
yellow, orange, red, light green and dark green.
Typically these correspond to the second, fourth, sixth, eighth and tenth
but in he case of the slower converging scheme with $\beta=1.3$
these are at $1,\;3,\;6,\;9,\;12$.
Note for the larger $\beta$ values the reconstruction progresses
by improvements of the values of smaller magnitude first
due to the causality inherent in the time trace situation as opposed
to one from final time measurements.

%  \PiCTeX file for BB5   Final time recontructions
\input colordvi

\input pictex
\font\smallsymbol = cmmi8
%\newdimen\xfiglen \newdimen\yfiglen
\newdimen\xfigdim \newdimen\yfigdim
\newdimen\xmindim \newdimen\ymindim
\newdimen\xmaxdim \newdimen\ymaxdim
\newbox\figurelegendone
\newbox\figureseven
\newbox\figureeight
\newbox\figurenine
\newbox\figureten
\newbox\figureeleven
\newbox\figuretwelve
%%%%%%%%%%%%%%%%%%%
%\xfiglen=2.5true in
%\yfiglen=1.67true in
%%%%%%%%%%%%%%%%%%%%
\xfigdim=0.571\xfiglen
\yfigdim=0.67\yfiglen
\setbox\figureseven=\vbox{\hsize=\xfiglen
% This is for: final time data, beta=-1, T=1, f_1(u)
\beginpicture
\footnotesize
  \setcoordinatesystem units <\xfigdim,\yfigdim>  point at 0 -1.5
  \setplotarea x from 0 to 1.75, y from -1.5 to 0
  \axis bottom shiftedto y=-1.5 ticks short numbered from 0 to 1.5 by 0.5 /
% unlabeled short quantity 8 / /
  \axis left ticks short numbered from -1.5 to 0 by 0.5 /
% unlabeled short quantity 16 / /
 \put {{\sevenrm $f_1(u)$}} [lb] at 0.02 0
 \put {{\sevenrm $u$}} [rt] at 1.75 -1.54
\put {\copy\figurelegendtwo} [l] at 0.3 -1
%\setquadratic
\setdashes <3pt>
\Black{
\plot
         0         0
    0.0478   -0.0777
    0.0957   -0.1392
    0.1435   -0.1860
    0.1913   -0.2193
    0.2392   -0.2405
    0.2870   -0.2509
    0.3348   -0.2517
    0.3827   -0.2444
    0.4305   -0.2302
    0.4784   -0.2105
    0.5262   -0.1864
    0.5740   -0.1594
    0.6219   -0.1308
    0.6697   -0.1019
    0.7175   -0.0740
    0.7654   -0.0484
    0.8132   -0.0264
    0.8610   -0.0093
    0.9089    0.0014
    0.9567    0.0047
    1.0045   -0.0010
    1.0524   -0.0168
    1.1002   -0.0442
    1.1480   -0.0844
    1.1959   -0.1386
    1.2437   -0.2083
    1.2916   -0.2948
    1.3394   -0.3994
    1.3872   -0.5234
    1.4351   -0.6681
    1.4829   -0.8348
    1.5307   -1.0249
    1.5786   -1.2396
    1.6264   -1.4803
/
 }\relax
\setsolid
\Goldenrod{\relax 
\plot
         0   -0.0025
    0.0478   -0.0636
    0.0957   -0.1163
    0.1435   -0.1588
    0.1913   -0.1900
    0.2392   -0.2095
    0.2870   -0.2176
    0.3348   -0.2151
    0.3827   -0.2034
    0.4305   -0.1838
    0.4784   -0.1579
    0.5262   -0.1271
    0.5740   -0.0927
    0.6219   -0.0563
    0.6697   -0.0191
    0.7175    0.0176
    0.7654    0.0524
    0.8132    0.0841
    0.8610    0.1116
    0.9089    0.1338
    0.9567    0.1497
    1.0045    0.1581
    1.0524    0.1575
    1.1002    0.1466
    1.1480    0.1240
    1.1959    0.0890
    1.2437    0.0406
    1.2916   -0.0227
    1.3394   -0.1024
    1.3872   -0.1999
    1.4351   -0.3158
    1.4829   -0.4505
    1.5307   -0.6069
    1.5786   -0.7848
    1.6264   -0.9862
/
 }\relax
\Orange{\relax 
\plot
         0   -0.0026
    0.0478   -0.0673
    0.0957   -0.1236
    0.1435   -0.1699
    0.1913   -0.2049
    0.2392   -0.2280
    0.2870   -0.2399
    0.3348   -0.2414
    0.3827   -0.2338
    0.4305   -0.2185
    0.4784   -0.1971
    0.5262   -0.1711
    0.5740   -0.1420
    0.6219   -0.1112
    0.6697   -0.0801
    0.7175   -0.0502
    0.7654   -0.0227
    0.8132    0.0009
    0.8610    0.0195
    0.9089    0.0318
    0.9567    0.0370
    1.0045    0.0334
    1.0524    0.0198
    1.1002   -0.0055
    1.1480   -0.0437
    1.1959   -0.0958
    1.2437   -0.1630
    1.2916   -0.2464
    1.3394   -0.3480
    1.3872   -0.4690
    1.4351   -0.6096
    1.4829   -0.7700
    1.5307   -0.9511
    1.5786   -1.1475
    1.6264   -1.3467
/
 }\relax
\Red{\relax 
\plot
         0   -0.0026
    0.0478   -0.0688
    0.0957   -0.1266
    0.1435   -0.1743
    0.1913   -0.2106
    0.2392   -0.2352
    0.2870   -0.2485
    0.3348   -0.2514
    0.3827   -0.2453
    0.4305   -0.2314
    0.4784   -0.2115
    0.5262   -0.1869
    0.5740   -0.1592
    0.6219   -0.1298
    0.6697   -0.1002
    0.7175   -0.0717
    0.7654   -0.0456
    0.8132   -0.0234
    0.8610   -0.0062
    0.9089    0.0047
    0.9567    0.0084
    1.0045    0.0036
    1.0524   -0.0115
    1.1002   -0.0381
    1.1480   -0.0778
    1.1959   -0.1315
    1.2437   -0.2001
    1.2916   -0.2852
    1.3394   -0.3886
    1.3872   -0.5118
    1.4351   -0.6552
    1.4829   -0.8197
    1.5307   -1.0079
    1.5786   -1.2193
    1.6264   -1.4520
/
 }\relax
\LimeGreen{\relax 
\plot
         0   -0.0026
    0.0478   -0.0688
    0.0957   -0.1267
    0.1435   -0.1745
    0.1913   -0.2108
    0.2392   -0.2356
    0.2870   -0.2490
    0.3348   -0.2520
    0.3827   -0.2459
    0.4305   -0.2321
    0.4784   -0.2122
    0.5262   -0.1877
    0.5740   -0.1601
    0.6219   -0.1308
    0.6697   -0.1013
    0.7175   -0.0729
    0.7654   -0.0469
    0.8132   -0.0248
    0.8610   -0.0077
    0.9089    0.0031
    0.9567    0.0067
    1.0045    0.0017
    1.0524   -0.0135
    1.1002   -0.0404
    1.1480   -0.0803
    1.1959   -0.1341
    1.2437   -0.2029
    1.2916   -0.2884
    1.3394   -0.3922
    1.3872   -0.5159
    1.4351   -0.6600
    1.4829   -0.8254
    1.5307   -1.0149
    1.5786   -1.2284
    1.6264   -1.4683
/
 }\relax
\PineGreen{\relax 
\plot
         0   -0.0026
    0.0478   -0.0688
    0.0957   -0.1268
    0.1435   -0.1745
    0.1913   -0.2109
    0.2392   -0.2356
    0.2870   -0.2490
    0.3348   -0.2520
    0.3827   -0.2460
    0.4305   -0.2322
    0.4784   -0.2123
    0.5262   -0.1879
    0.5740   -0.1603
    0.6219   -0.1310
    0.6697   -0.1014
    0.7175   -0.0731
    0.7654   -0.0471
    0.8132   -0.0250
    0.8610   -0.0080
    0.9089    0.0029
    0.9567    0.0064
    1.0045    0.0014
    1.0524   -0.0139
    1.1002   -0.0407
    1.1480   -0.0806
    1.1959   -0.1345
    1.2437   -0.2034
    1.2916   -0.2888
    1.3394   -0.3927
    1.3872   -0.5164
    1.4351   -0.6606
    1.4829   -0.8261
    1.5307   -1.0157
    1.5786   -1.2294
    1.6264   -1.4696
/
 }\relax
%
%\Blue{\relax 
%\plot
%/
% }\relax
%
\endpicture
}   % end of final time data for f1:  $\beta=-1$, $T=1$
\xfigdim=0.40\xfiglen
\yfigdim=0.37\yfiglen
\setbox\figureeight=\vbox{\hsize=\xfiglen
\beginpicture
\footnotesize
  \setcoordinatesystem units <\xfigdim,\yfigdim>  point at 0 -0.70
  \setplotarea x from 0.0 to 2.5, y from -0.7 to 2
  \axis bottom shiftedto y=-0.7 ticks short numbered from 0 to 2 by 0.5 /
  \axis left ticks short numbered from -0.5 to 2 by 0.5 /
 \put {{\sevenrm $f_2(u)$}} [lb] at 0 2
 \put {{\sevenrm $u$}} [rt] at 2.5 -0.74
\setquadratic
\setquadratic
\setdashes <3pt>
%\Black{
\plot
         0    0.0135
    0.0751    0.0272
    0.1503    0.0518
    0.2254    0.0945
    0.3006    0.1643
    0.3757    0.2707
    0.4508    0.4223
    0.5260    0.6226
    0.6011    0.8664
    0.6762    1.1386
    0.7514    1.4119
    0.8265    1.6522
    0.9017    1.8243
    0.9768    1.8992
    1.0519    1.8625
    1.1271    1.7178
    1.2022    1.4858
    1.2773    1.1982
    1.3525    0.8917
    1.4276    0.5976
    1.5028    0.3393
    1.5779    0.1276
    1.6530   -0.0362
    1.7282   -0.1575
    1.8033   -0.2458
    1.8785   -0.3107
    1.9536   -0.3605
    2.0287   -0.4015
    2.1039   -0.4381
    2.1790   -0.4729
    2.2541   -0.5073
    2.3293   -0.5422
    2.4044   -0.5780
    2.4796   -0.6148
    2.5547   -0.6527
/
%}
%
\setsolid
\Goldenrod{\relax 
\plot
         0   -0.0046
    0.0751    0.0228
    0.1503    0.0660
    0.2254    0.1296
    0.3006    0.2216
    0.3757    0.3506
    0.4508    0.5233
    0.5260    0.7408
    0.6011    0.9961
    0.6762    1.2732
    0.7514    1.5473
    0.8265    1.7892
    0.9017    1.9689
    0.9768    2.0615
    1.0519    2.0526
    1.1271    1.9411
    1.2022    1.7410
    1.2773    1.4790
    1.3525    1.1891
    1.4276    0.9057
    1.5028    0.6555
    1.5779    0.4528
    1.6530    0.2987
    1.7282    0.1855
    1.8033    0.1030
    1.8785    0.0436
    1.9536    0.0021
    2.0287   -0.0288
    2.1039   -0.0587
    2.1790   -0.0910
    2.2541   -0.1199
    2.3293   -0.1454
    2.4044   -0.1800
    2.4796   -0.2100
    2.5547   -0.2427
/
 }\relax
\Orange{\relax 
\plot
         0   -0.0048
    0.0751    0.0103
    0.1503    0.0410
    0.2254    0.0922
    0.3006    0.1716
    0.3757    0.2879
    0.4508    0.4477
    0.5260    0.6519
    0.6011    0.8936
    0.6762    1.1564
    0.7514    1.4155
    0.8265    1.6417
    0.9017    1.8050
    0.9768    1.8805
    1.0519    1.8542
    1.1271    1.7254
    1.2022    1.5083
    1.2773    1.2303
    1.3525    0.9255
    1.4276    0.6286
    1.5028    0.3663
    1.5779    0.1528
    1.6530   -0.0111
    1.7282   -0.1331
    1.8033   -0.2236
    1.8785   -0.2905
    1.9536   -0.3391
    2.0287   -0.3768
    2.1039   -0.4130
    2.1790   -0.4512
    2.2541   -0.4855
    2.3293   -0.5158
    2.4044   -0.5537
    2.4796   -0.5849
    2.5547   -0.6166
/
 }\relax
\Red{\relax 
\plot
         0   -0.0048
    0.0751    0.0095
    0.1503    0.0392
    0.2254    0.0895
    0.3006    0.1681
    0.3757    0.2836
    0.4508    0.4424
    0.5260    0.6459
    0.6011    0.8868
    0.6762    1.1488
    0.7514    1.4072
    0.8265    1.6327
    0.9017    1.7952
    0.9768    1.8700
    1.0519    1.8428
    1.1271    1.7129
    1.2022    1.4947
    1.2773    1.2154
    1.3525    0.9093
    1.4276    0.6111
    1.5028    0.3477
    1.5779    0.1334
    1.6530   -0.0311
    1.7282   -0.1535
    1.8033   -0.2441
    1.8785   -0.3110
    1.9536   -0.3593
    2.0287   -0.3966
    2.1039   -0.4325
    2.1790   -0.4705
    2.2541   -0.5047
    2.3293   -0.5351
    2.4044   -0.5745
    2.4796   -0.6088
    2.5547   -0.6434
/
 }\relax
\LimeGreen{\relax 
\plot
         0   -0.0048
    0.0751    0.0095
    0.1503    0.0393
    0.2254    0.0895
    0.3006    0.1682
    0.3757    0.2837
    0.4508    0.4426
    0.5260    0.6461
    0.6011    0.8870
    0.6762    1.1491
    0.7514    1.4075
    0.8265    1.6330
    0.9017    1.7955
    0.9768    1.8704
    1.0519    1.8432
    1.1271    1.7133
    1.2022    1.4952
    1.2773    1.2159
    1.3525    0.9098
    1.4276    0.6117
    1.5028    0.3483
    1.5779    0.1340
    1.6530   -0.0304
    1.7282   -0.1527
    1.8033   -0.2434
    1.8785   -0.3101
    1.9536   -0.3584
    2.0287   -0.3957
    2.1039   -0.4316
    2.1790   -0.4695
    2.2541   -0.5037
    2.3293   -0.5341
    2.4044   -0.5734
    2.4796   -0.6077
    2.5547   -0.6420
/
 }\relax
\PineGreen{\relax 
\plot
         0   -0.0048
    0.0751    0.0095
    0.1503    0.0393
    0.2254    0.0896
    0.3006    0.1682
    0.3757    0.2838
    0.4508    0.4426
    0.5260    0.6462
    0.6011    0.8871
    0.6762    1.1492
    0.7514    1.4076
    0.8265    1.6331
    0.9017    1.7956
    0.9768    1.8705
    1.0519    1.8432
    1.1271    1.7135
    1.2022    1.4954
    1.2773    1.2161
    1.3525    0.9100
    1.4276    0.6119
    1.5028    0.3485
    1.5779    0.1343
    1.6530   -0.0301
    1.7282   -0.1525
    1.8033   -0.2432
    1.8785   -0.3098
    1.9536   -0.3581
    2.0287   -0.3954
    2.1039   -0.4312
    2.1790   -0.4691
    2.2541   -0.5033
    2.3293   -0.5337
    2.4044   -0.5730
    2.4796   -0.6072
    2.5547   -0.6413
/
 }\relax
%
%\Blue{\relax 
%\plot
%/
% }\relax
%
\endpicture
}   % end of final time data for f2:  $\beta=-1$, $T=1$
\xfigdim=0.435\xfiglen
\yfigdim=0.118\yfiglen
\setbox\figurenine=\vbox{\hsize=\xfiglen
% This is for: final time data, beta=0.3, T=1, f_1(u)
\beginpicture
\footnotesize
  \setcoordinatesystem units <\xfigdim,\yfigdim>  point at 0 -7
  \setplotarea x from 0 to 2.3, y from -7 to 1.5
  \axis bottom shiftedto y=-7 ticks short numbered from 0 to 2 by 0.5 /
  \axis left ticks short numbered from -7 to 1 by 2 /
 \put {{\sevenrm $f_1(u)$}} [lb] at 0.02 1
 \put {{\sevenrm $u$}} [rt] at 2.25 -7.3
\setquadratic
\setdashes <3pt>
\Black{
\plot
         0         0
    0.0650   -0.1016
    0.1301   -0.1743
    0.1951   -0.2213
    0.2602   -0.2463
    0.3252   -0.2523
    0.3903   -0.2426
    0.4553   -0.2206
    0.5204   -0.1894
    0.5854   -0.1527
    0.6505   -0.1135
    0.7155   -0.0752
    0.7806   -0.0409
    0.8456   -0.0142
    0.9107    0.0017
    0.9757    0.0036
    1.0408   -0.0119
    1.1058   -0.0482
    1.1708   -0.1084
    1.2359   -0.1959
    1.3009   -0.3139
    1.3660   -0.4658
    1.4310   -0.6551
    1.4961   -0.8848
    1.5611   -1.1584
    1.6262   -1.4791
    1.6912   -1.8496
    1.7563   -2.2744
    1.8213   -2.7563
    1.8864   -3.2984
    1.9514   -3.9043
    2.0165   -4.5772
    2.0815   -5.3190
    2.1466   -6.1355
    2.2116   -7.0290
/
 }\relax
\setsolid
\Goldenrod{\relax 
\plot
        0   -0.0157
    0.0650   -0.0039
    0.1301    0.0105
    0.1951    0.0360
    0.2602    0.0793
    0.3252    0.1433
    0.3903    0.2272
    0.4553    0.3277
    0.5204    0.4405
    0.5854    0.5611
    0.6505    0.6858
    0.7155    0.8111
    0.7806    0.9339
    0.8456    1.0514
    0.9107    1.1608
    0.9757    1.2588
    1.0408    1.3425
    1.1058    1.4084
    1.1708    1.4534
    1.2359    1.4753
    1.3009    1.4716
    1.3660    1.4399
    1.4310    1.3768
    1.4961    1.2784
    1.5611    1.1414
    1.6262    0.9634
    1.6912    0.7436
    1.7563    0.4795
    1.8213    0.1672
    1.8864   -0.1981
    1.9514   -0.6171
    2.0165   -1.0879
    2.0815   -1.6160
    2.1466   -2.2060
    2.2116   -2.8570
/
 }\relax
\Orange{\relax 
\plot
         0   -0.0161
    0.0650   -0.0573
    0.1301   -0.0964
    0.1951   -0.1245
    0.2602   -0.1349
    0.3252   -0.1248
    0.3903   -0.0955
    0.4553   -0.0502
    0.5204    0.0066
    0.5854    0.0704
    0.6505    0.1372
    0.7155    0.2034
    0.7806    0.2658
    0.8456    0.3214
    0.9107    0.3672
    0.9757    0.4003
    1.0408    0.4173
    1.1058    0.4150
    1.1708    0.3904
    1.2359    0.3412
    1.3009    0.2658
    1.3660    0.1618
    1.4310    0.0267
    1.4961   -0.1421
    1.5611   -0.3467
    1.6262   -0.5869
    1.6912   -0.8599
    1.7563   -1.1637
    1.8213   -1.4999
    1.8864   -1.8878
    1.9514   -2.3294
    2.0165   -2.8230
    2.0815   -3.3741
    2.1466   -3.9874
    2.2116   -4.6628
/
 }\relax
\Red{\relax 
\plot
         0   -0.0163
    0.0650   -0.0783
    0.1301   -0.1383
    0.1951   -0.1874
    0.2602   -0.2190
    0.3252   -0.2302
    0.3903   -0.2223
    0.4553   -0.1988
    0.5204   -0.1639
    0.5854   -0.1223
    0.6505   -0.0782
    0.7155   -0.0351
    0.7806    0.0035
    0.8456    0.0347
    0.9107    0.0554
    0.9757    0.0623
    1.0408    0.0517
    1.1058    0.0202
    1.1708   -0.0355
    1.2359   -0.1181
    1.3009   -0.2301
    1.3660   -0.3744
    1.4310   -0.5546
    1.4961   -0.7748
    1.5611   -1.0385
    1.6262   -1.3480
    1.6912   -1.7041
    1.7563   -2.1098
    1.8213   -2.5682
    1.8864   -3.0830
    1.9514   -3.6527
    2.0165   -4.2718
    2.0815   -4.9352
    2.1466   -5.6187
    2.2116   -6.3406
/
 }\relax
\LimeGreen{\relax 
\plot
         0   -0.0163
    0.0650   -0.0814
    0.1301   -0.1447
    0.1951   -0.1970
    0.2602   -0.2315
    0.3252   -0.2460
    0.3903   -0.2413
    0.4553   -0.2210
    0.5204   -0.1894
    0.5854   -0.1512
    0.6505   -0.1103
    0.7155   -0.0707
    0.7806   -0.0354
    0.8456   -0.0078
    0.9107    0.0093
    0.9757    0.0123
    1.0408   -0.0021
    1.1058   -0.0378
    1.1708   -0.0979
    1.2359   -0.1854
    1.3009   -0.3025
    1.3660   -0.4527
    1.4310   -0.6395
    1.4961   -0.8671
    1.5611   -1.1394
    1.6262   -1.4592
    1.6912   -1.8275
    1.7563   -2.2483
    1.8213   -2.7257
    1.8864   -3.2649
    1.9514   -3.8665
    2.0165   -4.5287
    2.0815   -5.2555
    2.1466   -6.0461
    2.2116   -6.8596
/
 }\relax
\PineGreen{\relax 
\plot
         0   -0.0163
    0.0650   -0.0818
    0.1301   -0.1455
    0.1951   -0.1982
    0.2602   -0.2332
    0.3252   -0.2481
    0.3903   -0.2438
    0.4553   -0.2239
    0.5204   -0.1928
    0.5854   -0.1550
    0.6505   -0.1146
    0.7155   -0.0754
    0.7806   -0.0407
    0.8456   -0.0135
    0.9107    0.0030
    0.9757    0.0055
    1.0408   -0.0095
    1.1058   -0.0458
    1.1708   -0.1065
    1.2359   -0.1948
    1.3009   -0.3128
    1.3660   -0.4638
    1.4310   -0.6514
    1.4961   -0.8802
    1.5611   -1.1539
    1.6262   -1.4754
    1.6912   -1.8458
    1.7563   -2.2688
    1.8213   -2.7488
    1.8864   -3.2913
    1.9514   -3.8978
    2.0165   -4.5672
    2.0815   -5.3059
    2.1466   -6.1187
    2.2116   -6.9960
/
 }\relax
%
%\Blue{\relax 
%\plot
%/
% }\relax
%
\endpicture
}   % end of final time data for f1:  $\beta=0.3$, $T=1$
\xfigdim=0.25\xfiglen
\yfigdim=0.25\yfiglen
\setbox\figureten=\vbox{\hsize=\xfiglen
\beginpicture
\footnotesize
  \setcoordinatesystem units <\xfigdim,\yfigdim>  point at 0 -1.5
  \setplotarea x from 0.0 to 4, y from -1.5 to 2.5
  \axis bottom shiftedto y=-1.5 ticks short numbered from 0 to 4 by 1 /
  \axis left ticks short numbered from -1.0 to 2 by 1 /
 \put {{\sevenrm $f_2(u)$}} [lt] at 0.1 2.5
 \put {{\sevenrm $u$}} [rb] at 4 -1.4
\setquadratic
\setdashes <3pt>
\Black{
\plot
         0    0.0135
    0.1064    0.0358
    0.2128    0.0857
    0.3191    0.1869
    0.4255    0.3659
    0.5319    0.6405
    0.6383    0.9988
    0.7447    1.3882
    0.8510    1.7176
    0.9574    1.8903
    1.0638    1.8466
    1.1702    1.5935
    1.2766    1.2014
    1.3829    0.7697
    1.4893    0.3823
    1.5957    0.0845
    1.7021   -0.1196
    1.8084   -0.2509
    1.9148   -0.3362
    2.0212   -0.3976
    2.1276   -0.4492
    2.2340   -0.4980
    2.3403   -0.5475
    2.4467   -0.5986
    2.5531   -0.6518
    2.6595   -0.7073
    2.7659   -0.7650
    2.8722   -0.8250
    2.9786   -0.8872
    3.0850   -0.9518
    3.1914   -1.0185
    3.2978   -1.0875
    3.4041   -1.1588
    3.5105   -1.2324
    3.6169   -1.3082
/
}
\setsolid
\Goldenrod{\relax 
\plot
         0   -0.0332
    0.1064    0.0550
    0.2128    0.1680
    0.3191    0.3315
    0.4255    0.5714
    0.5319    0.8954
    0.6383    1.2809
    0.7447    1.6764
    0.8510    2.0078
    0.9574    2.2026
    1.0638    2.2157
    1.1702    2.0470
    1.2766    1.7436
    1.3829    1.3834
    1.4893    1.0432
    1.5957    0.7752
    1.7021    0.5966
    1.8084    0.4960
    1.9148    0.4499
    2.0212    0.4353
    2.1276    0.4346
    2.2340    0.4367
    2.3403    0.4354
    2.4467    0.4293
    2.5531    0.4216
    2.6595    0.4165
    2.7659    0.4144
    2.8722    0.4097
    2.9786    0.3974
    3.0850    0.3834
    3.1914    0.3776
    3.2978    0.3710
    3.4041    0.3505
    3.5105    0.3439
    3.6169    0.3258
/
 }\relax
\Orange{\relax 
\plot
         0   -0.0335
    0.1064    0.0443
    0.2128    0.1464
    0.3191    0.2989
    0.4255    0.5275
    0.5319    0.8394
    0.6383    1.2118
    0.7447    1.5929
    0.8510    1.9087
    0.9574    2.0874
    1.0638    2.0844
    1.1702    1.9008
    1.2766    1.5842
    1.3829    1.2129
    1.4893    0.8632
    1.5957    0.5868
    1.7021    0.4003
    1.8084    0.2919
    1.9148    0.2380
    2.0212    0.2152
    2.1276    0.2062
    2.2340    0.1999
    2.3403    0.1901
    2.4467    0.1754
    2.5531    0.1588
    2.6595    0.1451
    2.7659    0.1346
    2.8722    0.1214
    2.9786    0.1008
    3.0850    0.0792
    3.1914    0.0665
    3.2978    0.0528
    3.4041    0.0252
    3.5105    0.0116
    3.6169   -0.0137
/
 }\relax
\Red{\relax 
\plot
         0   -0.0339
    0.1064    0.0255
    0.2128    0.1089
    0.3191    0.2427
    0.4255    0.4525
    0.5319    0.7457
    0.6383    1.0994
    0.7447    1.4616
    0.8510    1.7582
    0.9574    1.9171
    1.0638    1.8936
    1.1702    1.6882
    1.2766    1.3491
    1.3829    0.9547
    1.4893    0.5818
    1.5957    0.2822
    1.7021    0.0728
    1.8084   -0.0584
    1.9148   -0.1353
    2.0212   -0.1814
    2.1276   -0.2142
    2.2340   -0.2449
    2.3403   -0.2796
    2.4467   -0.3198
    2.5531   -0.3624
    2.6595   -0.4030
    2.7659   -0.4412
    2.8722   -0.4825
    2.9786   -0.5318
    3.0850   -0.5832
    3.1914   -0.6264
    3.2978   -0.6696
    3.4041   -0.7253
    3.5105   -0.7646
    3.6169   -0.8165
/
 }\relax
\LimeGreen{\relax 
\plot
         0   -0.0341
    0.1064    0.0177
    0.2128    0.0932
    0.3191    0.2193
    0.4255    0.4214
    0.5319    0.7067
    0.6383    1.0527
    0.7447    1.4071
    0.8510    1.6958
    0.9574    1.8467
    1.0638    1.8148
    1.1702    1.6007
    1.2766    1.2523
    1.3829    0.8483
    1.4893    0.4657
    1.5957    0.1565
    1.7021   -0.0622
    1.8084   -0.2029
    1.9148   -0.2888
    2.0212   -0.3442
    2.1276   -0.3862
    2.2340   -0.4263
    2.3403   -0.4708
    2.4467   -0.5210
    2.5531   -0.5737
    2.6595   -0.6247
    2.7659   -0.6737
    2.8722   -0.7264
    2.9786   -0.7876
    3.0850   -0.8513
    3.1914   -0.9076
    3.2978   -0.9656
    3.4041   -1.0368
    3.5105   -1.0931
    3.6169   -1.1540
/
 }\relax
\PineGreen{\relax 
\plot
         0   -0.0342
    0.1064    0.0155
    0.2128    0.0889
    0.3191    0.2128
    0.4255    0.4127
    0.5319    0.6959
    0.6383    1.0397
    0.7447    1.3919
    0.8510    1.6785
    0.9574    1.8271
    1.0638    1.7929
    1.1702    1.5762
    1.2766    1.2253
    1.3829    0.8186
    1.4893    0.4332
    1.5957    0.1213
    1.7021   -0.1003
    1.8084   -0.2435
    1.9148   -0.3321
    2.0212   -0.3901
    2.1276   -0.4348
    2.2340   -0.4776
    2.3403   -0.5249
    2.4467   -0.5780
    2.5531   -0.6336
    2.6595   -0.6877
    2.7659   -0.7398
    2.8722   -0.7958
    2.9786   -0.8607
    3.0850   -0.9283
    3.1914   -0.9888
    3.2978   -1.0517
    3.4041   -1.1292
    3.5105   -1.1930
    3.6169   -1.2625
/
 }\relax
\Blue{\relax 
\plot
         0   -0.0342
    0.1064    0.0150
    0.2128    0.0877
    0.3191    0.2111
    0.4255    0.4104
    0.5319    0.6930
    0.6383    1.0362
    0.7447    1.3878
    0.8510    1.6739
    0.9574    1.8220
    1.0638    1.7870
    1.1702    1.5698
    1.2766    1.2181
    1.3829    0.8108
    1.4893    0.4246
    1.5957    0.1120
    1.7021   -0.1103
    1.8084   -0.2542
    1.9148   -0.3435
    2.0212   -0.4022
    2.1276   -0.4476
    2.2340   -0.4911
    2.3403   -0.5392
    2.4467   -0.5930
    2.5531   -0.6495
    2.6595   -0.7044
    2.7659   -0.7573
    2.8722   -0.8143
    2.9786   -0.8802
    3.0850   -0.9488
    3.1914   -1.0104
    3.2978   -1.0747
    3.4041   -1.1540
    3.5105   -1.2208
    3.6169   -1.2952
/
 }\relax
\endpicture
}   % end of final time data for f2:  $\beta=0.3$, $T=1$
\xfigdim=0.571\xfiglen
\yfigdim=0.526\yfiglen
\setbox\figureeleven=\vbox{\hsize=\xfiglen
\beginpicture
\footnotesize
  \setcoordinatesystem units <\xfigdim,\yfigdim>  point at 0 -1.7
  \setplotarea x from 0.0 to 1.75, y from -1.7 to 0.2
  \axis bottom shiftedto y=-1.7 ticks short numbered from 0 to 1.5 by 0.5 /
  \axis left ticks short numbered from -1.5 to 0 by 0.5 /
 \put {{\sevenrm $f_1(u)$}} [lt] at 0.04 0.2
 \put {{\sevenrm $u$}} [rb] at 1.75 -1.6
\setquadratic
\setdashes <3pt>
%\Black{
\plot
         0         0
    0.0487   -0.0789
    0.0974   -0.1411
    0.1461   -0.1881
    0.1948   -0.2212
    0.2435   -0.2419
    0.2921   -0.2514
    0.3408   -0.2513
    0.3895   -0.2428
    0.4382   -0.2274
    0.4869   -0.2064
    0.5356   -0.1813
    0.5843   -0.1534
    0.6330   -0.1241
    0.6817   -0.0948
    0.7304   -0.0668
    0.7790   -0.0416
    0.8277   -0.0206
    0.8764   -0.0051
    0.9251    0.0035
    0.9738    0.0038
    1.0225   -0.0056
    1.0712   -0.0261
    1.1199   -0.0590
    1.1686   -0.1058
    1.2173   -0.1678
    1.2660   -0.2464
    1.3146   -0.3430
    1.3633   -0.4590
    1.4120   -0.5958
    1.4607   -0.7547
    1.5094   -0.9371
    1.5581   -1.1445
    1.6068   -1.3782
    1.6555   -1.6396
/
%}
%
\setsolid
\Goldenrod{\relax 
\plot
         0   -0.0074
    0.0487   -0.0649
    0.0974   -0.1218
    0.1461   -0.1716
    0.1948   -0.2097
    0.2435   -0.2341
    0.2921   -0.2450
    0.3408   -0.2439
    0.3895   -0.2327
    0.4382   -0.2133
    0.4869   -0.1874
    0.5356   -0.1564
    0.5843   -0.1216
    0.6330   -0.0842
    0.6817   -0.0453
    0.7304   -0.0062
    0.7790    0.0320
    0.8277    0.0682
    0.8764    0.1016
    0.9251    0.1310
    0.9738    0.1552
    1.0225    0.1727
    1.0712    0.1822
    1.1199    0.1826
    1.1686    0.1733
    1.2173    0.1534
    1.2660    0.1216
    1.3146    0.0762
    1.3633    0.0157
    1.4120   -0.0601
    1.4607   -0.1513
    1.5094   -0.2604
    1.5581   -0.3895
    1.6068   -0.5362
    1.6555   -0.7043
/
 }\relax
\Orange{\relax 
\plot
         0   -0.0074
    0.0487   -0.0568
    0.0974   -0.1056
    0.1461   -0.1475
    0.1948   -0.1778
    0.2435   -0.1945
    0.2921   -0.1981
    0.3408   -0.1901
    0.3895   -0.1728
    0.4382   -0.1480
    0.4869   -0.1178
    0.5356   -0.0838
    0.5843   -0.0473
    0.6330   -0.0098
    0.6817    0.0273
    0.7304    0.0627
    0.7790    0.0948
    0.8277    0.1226
    0.8764    0.1449
    0.9251    0.1604
    0.9738    0.1677
    1.0225    0.1652
    1.0712    0.1515
    1.1199    0.1253
    1.1686    0.0861
    1.2173    0.0331
    1.2660   -0.0346
    1.3146   -0.1183
    1.3633   -0.2185
    1.4120   -0.3338
    1.4607   -0.4617
    1.5094   -0.6010
    1.5581   -0.7463
    1.6068   -0.8894
    1.6555   -1.0529
/
 }\relax
\Red{\relax 
\plot
         0   -0.0074
    0.0487   -0.0638
    0.0974   -0.1196
    0.1461   -0.1684
    0.1948   -0.2056
    0.2435   -0.2292
    0.2921   -0.2396
    0.3408   -0.2383
    0.3895   -0.2276
    0.4382   -0.2093
    0.4869   -0.1854
    0.5356   -0.1575
    0.5843   -0.1270
    0.6330   -0.0953
    0.6817   -0.0637
    0.7304   -0.0337
    0.7790   -0.0067
    0.8277    0.0162
    0.8764    0.0337
    0.9251    0.0447
    0.9738    0.0476
    1.0225    0.0408
    1.0712    0.0228
    1.1199   -0.0078
    1.1686   -0.0518
    1.2173   -0.1102
    1.2660   -0.1843
    1.3146   -0.2761
    1.3633   -0.3870
    1.4120   -0.5171
    1.4607   -0.6661
    1.5094   -0.8355
    1.5581   -1.0248
    1.6068   -1.2243
    1.6555   -1.4082
/
 }\relax
\LimeGreen{\relax 
\plot
         0   -0.0074
    0.0487   -0.0656
    0.0974   -0.1232
    0.1461   -0.1739
    0.1948   -0.2129
    0.2435   -0.2382
    0.2921   -0.2504
    0.3408   -0.2510
    0.3895   -0.2420
    0.4382   -0.2255
    0.4869   -0.2033
    0.5356   -0.1771
    0.5843   -0.1484
    0.6330   -0.1183
    0.6817   -0.0885
    0.7304   -0.0602
    0.7790   -0.0348
    0.8277   -0.0137
    0.8764    0.0022
    0.9251    0.0114
    0.9738    0.0125
    1.0225    0.0039
    1.0712   -0.0161
    1.1199   -0.0488
    1.1686   -0.0951
    1.2173   -0.1561
    1.2660   -0.2333
    1.3146   -0.3286
    1.3633   -0.4437
    1.4120   -0.5788
    1.4607   -0.7344
    1.5094   -0.9131
    1.5581   -1.1162
    1.6068   -1.3391
    1.6555   -1.5587
/
 }\relax
\PineGreen{\relax 
\plot
         0   -0.0074
    0.0487   -0.0661
    0.0974   -0.1242
    0.1461   -0.1753
    0.1948   -0.2147
    0.2435   -0.2405
    0.2921   -0.2532
    0.3408   -0.2542
    0.3895   -0.2456
    0.4382   -0.2296
    0.4869   -0.2078
    0.5356   -0.1821
    0.5843   -0.1537
    0.6330   -0.1241
    0.6817   -0.0947
    0.7304   -0.0668
    0.7790   -0.0419
    0.8277   -0.0211
    0.8764   -0.0056
    0.9251    0.0032
    0.9738    0.0039
    1.0225   -0.0052
    1.0712   -0.0257
    1.1199   -0.0589
    1.1686   -0.1058
    1.2173   -0.1673
    1.2660   -0.2451
    1.3146   -0.3413
    1.3633   -0.4576
    1.4120   -0.5943
    1.4607   -0.7519
    1.5094   -0.9333
    1.5581   -1.1405
    1.6068   -1.3712
    1.6555   -1.6275
/
 }\relax
%
%\Blue{\relax 
% }\relax
%
\endpicture
}   % end of final time data for f1:  $\beta=0.5$, $T=0.7$
\xfigdim=0.4\xfiglen
\yfigdim=0.333\yfiglen
\setbox\figuretwelve=\vbox{\hsize=\xfiglen
% This is for: final time data, beta=1.3, T=1, f_1(u)
\beginpicture
\footnotesize
  \setcoordinatesystem units <\xfigdim,\yfigdim>  point at 0 -1
  \setplotarea x from 0 to 2.5, y from -1 to 2
  \axis bottom shiftedto y=-1 ticks short numbered from 0 to 2.5 by 0.5 /
  \axis left ticks short numbered from -1 to 2 by 1 /
 \put {{\sevenrm $f_2(u)$}} [lt] at 0.02 2
 \put {{\sevenrm $u$}} [rb] at 2.5 -0.92
\setquadratic
\setdashes <3pt>
\Black{
\plot
         0    0.0135
    0.0733    0.0268
    0.1467    0.0503
    0.2200    0.0906
    0.2933    0.1561
    0.3667    0.2557
    0.4400    0.3975
    0.5133    0.5856
    0.5867    0.8168
    0.6600    1.0784
    0.7333    1.3477
    0.8067    1.5940
    0.8800    1.7836
    0.9533    1.8874
    1.0266    1.8875
    1.1000    1.7815
    1.1733    1.5834
    1.2466    1.3201
    1.3200    1.0245
    1.3933    0.7287
    1.4666    0.4582
    1.5400    0.2283
    1.6133    0.0447
    1.6866   -0.0951
    1.7600   -0.1983
    1.8333   -0.2740
    1.9066   -0.3307
    1.9800   -0.3756
    2.0533   -0.4138
    2.1266   -0.4487
    2.2000   -0.4825
    2.2733   -0.5162
    2.3466   -0.5504
    2.4200   -0.5855
    2.4933   -0.6216
/
}
\setsolid
\Goldenrod{\relax 
\plot
         0   -0.0097
    0.0733    0.0330
    0.1467    0.0831
    0.2200    0.1488
    0.2933    0.2406
    0.3667    0.3683
    0.4400    0.5386
    0.5133    0.7530
    0.5867    1.0057
    0.6600    1.2832
    0.7333    1.5644
    0.8067    1.8231
    0.8800    2.0311
    0.9533    2.1631
    1.0266    2.2016
    1.1000    2.1406
    1.1733    1.9883
    1.2466    1.7660
    1.3200    1.5042
    1.3933    1.2355
    1.4666    0.9872
    1.5400    0.7764
    1.6133    0.6092
    1.6866    0.4836
    1.7600    0.3937
    1.8333    0.3324
    1.9066    0.2917
    1.9800    0.2618
    2.0533    0.2355
    2.1266    0.2117
    2.2000    0.1931
    2.2733    0.1747
    2.3466    0.1506
    2.4200    0.1327
    2.4933    0.1120
/
 }\relax
\Orange{\relax 
\plot
         0   -0.0098
    0.0733    0.0208
    0.1467    0.0587
    0.2200    0.1122
    0.2933    0.1916
    0.3667    0.3066
    0.4400    0.4636
    0.5133    0.6640
    0.5867    0.9018
    0.6600    1.1630
    0.7333    1.4264
    0.8067    1.6654
    0.8800    1.8521
    0.9533    1.9614
    1.0266    1.9763
    1.1000    1.8916
    1.1733    1.7164
    1.2466    1.4730
    1.3200    1.1925
    1.3933    0.9082
    1.4666    0.6474
    1.5400    0.4271
    1.6133    0.2530
    1.6866    0.1225
    1.7600    0.0293
    1.8333   -0.0341
    1.9066   -0.0764
    1.9800   -0.1071
    2.0533   -0.1338
    2.1266   -0.1574
    2.2000   -0.1756
    2.2733   -0.1926
    2.3466   -0.2144
    2.4200   -0.2279
    2.4933   -0.2437
/
 }\relax
\Red{\relax 
\plot
         0   -0.0099
    0.0733    0.0156
    0.1467    0.0482
    0.2200    0.0964
    0.2933    0.1705
    0.3667    0.2802
    0.4400    0.4320
    0.5133    0.6271
    0.5867    0.8596
    0.6600    1.1156
    0.7333    1.3736
    0.8067    1.6073
    0.8800    1.7885
    0.9533    1.8922
    1.0266    1.9010
    1.1000    1.8098
    1.1733    1.6275
    1.2466    1.3766
    1.3200    1.0881
    1.3933    0.7955
    1.4666    0.5262
    1.5400    0.2974
    1.6133    0.1149
    1.6866   -0.0238
    1.7600   -0.1250
    1.8333   -0.1963
    1.9066   -0.2464
    1.9800   -0.2852
    2.0533   -0.3202
    2.1266   -0.3524
    2.2000   -0.3793
    2.2733   -0.4059
    2.3466   -0.4375
    2.4200   -0.4618
    2.4933   -0.4840
/
 }\relax
\LimeGreen{\relax 
\plot
       0   -0.0099
    0.0733    0.0137
    0.1467    0.0446
    0.2200    0.0909
    0.2933    0.1633
    0.3667    0.2712
    0.4400    0.4211
    0.5133    0.6144
    0.5867    0.8450
    0.6600    1.0992
    0.7333    1.3554
    0.8067    1.5872
    0.8800    1.7666
    0.9533    1.8682
    1.0266    1.8749
    1.1000    1.7814
    1.1733    1.5967
    1.2466    1.3431
    1.3200    1.0519
    1.3933    0.7563
    1.4666    0.4839
    1.5400    0.2521
    1.6133    0.0665
    1.6866   -0.0753
    1.7600   -0.1795
    1.8333   -0.2537
    1.9066   -0.3068
    1.9800   -0.3485
    2.0533   -0.3865
    2.1266   -0.4219
    2.2000   -0.4522
    2.2733   -0.4825
    2.3466   -0.5184
    2.4200   -0.5478
    2.4933   -0.5750
/
 }\relax
\PineGreen{\relax 
\plot
         0   -0.0099
    0.0733    0.0132
    0.1467    0.0434
    0.2200    0.0893
    0.2933    0.1610
    0.3667    0.2683
    0.4400    0.4177
    0.5133    0.6104
    0.5867    0.8405
    0.6600    1.0941
    0.7333    1.3497
    0.8067    1.5810
    0.8800    1.7597
    0.9533    1.8607
    1.0266    1.8668
    1.1000    1.7726
    1.1733    1.5872
    1.2466    1.3328
    1.3200    1.0406
    1.3933    0.7441
    1.4666    0.4708
    1.5400    0.2380
    1.6133    0.0514
    1.6866   -0.0913
    1.7600   -0.1965
    1.8333   -0.2717
    1.9066   -0.3256
    1.9800   -0.3683
    2.0533   -0.4073
    2.1266   -0.4436
    2.2000   -0.4750
    2.2733   -0.5064
    2.3466   -0.5438
    2.4200   -0.5752
    2.4933   -0.6043
/
 }\relax
%
%\Blue{\relax 
% }\relax
%
\endpicture
}   % end of final time data for f1:  $\beta=1.3$, $T=1$

\begin{figure}[h]
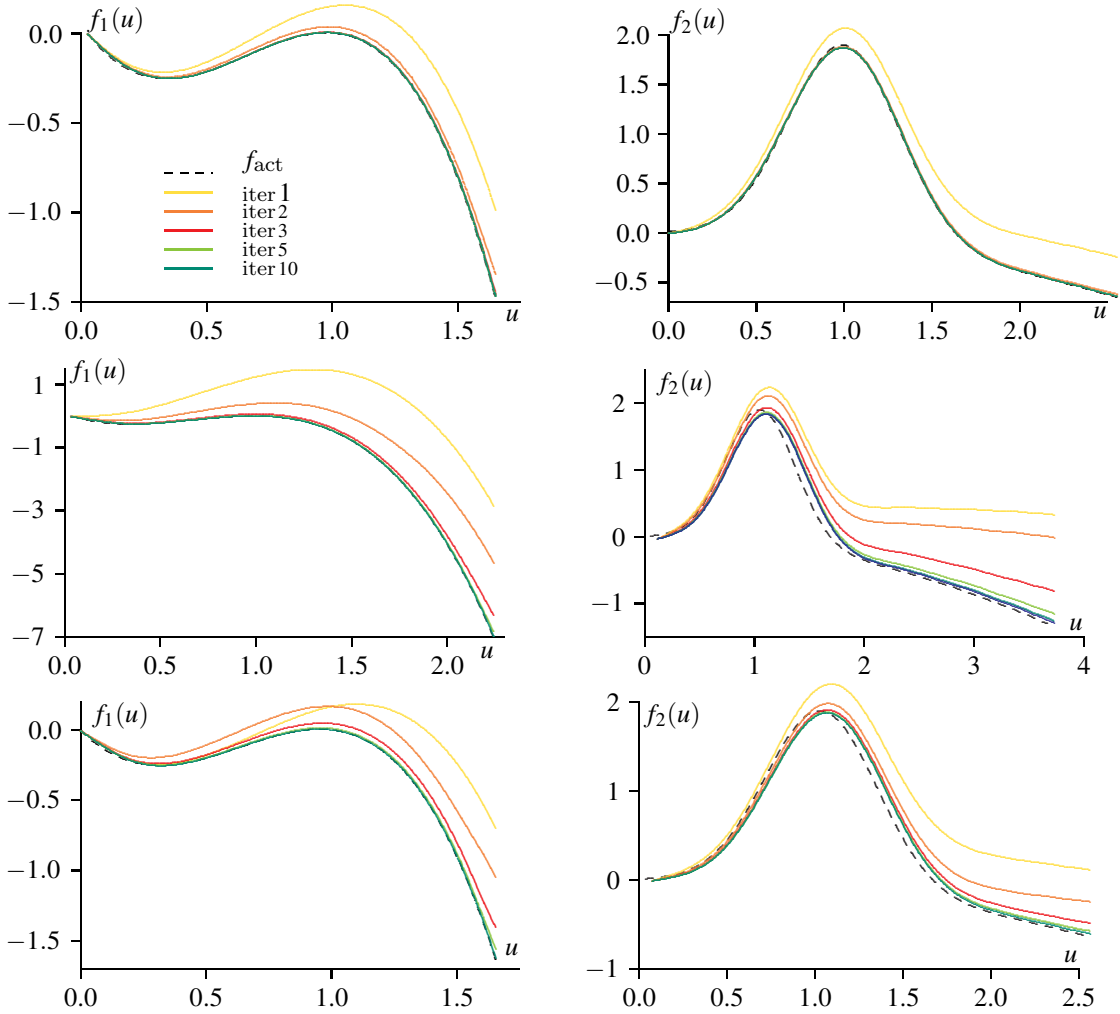

\hbox to\hsize{\hss\copy\figureseven\hss\hss\hss\copy\figureeight\hss}
\medskip
\hbox to\hsize{\hss\copy\figurenine\hss\hss\hss\copy\figureten\hss}
\medskip
\hbox to\hsize{\hss\copy\figureeleven\hss\hss\hss\copy\figuretwelve\hss}
%\centerline{\small {\bf Plots of iterations:}
%Final time data. $T=1$, $\beta=-1$ (top), $\beta=0.3$ (middle);
%$T=0.7$, $\beta=0.5$ (bottom)}
\caption{\small
 Final time data.\quad
Reconstructions of $f_1$ and $f_2$ at selected  iterations as a function of
$\beta$. 
$\beta=-1$ (top), $\beta=0.3$ (middle), $\beta=0.5$ (bottom)
}
\label{fig:fiti_reconstructions_f_with_beta}
\end{figure}
We should expect in fact that the convergence rate will decrease with
increasing $\beta$ and this is entirely borne out by numerical computations
As expected, there is also a difference depending on the sign of
$\beta$: the case of $\beta<0$ corresponds to a negative interaction
and might be expected to be a more stable case and
turns out to be the situation.
The case $\beta>0$ indicates symbiosis between the quantities represented
by $u$ and $v$ and should be expected to be a more unstable situation.
This is certainly the outcome in the ordinary differential equation
version and sufficiently strong reaction terms will dominate the diffusion
effect supplied by the elliptic operator.
Thus one of the two of $u$ or $v$ should dominate and lead to
blow up of the solution - perhaps in finite time.
The conditions for this are quite complex.
In the case of an ordinary differential equation the nonlinear
right hand side terms must, in general, be uniformly Lipschitz
and if they are positive and have asymptotic growth of order $u^\delta$ for
$\delta>1$ then this occurs.  Under the effects of a diffusion operator
this is more complex; the power law growth is coupled to, among other things,
the dimension of the space, 
see \cite{HollisLMartinPierre:1987,Pierre:2010,Wang:2000}.
The limitation we must impose
in this case is that the measurements taken be up to or at a time $T$ before
the onset of the solution failing to exist.
Of course, long before this situation the iteration scheme would most likely
have failed.
This is apparent from Figure~\oldref{fig:cgt_rates_beta} which shows
the reconstructions for the  values $\beta=-1,\,0.3,\,1,\,1.3$ and
the decreasing rate of convergence with increasing $\beta$ is evident.
The choice of the largest of these values of $\beta$ closely corresponds
to the onset of failure of the iteration scheme in the case of time trace
data on $0\leq t\leq T=1$, and to a smaller combination of $\beta$ and $T$
in the case of final time data.
%\begin{wrapfigure}{r}{0.4\linewidth}
%\centering
%\box\figurelegendtwo
%%\rule{0.9\linewidth}{0.75\linewidth}
%\caption{Legend for figures.}
%\label{fig:legend}
%\end{wrapfigure}

However, and we stress this part, we make no prior assumptions on any
of the functions $f_i$ or $\phi_i$ of analyticity or indeed any form
that would be non-local in nature; the possibility that these reaction
terms could have very different properties in different ranges of their
arguments is a central hypothesis.
Thus simply assuming that small time measurements leading to small values
of $u$ and $v$ and the associated reconstruction of $f_i$ and $\phi_i$
over this range extends to a global reconstruction is too restrictive
in many physical situations.

Returning to Figure~\oldref{fig:cgt_rates_beta} the topmost pair shows the
decrease in $\|f_i^{(n)} - f_i\|$ for $i=1,\,2$
using time trace data on the right boundary point where the solution
is being driven.
For both $f_1$ and $f_2$ the initial approximation was the zero function
showing that a good initial approximation isn't critical here.
For $\beta$ greater than about $1.3$ the scheme no longer converged.
The bottom pair shows the analogous result for final time data.
The scheme converged for $\beta$ less than about $0.3$ with the time
measurement taken to be $T=1$, but if this were reduced, somewhat larger
values of $\beta$ can be used.
The case of $\beta=0.5$ and $T=0.75$ which approximately corresponds
to the maximum value for $T$ with is $\beta$, illustrates this situation.

The actual final reconstructions obtained from the above choices of
$\beta$ differ slightly.
But as Figure~\oldref{fig:cgt_rates_beta} shows the individual iterations
show considerable differences.
Note also that the range of values of $u$ and $v$ depend quite strongly
on $\beta$.
The reconstructions obtained are shown in
Figure~\oldref{fig:titr_reconstructions_f_with_beta}
for the case of time trace data and in
Figure~\oldref{fig:fiti_reconstructions_f_with_beta}
for the case of final time data with parameters as described above.

% values below are important - the x length and y height of each picture
%\xfiglen=2.5 true in  
%\yfiglen=1.6 true in
%
%\input recon_fig_legend
%\copy\figurelegendtwo

\subsection{Reconstructions under added noise}

We show below similar reconstructions but under noise in the data.
Although from a technical perspective of considering only the norms of
the unknowns and the data space these recovery problems appear to be
only mildly ill-conditioned this does not take into account the relevant
constants that are strongly influenced by the nonlinearities in the model.
Thus while good reconstructions can be obtained with a few percent added
uniform noise, much more leads to quite poor reconstructions.
In Figure~\oldref{fig:titr_reconstructions_f_with_noise} we show the same
functions $f_1(u)$ and $f_2(v)$ reconstructed under noise levels of
$0.1\%$ and $1\%$.
These show the quite obvious distinction between the previous reconstructions
where a very mild filter was added of the type noted above to offset
any inaccuracy in the direct data due to truncation error in the direct solver
and filtering to remove active noise.

%  \PiCTeX file for BB5 - time trace data for $f_1$ and $f_2$ under noise
\input colordvi

\input pictex
\font\smallsymbol = cmmi8
%\newdimen\xfiglen \newdimen\yfiglen
\newdimen\xfigdim \newdimen\yfigdim
\newdimen\xmindim \newdimen\ymindim
\newdimen\xmaxdim \newdimen\ymaxdim
\newbox\figurelegendone
\newbox\figureone
\newbox\figuretwo
\newbox\figurethree 
\newbox\figurefour
%\newbox\figurefive
%\newbox\figuresix
%\newbox\figureseven
%\newbox\figureeight
%\newbox\figurenine
%\newbox\figureten
%%%%%%%%%%%%%%%%%%%
%\xfiglen=2.5true in
%\yfiglen=1.6true in
%%%%%%%%%%%%%%%%%%%%

\setbox\figurelegendone=\hbox{
\small
%\eightpoint
\beginpicture
  \setcoordinatesystem units <0.3\xfiglen,0.5\yfiglen> %point at 0 -0.7
  \setplotarea x from 0 to 0.8, y from 0 to 0.7
\linethickness=0.7pt
\ninerm
\setsolid
  \Goldenrod{\relax
  \putrule from 0 0.0 to 0.25 0.0 }\relax
  \put {iteration 10} [l] at 0.35 0.0
  \putrule from 0 0.2 to 0.25 0.2  %\relax
  \Orange{\relax
  \putrule from 0 0.2 to 0.25 0.2 }\relax
  \put {iteration 3} [l] at 0.35 0.2
  \LimeGreen{\relax
  \putrule from 0 0.4 to 0.25 0.4 }\relax
  \putrule from 0 0.4 to 0.2 0.4
  \put {iteration 1} [l] at 0.35 0.4
%\setquadratic
  \setdashes <3pt>
  \Black
  \putrule from 0 0.6 to 0.25 0.6
  \put {actual $q$} [l] at 0.35 0.6
\endpicture
\relax
}
\xfigdim=0.571\xfiglen
\yfigdim=0.625\yfiglen
\setbox\figureone=\vbox{\hsize=\xfiglen
% This is for: time trace data, beta=-1, T=1, f_1(u)
\beginpicture
\footnotesize
  \setcoordinatesystem units <\xfigdim,\yfigdim>  point at 0 -1.6
  \setplotarea x from 0 to 1.75, y from -1.6 to 0
  \axis bottom shiftedto y=-1.6 ticks short numbered from 0 to 1.5 by 0.5 /
% unlabeled short quantity 8 / /
  \axis left ticks short numbered from -1.5 to 0 by 0.5 /
% unlabeled short quantity 16 / /
 \put {{\sevenrm $f_1(u)$}} [lb] at 0.02 0
 \put {{\sevenrm $u$}} [rt] at 1.75 -1.54
\put {\copy\figurelegendtwo} [l] at 0.3 -1
%\setquadratic
\footnotesize
%\setquadratic
\setdashes <3pt>
\Black{        % exact
\plot
         0         0
    0.0484   -0.0784
    0.0968   -0.1404
    0.1452   -0.1873
    0.1935   -0.2205
    0.2419   -0.2414
    0.2903   -0.2512
    0.3387   -0.2514
    0.3871   -0.2434
    0.4355   -0.2284
    0.4838   -0.2079
    0.5322   -0.1831
    0.5806   -0.1555
    0.6290   -0.1265
    0.6774   -0.0973
    0.7258   -0.0694
    0.7742   -0.0440
    0.8225   -0.0226
    0.8709   -0.0065
    0.9193    0.0029
    0.9677    0.0042
    1.0161   -0.0038
    1.0645   -0.0226
    1.1129   -0.0535
    1.1612   -0.0978
    1.2096   -0.1570
    1.2580   -0.2324
    1.3064   -0.3253
    1.3548   -0.4372
    1.4032   -0.5693
    1.4515   -0.7230
    1.4999   -0.8997
    1.5483   -1.1008 
    1.5967   -1.3276
    1.6451   -1.5814
/
}\relax
\setsolid
\Goldenrod{\relax 
\plot
  0.0000   0.0093
  0.0484  -0.0726
  0.0968  -0.1224
  0.1452  -0.1585
  0.1935  -0.1816
  0.2419  -0.1934
  0.2903  -0.1959
  0.3387  -0.1910
  0.3871  -0.1795
  0.4355  -0.1630
  0.4839  -0.1407
  0.5322  -0.1132
  0.5806  -0.0823
  0.6290  -0.0511
  0.6774  -0.0204
  0.7258   0.0095
  0.7742   0.0370
  0.8226   0.0608
  0.8710   0.0808
  0.9193   0.0962
  0.9677   0.1058
  1.0161   0.1101
  1.0645   0.1054
  1.1129   0.0900
  1.1613   0.0684
  1.2580  -0.0045
  1.3064  -0.0522
  1.3548  -0.1139
  1.4032  -0.1909
  1.4516  -0.2822
  1.5000  -0.3763
  1.5484  -0.4846
  1.5967  -0.6441
  1.6451  -0.8988
/
 }\relax
\Orange{\relax 
\plot
  0.0000   0.0095
  0.0484  -0.0785
  0.0968  -0.1364
  0.1452  -0.1814
  0.1935  -0.2124
  0.2419  -0.2307
  0.2903  -0.2382
  0.3387  -0.2369
  0.3871  -0.2275
  0.4355  -0.2120
  0.4839  -0.1894
  0.5322  -0.1605
  0.5806  -0.1279
  0.6290  -0.0951
  0.6774  -0.0629
  0.7258  -0.0319
  0.7742  -0.0037
  0.8226   0.0200
  0.8710   0.0389
  0.9193   0.0522
  0.9677   0.0580
  1.0161   0.0564
  1.0645   0.0430
  1.1129   0.0159
  1.1613  -0.0203
  1.2097  -0.0700
  1.2580  -0.1320
  1.3064  -0.2056
  1.3548  -0.2994
  1.4032  -0.4140
  1.4516  -0.5475
  1.5000  -0.6956
  1.5484  -0.8863
  1.5967  -1.1233
/
 }\relax
\Red{\relax 
\plot
  0.0000   0.0095
  0.0484  -0.0790
  0.0968  -0.1381
  0.1452  -0.1847
  0.1935  -0.2176
  0.2419  -0.2378
  0.2903  -0.2470
  0.3387  -0.2474
  0.3871  -0.2396
  0.4355  -0.2255
  0.4839  -0.2043
  0.5322  -0.1767
  0.5806  -0.1455
  0.6290  -0.1141
  0.6774  -0.0833
  0.7258  -0.0539
  0.7742  -0.0273
  0.8226  -0.0051
  0.8710   0.0121
  0.9193   0.0234
  0.9677   0.0271
  1.0161   0.0227
  1.0645   0.0058
  1.1129  -0.0257
  1.1613  -0.0674
  1.2097  -0.1239
  1.2580  -0.1938
  1.3064  -0.2764
  1.3548  -0.3805
  1.4032  -0.5084
  1.4516  -0.6597
  1.5000  -0.8263
  1.5484  -1.0281
  1.5967   -1.35
  1.6451   -1.6
/
 }\relax
\LimeGreen{\relax 
\plot
  0.0000   0.0095
  0.0484  -0.0790
  0.0968  -0.1383
  0.1452  -0.1852
  0.1935  -0.2184
  0.2419  -0.2391
  0.2903  -0.2489
  0.3387  -0.2497
  0.3871  -0.2425
  0.4355  -0.2290
  0.4839  -0.2085
  0.5322  -0.1817
  0.5806  -0.1513
  0.6290  -0.1207
  0.6774  -0.0909
  0.7258  -0.0623
  0.7742  -0.0366
  0.8226  -0.0153
  0.8710   0.0012
  0.9193   0.0118
  0.9677   0.0146
  1.0161   0.0094
  1.0645  -0.0084
  1.1129  -0.0414
  1.1613  -0.0848
  1.2097  -0.1433
  1.2580  -0.2157
  1.3064  -0.3006
  1.3548  -0.4075
  1.4032  -0.5396
  1.4516  -0.6961
  1.5000  -0.8648
  1.5484  -1.0655
  1.5967  -1.3496
  1.6451  -1.7556
/
 }\relax
\PineGreen{\relax 
\plot
  0.0000   0.0095
  0.0484  -0.0790
  0.0968  -0.1384
  0.1452  -0.1853
  0.1935  -0.2186
  0.2419  -0.2394
  0.2903  -0.2492
  0.3387  -0.2503
  0.3871  -0.2432
  0.4355  -0.2299
  0.4839  -0.2096
  0.5322  -0.1832
  0.5806  -0.1531
  0.6290  -0.1229
  0.6774  -0.0934
  0.7258  -0.0652
  0.7742  -0.0399
  0.8226  -0.0188
  0.8710  -0.0026
  0.9193   0.0077
  0.9677   0.0104
  1.0161   0.0049
  1.0645  -0.0131
  1.1129  -0.0464
  1.1613  -0.0903
  1.2097  -0.1495
  1.2580  -0.2225
  1.3064  -0.3080
  1.3548  -0.4157
  1.4032  -0.5489
  1.4516  -0.7065
  1.5000  -0.8745
  1.5484  -1.0734
  1.5967  -1.3638
  1.6451  -1.7956
/
 }\relax
\endpicture
}   % end of time trace data for f1:  $\beta=-1$, $T=1$ noise = 0.1%
\xfigdim=0.40\xfiglen
\yfigdim=0.37\yfiglen
\setbox\figuretwo=\vbox{\hsize=\xfiglen
\beginpicture
\footnotesize
  \setcoordinatesystem units <\xfigdim,\yfigdim>  point at 0 -0.70
  \setplotarea x from 0.0 to 2.4, y from -0.7 to 2
  \axis bottom shiftedto y=-0.7 ticks short numbered from 0 to 2 by 0.5 /
  \axis left ticks short numbered from -0.5 to 2 by 0.5 /
 \put {{\sevenrm $f_2(u)$}} [lb] at 0 2
 \put {{\sevenrm $u$}} [rt] at 2.4 -0.74
\setquadratic
\setdashes <3pt>
%\Black{
\plot
        0    0.0135
    0.03      0.02
    0.0736    0.0268
    0.1472    0.0505
    0.2208    0.0912
    0.2944    0.1572
    0.3680    0.2579
    0.4416    0.4011
    0.5152    0.5909
    0.5888    0.8240
    0.6624    1.0872
    0.7360    1.3572
    0.8096    1.6028
    0.8832    1.7901
    0.9568    1.8899
    1.0304    1.8846
    1.1040    1.7729
    1.1776    1.5697
    1.2512    1.3025
    1.3247    1.0049
    1.3983    0.7091
    1.4719    0.4401
    1.5455    0.2127
    1.6191    0.0320
    1.6927   -0.1050
    1.7663   -0.2059
    1.8399   -0.2798
    1.9135   -0.3353
    1.9871   -0.3796
    2.0607   -0.4174
    2.1343   -0.4523
    2.2079   -0.4861
    2.2815   -0.5200
    2.3551   -0.5544
    2.4287   -0.5898
%   2.5023   -0.6261
/
%}
%
\setquadratic
\setsolid
\Goldenrod{\relax 
\plot
  0.0000   0.0177
    0.03      0.02
  0.0736   0.0240
  0.1472   0.0423
  0.2208   0.0732
  0.2944   0.1242
  0.3680   0.2044
  0.4416   0.3220
  0.5152   0.4778
  0.5888   0.6620
  0.6624   0.8575
  0.7360   1.0445
  0.8096   1.1999
  0.8832   1.3079
  0.9568   1.3586
  1.0304   1.3459
  1.1040   1.2697
  1.1776   1.1333
  1.2512   0.9515
  1.3248   0.7417
  1.3984   0.5264
  1.4720   0.3229
  1.5456   0.1492
  1.6192   0.0070
  1.6928  -0.1046
  1.7663  -0.1860
  1.8399  -0.2459
  1.9135  -0.3028
  1.9871  -0.3537
  2.0607  -0.4017
  2.1343  -0.4285
  2.2079  -0.4444
  2.2815  -0.4691
  2.3551  -0.5032
  2.4287  -0.6
% 2.5023  -0.65
/
 }\relax
\Orange{\relax 
\plot
  0.0000   0.0176
    0.03      0.02
  0.0736   0.0258
  0.1472   0.0480
  0.2208   0.0853
  0.2944   0.1470
  0.3680   0.2445
  0.4416   0.3885
  0.5152   0.5803
  0.5888   0.8084
  0.6624   1.0522
  0.7360   1.2874
  0.8096   1.4862
  0.8832   1.6296
  0.9568   1.7054
  1.0304   1.7052
  1.1040   1.6279
  1.1776   1.4756
  1.2512   1.2631
  1.3248   1.0094
  1.3984   0.7407
  1.4720   0.4788
  1.5456   0.2502
  1.6192   0.0610
  1.6928  -0.0863
  1.7663  -0.1923
  1.8399  -0.2665
  1.9135  -0.3331
  1.9871  -0.3935
  2.0607  -0.4518
  2.1343  -0.4857
  2.2079  -0.5032
  2.2815  -0.5284
  2.3551  -0.5680
  2.4287  -0.61
/
 }\relax
\Red{\relax 
\plot
  0.0000   0.0176
    0.03      0.02
  0.0736   0.0260
  0.1472   0.0487
  0.2208   0.0870
  0.2944   0.1508
  0.3680   0.2522
  0.4416   0.4027
  0.5152   0.6041
  0.5888   0.8444
  0.6624   1.1018
  0.7360   1.3504
  0.8096   1.5604
  0.8832   1.7118
  0.9568   1.7916
  1.0304   1.7907
  1.1040   1.7077
  1.1776   1.5450
  1.2512   1.3183
  1.3248   1.0484
  1.3984   0.7640
  1.4720   0.4886
  1.5456   0.2503
  1.6192   0.0555
  1.6928  -0.0935
  1.7663  -0.1978
  1.8399  -0.2691
  1.9135  -0.3339
  1.9871  -0.3931
  2.0607  -0.4498
  2.1343  -0.4797
  2.2079  -0.4930
  2.2815  -0.5152
  2.3551  -0.5575
  2.4287  -0.58
/
 }\relax
\LimeGreen{\relax 
\plot
 0.0000   0.0176
    0.03      0.02
  0.0736   0.0260
  0.1472   0.0488
  0.2208   0.0873
  0.2944   0.1514
  0.3680   0.2537
  0.4416   0.4057
  0.5152   0.6097
  0.5888   0.8533
  0.6624   1.1146
  0.7360   1.3669
  0.8096   1.5798
  0.8832   1.7328
  0.9568   1.8126
  1.0304   1.8100
  1.1040   1.7237
  1.1776   1.5562
  1.2512   1.3239
  1.3248   1.0484
  1.3984   0.7594
  1.4720   0.4812
  1.5456   0.2422
  1.6192   0.0486
  1.6928  -0.0978
  1.7663  -0.1985
  1.8399  -0.2666
  1.9135  -0.3290
  1.9871  -0.3864
  2.0607  -0.4406
  2.1343  -0.4670
  2.2079  -0.4768
  2.2815  -0.4967
  2.3551  -0.5419
  2.4287  -0.6
/
 }\relax
\PineGreen{\relax 
\plot
  0.0000   0.0176
    0.03      0.02
  0.0736   0.0260
  0.1472   0.0488
  0.2208   0.0873
  0.2944   0.1515
  0.3680   0.2539
  0.4416   0.4064
  0.5152   0.6110
  0.5888   0.8556
  0.6624   1.1178
  0.7360   1.3711
  0.8096   1.5847
  0.8832   1.7381
  0.9568   1.8177
  1.0304   1.8142
  1.1040   1.7266
  1.1776   1.5573
  1.2512   1.3231
  1.3248   1.0458
  1.3984   0.7556
  1.4720   0.4770
  1.5456   0.2384
  1.6192   0.0459
  1.6928  -0.0989
  1.7663  -0.1976
  1.8399  -0.2640
  1.9135  -0.3250
  1.9871  -0.3813
  2.0607  -0.4339
  2.1343  -0.4583
  2.2079  -0.4661
  2.2815  -0.4849
  2.3551  -0.5326
  2.4287  -0.62
/
 }\relax
\endpicture
}   % end of time trace data for f2:  $\beta=-1$, $T=1$
\xfigdim=0.571\xfiglen
\yfigdim=0.625\yfiglen
\setbox\figurethree=\vbox{\hsize=\xfiglen
% This is for: time trace data, beta=1, T=1, f_1(u)   1% noise
\beginpicture
\footnotesize
  \setcoordinatesystem units <\xfigdim,\yfigdim>  point at 0 -1.6
  \setplotarea x from 0 to 1.75, y from -1.6 to 0
  \axis bottom shiftedto y=-1.6 ticks short numbered from 0 to 1.5 by 0.5 /
% unlabeled short quantity 8 / /
  \axis left ticks short numbered from -1.5 to 0 by 0.5 /
% unlabeled short quantity 16 / /
 \put {{\sevenrm $f_1(u)$}} [lb] at 0.02 0
 \put {{\sevenrm $u$}} [rt] at 1.75 -1.54
\setquadratic
\footnotesize
\setquadratic
\setdashes <3pt>
\Black{
\plot
         0         0
    0.0484   -0.0784
    0.0968   -0.1404
    0.1452   -0.1873
    0.1935   -0.2205
    0.2419   -0.2414
    0.2903   -0.2512
    0.3387   -0.2514
    0.3871   -0.2434
    0.4355   -0.2284
    0.4838   -0.2079
    0.5322   -0.1831
    0.5806   -0.1555
    0.6290   -0.1265
    0.6774   -0.0973
    0.7258   -0.0694
    0.7742   -0.0440
    0.8225   -0.0226
    0.8709   -0.0065
    0.9193    0.0029
    0.9677    0.0042
    1.0161   -0.0038
    1.0645   -0.0226
    1.1129   -0.0535
    1.1612   -0.0978
    1.2096   -0.1570
    1.2580   -0.2324
    1.3064   -0.3253
    1.3548   -0.4372
    1.4032   -0.5693
    1.4515   -0.7230
    1.4999   -0.8997
    1.5483   -1.1008 
    1.5967   -1.3276
    1.6451   -1.5814
/
}
\Goldenrod{\relax 
\setsolid
\plot
  0.0000   0.0338
  0.0484  -0.0835
  0.0968  -0.1259
  0.1452  -0.1580
  0.1937  -0.1800
  0.2421  -0.1935
  0.2905  -0.1979
  0.3389  -0.1955
  0.3873  -0.1847
  0.4357  -0.1678
  0.4842  -0.1448
  0.5326  -0.1130
  0.5810  -0.0779
  0.6294  -0.0436
  0.6778  -0.0087
  0.7262   0.0208
  0.7747   0.0478
  0.8231   0.0679
  0.8715   0.0843
  0.9199   0.0976
  0.9683   0.0997
  1.0167   0.1013
  1.0652   0.0996
  1.1136   0.0663
  1.1620   0.0468
  1.2104   0.0221
  1.2588  -0.0234
  1.3072  -0.0321
  1.3557  -0.0751
  1.4041  -0.1228
  1.4525  -0.2202
  1.5009  -0.3338
  1.5493  -0.5470
  1.5977  -0.9617
  1.6462  -1.4
/
 }\relax
\Orange{\relax 
\plot
  0.0000   0.0340
  0.0484  -0.0900
  0.0968  -0.1394
  0.1452  -0.1797
  0.1937  -0.2097
  0.2421  -0.2304
  0.2905  -0.2407
  0.3389  -0.2428
  0.3873  -0.2346
  0.4357  -0.2188
  0.4842  -0.1950
  0.5326  -0.1609
  0.5810  -0.1229
  0.6294  -0.0858
  0.6778  -0.0483
  0.7262  -0.0172
  0.7747   0.0105
  0.8231   0.0297
  0.8715   0.0437
  0.9199   0.0529
  0.9683   0.0492
  1.0167   0.0422
  1.0652   0.0297
  1.1136  -0.0180
  1.1620  -0.0500
  1.2104  -0.0893
  1.2588  -0.1487
  1.3072  -0.1712
  1.3557  -0.2356
  1.4041  -0.3154
  1.4525  -0.4749
  1.5009  -0.7031
  1.5493  -1.1187
  1.5977  -1.4
  1.6462  -1.6
/
 }\relax
\Red{\relax 
\plot
  0.0000   0.0340
  0.0484  -0.0905
  0.0968  -0.1410
  0.1452  -0.1829
  0.1937  -0.2147
  0.2421  -0.2375
  0.2905  -0.2497
  0.3389  -0.2537
  0.3873  -0.2473
  0.4357  -0.2330
  0.4842  -0.2106
  0.5326  -0.1776
  0.5810  -0.1406
  0.6294  -0.1044
  0.6778  -0.0680
  0.7262  -0.0380
  0.7747  -0.0120
  0.8231   0.0054
  0.8715   0.0173
  0.9199   0.0237
  0.9683   0.0170
  1.0167   0.0063
  1.0652  -0.0099
  1.1136  -0.0641
  1.1620  -0.1017
  1.2104  -0.1473
  1.2588  -0.2139
  1.3072  -0.2400
  1.3557  -0.3094
  1.4041  -0.4003
  1.4525  -0.5850
  1.5009  -0.8344
  1.5493  -1.2562
  1.5977  -1.45
  1.6462  -1.64
/
 }\relax
\LimeGreen{\relax 
\plot
  0.0000   0.0340
  0.0484  -0.0905
  0.0968  -0.1412
  0.1452  -0.1834
  0.1937  -0.2156
  0.2421  -0.2389
  0.2905  -0.2515
  0.3389  -0.2562
  0.3873  -0.2503
  0.4357  -0.2367
  0.4842  -0.2150
  0.5326  -0.1828
  0.5810  -0.1465
  0.6294  -0.1109
  0.6778  -0.0753
  0.7262  -0.0461
  0.7747  -0.0210
  0.8231  -0.0045
  0.8715   0.0065
  0.9199   0.0118
  0.9683   0.0041
  1.0167  -0.0077
  1.0652  -0.0250
  1.1136  -0.0815
  1.1620  -0.1216
  1.2104  -0.1699
  1.2588  -0.2394
  1.3072  -0.2659
  1.3557  -0.3356
  1.4041  -0.4280
  1.4525  -0.6137
  1.5009  -0.8523
  1.5493  -1.2544
  1.61    -1.46
  1.6462  -1.7
/
 }\relax
\PineGreen{\relax 
\plot
  0.0000   0.0340
  0.0484  -0.0905
  0.0968  -0.1412
  0.1452  -0.1834
  0.1937  -0.2157
  0.2421  -0.2391
  0.2905  -0.2519
  0.3389  -0.2567
  0.3873  -0.2511
  0.4357  -0.2377
  0.4842  -0.2162
  0.5326  -0.1843
  0.5810  -0.1483
  0.6294  -0.1131
  0.6778  -0.0777
  0.7262  -0.0489
  0.7747  -0.0241
  0.8231  -0.0080
  0.8715   0.0026
  0.9199   0.0076
  0.9683  -0.0003
  1.0167  -0.0125
  1.0652  -0.0301
  1.1136  -0.0874
  1.1620  -0.1284
  1.2104  -0.1777
  1.2588  -0.2483
  1.3072  -0.2750
  1.3557  -0.3444
  1.4041  -0.4358
  1.4525  -0.6177
  1.5009  -0.8462
  1.5493  -1.2355
  1.61    -1.6
  1.6462  -1.8
/
 }\relax
\endpicture
}   % end of time trace data for f1:  $\beta=0.3$, $T=1$
\xfigdim=0.40\xfiglen
\yfigdim=0.37\yfiglen
\setbox\figurefour=\vbox{\hsize=\xfiglen
\beginpicture
\footnotesize
  \setcoordinatesystem units <\xfigdim,\yfigdim>  point at 0 -0.70
  \setplotarea x from 0.0 to 2.4, y from -0.7 to 2
  \axis bottom shiftedto y=-0.7 ticks short numbered from 0 to 2 by 0.5 /
  \axis left ticks short numbered from -0.5 to 2 by 0.5 /
 \put {{\sevenrm $f_2(u)$}} [lb] at 0 2
 \put {{\sevenrm $u$}} [rt] at 2.4 -0.74
\setquadratic
\setlinear
\setdashes <3pt>
\plot
        0    0.0135
    0.0736    0.0268
    0.1472    0.0505
    0.2208    0.0912
    0.2944    0.1572
    0.3680    0.2579
    0.4416    0.4011
    0.5152    0.5909
    0.5888    0.8240
    0.6624    1.0872
    0.7360    1.3572
    0.8096    1.6028
    0.8832    1.7901
    0.9568    1.8899
    1.0304    1.8846
    1.1040    1.7729
    1.1776    1.5697
    1.2512    1.3025
    1.3247    1.0049
    1.3983    0.7091
    1.4719    0.4401
    1.5455    0.2127
    1.6191    0.0320
    1.6927   -0.1050
    1.7663   -0.2059
    1.8399   -0.2798
    1.9135   -0.3353
    1.9871   -0.3796
    2.0607   -0.4174
    2.1343   -0.4523
    2.2079   -0.4861
    2.2815   -0.5200
    2.3551   -0.5544
    2.4287   -0.5898 /
%    2.5023   -0.6261
%
\setlinear
\setsolid
\Goldenrod{\relax 
\plot
  0.0000   0.0512
  0.0736  -0.0039
  0.1472   0.0200
  0.2208   0.0751
  0.2944   0.1634
  0.3680   0.2815
  0.4416   0.4266
  0.5152   0.5828
  0.5888   0.7364
  0.6624   0.8763
  0.7360   0.9965
  0.8096   1.0809
  0.8833   1.1352
  0.9569   1.1573
  1.0305   1.1516
  1.1041   1.1122
  1.1777   1.0252
  1.2513   0.9102
  1.3249   0.7768
  1.3985   0.6129
  1.4721   0.4204
  1.5457   0.2572
  1.6193   0.0892
  1.6929  -0.0446
  1.7665  -0.1676
  1.8401  -0.2153
  1.9137  -0.2842
  1.9873  -0.3071
  2.0609  -0.3320
  2.1345  -0.3514
  2.2081  -0.3955
  2.2817  -0.53
  2.3553  -0.65
/
 }\relax
\setlinear
\Orange{\relax 
\plot
  0.0000   0.0512
  0.0736  -0.0040
  0.1472   0.0248
  0.2208   0.0894
  0.2944   0.1941
  0.3680   0.3364
  0.4416   0.5138
  0.5152   0.7072
  0.5888   0.8988
  0.6624   1.0757
  0.7360   1.2302
  0.8096   1.3412
  0.8833   1.4171
  0.9569   1.4567
  1.0305   1.4605
  1.1041   1.4241
  1.1777   1.3319
  1.2513   1.2032
  1.3249   1.0477
  1.3985   0.8523
  1.4721   0.6143
  1.5457   0.4056
  1.6193   0.1886
  1.6929   0.0105
  1.7665  -0.1539
  1.8401  -0.2234
  1.9137  -0.3049
  1.9873  -0.3339
  2.0609  -0.3592
  2.1345  -0.3746
  2.2081  -0.4223
  2.2817  -0.55
  2.3553  -0.7
/
 }\relax
\Red{\relax 
\plot
  0.0000   0.0512
  0.0736  -0.0040
  0.1472   0.0253
  0.2208   0.0914
  0.2944   0.1992
  0.3680   0.3469
  0.4416   0.5323
  0.5152   0.7356
  0.5888   0.9377
  0.6624   1.1247
  0.7360   1.2885
  0.8096   1.4062
  0.8833   1.4871
  0.9569   1.5301
  1.0305   1.5351
  1.1041   1.4978
  1.1777   1.4016
  1.2513   1.2662
  1.3249   1.1016
  1.3985   0.8948
  1.4721   0.6432
  1.5457   0.4234
  1.6193   0.1961
  1.6929   0.0103
  1.7665  -0.1570
  1.8401  -0.2226
  1.9137  -0.3004
  1.9873  -0.3236
  2.0609  -0.3411
  2.1345  -0.3520
  2.2081  -0.4082
  2.2817  -0.5151
  2.3553  -0.6
/
 }\relax
\LimeGreen{\relax 
\plot
  0.0000   0.0512
  0.0736  -0.0040
  0.1472   0.0254
  0.2208   0.0917
  0.2944   0.2001
  0.3680   0.3490
  0.4416   0.5364
  0.5152   0.7422
  0.5888   0.9471
  0.6624   1.1369
  0.7360   1.3031
  0.8096   1.4224
  0.8833   1.5045
  0.9569   1.5480
  1.0305   1.5528
  1.1041   1.5147
  1.1777   1.4164
  1.2513   1.2782
  1.3249   1.1101
  1.3985   0.8990
  1.4721   0.6431
  1.5457   0.4204
  1.6193   0.1915
  1.6929   0.0057
  1.7665  -0.1583
  1.8401  -0.2199
  1.9137  -0.2931
  1.9873  -0.3114
  2.0609  -0.3227
  2.1345  -0.3322
  2.2081  -0.3977
  2.2817  -0.5
  2.3553  -0.65
/
 }\relax
\PineGreen{\relax 
\plot
  0.0000   0.0512
  0.0736  -0.0040
  0.1472   0.0254
  0.2208   0.0918
  0.2944   0.2002
  0.3680   0.3494
  0.4416   0.5372
  0.5152   0.7437
  0.5888   0.9494
  0.6624   1.1399
  0.7360   1.3067
  0.8096   1.4265
  0.8833   1.5088
  0.9569   1.5523
  1.0305   1.5569
  1.1041   1.5184
  1.1777   1.4193
  1.2513   1.2801
  1.3249   1.1108
  1.3985   0.8983
  1.4721   0.6411
  1.5457   0.4177
  1.6193   0.1888
  1.6929   0.0038
  1.7665  -0.1580
  1.8401  -0.2173
  1.9137  -0.2879
  1.9873  -0.3034
  2.0609  -0.3115
  2.1345  -0.3212
  2.2081  -0.3940
  2.2817  -0.55
  2.3553  -0.7
/
 }\relax
\endpicture
}   % end of time trace data for f2:  $\beta=-1$, $T=1$

\begin{figure}[h]
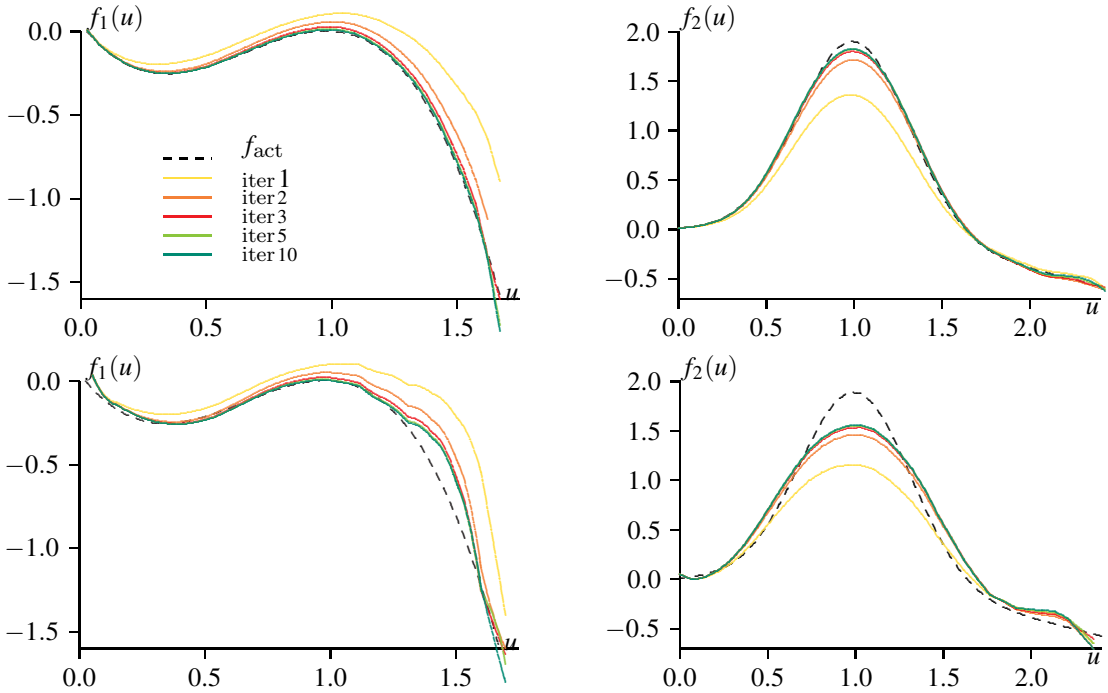

\hbox to\hsize{\hss\copy\figureone\hss\hss\hss\copy\figuretwo\hss}
\medskip
\hbox to\hsize{\hss\copy\figurethree\hss\hss\hss\copy\figurefour\hss}
\medskip
%\hbox to\hsize{\hss\copy\figureeleven\hss\hss\hss\copy\figuretwelve\hss}
\caption{\small Reconstructions of $f_1$ and $f_2$ under noise from time trace data.
Top: 0.1\%,\ Bottom: 1\%.}
\label{fig:titr_reconstructions_f_with_noise}
\end{figure}

\subsection{Reconstructions of the interaction terms $\phi_i$}\label{sec:phi1phi2}

In this section we assume that both $f_1$ and $f_2$ are known,
but $\phi_1(w)$ and $\phi_2(w)$ in 
\begin{equation}\label{phi1phi2}
\begin{aligned}
D_t u - \triangle u &= f_1(u)+\beta^u \phi_1(w(u,v)) +r_u\\
D_t v - \triangle v &= f_2(v)+\beta^v \phi_2(w(u,v)) +r_v 
\end{aligned}
\end{equation}
are unknown and have to be recovered from either a pair of time trace
or a pair of final time data at $t=T$.
The sample values to be reconstructed are
\begin{equation}\label{egn:phi_act_funcs}
\begin{aligned}
\phi_1(w) &= (\mbox{arctan}(w) + 2w.e^{-(w-1).^2}) \\
\phi_2(w) &= \begin{cases}
 0.1\cdot(27 - (3-w)^2(3+2w))&\mbox{ if }w<3\\2.7&\mbox{ else}\end{cases} 
\end{aligned}
\end{equation}
where $\beta_u$ and $\beta_v$ are constants that will be used to test
how large these functions can be and the iteration scheme still converge.
The known functions in the model were set to
\[
\begin{aligned}
&f_1(u) = u(1-u)\,, \quad
&&f_2(v) = v(2-v)\\ 
&r_u(x,t) = 10\sin(\tfrac{\pi}{2}x)\cdot t\,, \quad 
&&r_v(x,t) = 12(2x - x^2)\cdot t
\end{aligned}
\]

We show the reconstructions of $\phi_1$ and $\phi_2$
graphically by again displaying four pairs of
reconstruction iterates, for each of the values
$\beta=-1,\;0.1,\;1,\;10$, in Figure~\oldref{fig:titr_reconstructions_phi}.
Of course different values of $\beta$ give different solution with differing
ranges and this is clearly shown.
The exact value is the dotted curve, the curve in yellow is the first
iterate, the one in orange is the second, and the one in red the third or fifth.
This shows the rapid convergence for a wide range of $\beta$ values.
\xfiglen=2.3 true in  
\yfiglen=1.4 true in
%  \PiCTeX file for BB5   Final time recontructions
\input colordvi

\input pictex
\font\smallsymbol = cmmi8
%\newdimen\xfiglen \newdimen\yfiglen
\newdimen\xfigdim \newdimen\yfigdim
\newdimen\xmindim \newdimen\ymindim
\newdimen\xmaxdim \newdimen\ymaxdim
\newbox\figurelegendten
\newbox\figurethirteen
\newbox\figurefourteen
\newbox\figurefifteen
\newbox\figuresixteen
\newbox\figureseventeen
\newbox\figureeighteen
\newbox\figurenineteen
\newbox\figuretwenty
%%%%%%%%%%%%%%%%%%%
%\xfiglen=2.5true in
%\yfiglen=1.6true in
%%%%%%%%%%%%%%%%%%%%
\xfigdim=0.21\xfiglen
\yfigdim=0.29\yfiglen
\setbox\figurethirteen=\vbox{\hsize=\xfiglen
% This is for: final time data for phi1:  $\betau=\betav-1$, $T=1$
\beginpicture
\footnotesize
  \setcoordinatesystem units <\xfigdim,\yfigdim>  point at 0 0
  \setplotarea x from 0 to 4.75, y from 0 to 3.7
  \axis bottom shiftedto y=0 ticks short numbered from 0 to 4 by 1 /
% unlabeled short quantity 8 / /
  \axis left ticks short numbered from 0 to 3 by 1 /
% unlabeled short quantity 16 / /
 \put {{\sevenrm $\phi_1(w)$}} [l] at 0.04 3.7
 \put {{\sevenrm $w$}} [rb] at 4.75 0.05
 \put {$\beta=-1$} [rt] at 4.5 3.5
%\setquadratic
\setdashes <3pt>
\Black{
\plot
  0.0000   0.0000
  0.1362   0.2644
  0.2723   0.5866
  0.4085   0.9635
  0.5446   1.3839
  0.6808   1.8273
  0.8169   2.2649
  0.9531   2.6633
  1.0892   2.9892
  1.2254   3.2157
  1.3615   3.3267
  1.4977   3.3203
  1.6338   3.2082
  1.7700   3.0132
  1.9061   2.7649
  2.0423   2.4938
  2.1784   2.2271
  2.3146   1.9852
  2.4507   1.7809
  2.5869   1.6190
  2.7230   1.4986
  2.8592   1.4147
  2.9953   1.3604
  3.1315   1.3283
  3.2676   1.3120
  3.4038   1.3061
  3.5399   1.3067
  3.6761   1.3109
  3.8122   1.3171
  3.9484   1.3241
  4.0845   1.3313
  4.2207   1.3384
  4.3568   1.3453
  4.4930   1.3518
  4.6291   1.3581
/
 }\relax
\setsolid
\Goldenrod{\relax 
\plot
  0.0000  -0.0142
  0.1362   0.3345
  0.2723   0.6880
  0.4085   1.0956
  0.5446   1.5451
  0.6808   2.0164
  0.8169   2.4830
  0.9531   2.9123
  1.0892   3.2708
  1.2254   3.5318
  1.3615   3.6798
  1.4977   3.7120
  1.6338   3.6390
  1.7700   3.4820
  1.9061   3.2689
  2.0423   3.0290
  2.1784   2.7881
  2.3146   2.5658
  2.4507   2.3744
  2.5869   2.2200
  2.7230   2.1036
  2.8592   2.0230
  2.9953   1.9730
  3.1315   1.9465
  3.2676   1.9356
  3.4038   1.9333
  3.5399   1.9349
  3.6761   1.9391
  3.8122   1.9466
  3.9484   1.9582
  4.0845   1.9711
  4.2207   1.9801
  4.3568   1.9842
  4.4930   1.9944
  4.6291   1.9958
/
 }\relax
\Orange{\relax 
\plot
  0.0000  -0.0142
  0.1362   0.2989
  0.2723   0.6339
  0.4085   1.0236
  0.5446   1.4538
  0.6808   1.9036
  0.8169   2.3461
  0.9531   2.7482
  1.0892   3.0769
  1.2254   3.3057
  1.3615   3.4198
  1.4977   3.4176
  1.6338   3.3104
  1.7700   3.1205
  1.9061   2.8766
  2.0423   2.6087
  2.1784   2.3431
  2.3146   2.0995
  2.4507   1.8901
  2.5869   1.7207
  2.7230   1.5923
  2.8592   1.5020
  2.9953   1.4443
  3.1315   1.4115
  3.2676   1.3954
  3.4038   1.3886
  3.5399   1.3862
  3.6761   1.3867
  3.8122   1.3910
  3.9484   1.3994
  4.0845   1.4094
  4.2207   1.4151
  4.3568   1.4156
  4.4930   1.4221
  4.6291   1.4206
/
 }\relax
\Red{\relax 
\plot
  0.0000  -0.0142
  0.1362   0.2662
  0.2723   0.5875
  0.4085   0.9666
  0.5446   1.3879
  0.6808   1.8298
  0.8169   2.2652
  0.9531   2.6609
  1.0892   2.9838
  1.2254   3.2076
  1.3615   3.3173
  1.4977   3.3115
  1.6338   3.2015
  1.7700   3.0097
  1.9061   2.7646
  2.0423   2.4962
  2.1784   2.2306
  2.3146   1.9874
  2.4507   1.7787
  2.5869   1.6101
  2.7230   1.4825
  2.8592   1.3930
  2.9953   1.3359
  3.1315   1.3037
  3.2676   1.2881
  3.4038   1.2816
  3.5399   1.2795
  3.6761   1.2803
  3.8122   1.2848
  3.9484   1.2935
  4.0845   1.3036
  4.2207   1.3095
  4.3568   1.3106
  4.4930   1.3188
  4.6291   1.3195
/
 }\relax
\LimeGreen{\relax 
\plot
  0.0000  -0.0142
  0.1362   0.2623
  0.2723   0.5823
  0.4085   0.9604
  0.5446   1.3810
  0.6808   1.8224
  0.8169   2.2574
  0.9531   2.6529
  1.0892   2.9757
  1.2254   3.1995
  1.3615   3.3097
  1.4977   3.3046
  1.6338   3.1956
  1.7700   3.0051
  1.9061   2.7616
  2.0423   2.4948
  2.1784   2.2310
  2.3146   1.9895
  2.4507   1.7824
  2.5869   1.6154
  2.7230   1.4891
  2.8592   1.4007
  2.9953   1.3446
  3.1315   1.3131
  3.2676   1.2980
  3.4038   1.2920
  3.5399   1.2904
  3.6761   1.2917
  3.8122   1.2966
  3.9484   1.3058
  4.0845   1.3163
  4.2207   1.3230
  4.3568   1.3253
  4.4930   1.3348
  4.6291   1.3334
/
 }\relax
%
%\PineGreen{\relax 
%\plot
%/
% }\relax
%
%\Blue{\relax 
%\plot
%/
% }\relax
%
\endpicture
}   % end of final time data for phi1:  $\betau=\betav-1$, $T=1$
\xfigdim=0.21\xfiglen
\yfigdim=0.34\yfiglen
\setbox\figurefourteen=\vbox{\hsize=\xfiglen
% This is for: final time data for phi2:  $\betau=\betav-1$, $T=1$
\beginpicture
\footnotesize
  \setcoordinatesystem units <\xfigdim,\yfigdim>  point at 0 0
  \setplotarea x from 0 to 4.75, y from 0 to 3
  \axis bottom shiftedto y=0 ticks short numbered from 0 to 4 by 1 /
% unlabeled short quantity 8 / /
  \axis left ticks short numbered from 0 to 3 by 1 /
% unlabeled short quantity 16 / /
 \put {{\sevenrm $\phi_2(w)$}} [l] at 0.04 3
 \put {{\sevenrm $w$}} [rb] at 4.75 0.05
%\setquadratic
\setdashes <3pt>
\Black{
\plot
  0.0000   0.0000
  0.1362   0.0162
  0.2723   0.0627
  0.4085   0.1365
  0.5446   0.2346
  0.6808   0.3540
  0.8169   0.4916
  0.9531   0.6443
  1.0892   0.8093
  1.2254   0.9834
  1.3615   1.1636
  1.4977   1.3468
  1.6338   1.5302
  1.7700   1.7105
  1.9061   1.8849
  2.0423   2.0502
  2.1784   2.2034
  2.3146   2.3416
  2.4507   2.4616
  2.5869   2.5605
  2.7230   2.6352
  2.8592   2.6827
  2.9953   2.7000
  3.1315   2.7000
  3.2676   2.7000
  3.4038   2.7000
  3.5399   2.7000
  3.6761   2.7000
  3.8122   2.7000
  3.9484   2.7000
  4.0845   2.7000
  4.2207   2.7000
  4.3568   2.7000
  4.4930   2.7000
  4.6291   2.7000
/
 }\relax
\setsolid
\Goldenrod{\relax 
\plot
  0.0000  -0.0175
  0.1362  -0.1335
  0.2723  -0.1538
  0.4085  -0.1249
  0.5446  -0.0649
  0.6808   0.0156
  0.8169   0.1172
  0.9531   0.2411
  1.0892   0.3837
  1.2254   0.5388
  1.3615   0.7003
  1.4977   0.8647
  1.6338   1.0304
  1.7700   1.1966
  1.9061   1.3615
  2.0423   1.5221
  2.1784   1.6741
  2.3146   1.8123
  2.4507   1.9320
  2.5869   2.0292
  2.7230   2.1020
  2.8592   2.1510
  2.9953   2.1800
  3.1315   2.1946
  3.2676   2.2011
  3.4038   2.2041
  3.5399   2.2062
  3.6761   2.2083
  3.8122   2.2122
  3.9484   2.2202
  4.0845   2.2304
  4.2207   2.2360
  4.3568   2.2351
  4.4930   2.2433
  4.6291   2.2436
/
 }\relax
\Orange{\relax 
\plot
  0.0000  -0.0175
  0.1362   0.0101
  0.2723   0.0506
  0.4085   0.1268
  0.5446   0.2271
  0.6808   0.3431
  0.8169   0.4766
  0.9531   0.6294
  1.0892   0.7982
  1.2254   0.9769
  1.3615   1.1596
  1.4977   1.3428
  1.6338   1.5251
  1.7700   1.7057
  1.9061   1.8828
  2.0423   2.0534
  2.1784   2.2132
  2.3146   2.3573
  2.4507   2.4806
  2.5869   2.5797
  2.7230   2.6527
  2.8592   2.7006
  2.9953   2.7273
  3.1315   2.7390
  3.2676   2.7420
  3.4038   2.7415
  3.5399   2.7398
  3.6761   2.7382
  3.8122   2.7385
  3.9484   2.7427
  4.0845   2.7492
  4.2207   2.7506
  4.3568   2.7449
  4.4930   2.7481
  4.6291   2.7444
/
 }\relax
\Red{\relax 
\plot
  0.0000  -0.0175
  0.1362   0.0222
  0.2723   0.0672
  0.4085   0.1463
  0.5446   0.2486
  0.6808   0.3661
  0.8169   0.5005
  0.9531   0.6537
  1.0892   0.8226
  1.2254   1.0009
  1.3615   1.1829
  1.4977   1.3652
  1.6338   1.5462
  1.7700   1.7251
  1.9061   1.9004
  2.0423   2.0691
  2.1784   2.2267
  2.3146   2.3683
  2.4507   2.4893
  2.5869   2.5859
  2.7230   2.6565
  2.8592   2.7021
  2.9953   2.7266
  3.1315   2.7362
  3.2676   2.7373
  3.4038   2.7349
  3.5399   2.7314
  3.6761   2.7279
  3.8122   2.7264
  3.9484   2.7288
  4.0845   2.7333
  4.2207   2.7326
  4.3568   2.7255
  4.4930   2.7286
  4.6291   2.7253
/
 }\relax
\LimeGreen{\relax 
\plot
  0.0000  -0.0175
  0.1362   0.0185
  0.2723   0.0618
  0.4085   0.1396
  0.5446   0.2406
  0.6808   0.3569
  0.8169   0.4901
  0.9531   0.6420
  1.0892   0.8096
  1.2254   0.9867
  1.3615   1.1674
  1.4977   1.3484
  1.6338   1.5282
  1.7700   1.7058
  1.9061   1.8798
  2.0423   2.0472
  2.1784   2.2035
  2.3146   2.3440
  2.4507   2.4638
  2.5869   2.5594
  2.7230   2.6291
  2.8592   2.6739
  2.9953   2.6979
  3.1315   2.7070
  3.2676   2.7078
  3.4038   2.7052
  3.5399   2.7015
  3.6761   2.6979
  3.8122   2.6962
  3.9484   2.6986
  4.0845   2.7030
  4.2207   2.7026
  4.3568   2.6964
  4.4930   2.7004
  4.6291   2.6934
/
 }\relax
%
%\PineGreen{\relax 
%\plot
%/
% }\relax
%
%\Blue{\relax 
%\plot
%/
% }\relax
%
\endpicture
}   % end of final time data for phi2:  $\betau=\betav-1$, $T=1$
\xfigdim=0.14\xfiglen
\yfigdim=0.27\yfiglen
\setbox\figurefifteen=\vbox{\hsize=\xfiglen
% This is for: final time data for phi1:  $\betau=\betav=0.1$, $T=1$
\beginpicture
\footnotesize
  \setcoordinatesystem units <\xfigdim,\yfigdim>  point at 0 0
  \setplotarea x from 0 to 7.2, y from 0 to 3.7
  \axis bottom shiftedto y=0 ticks short numbered from 0 to 7 by 1 /
% unlabeled short quantity 8 / /
  \axis left ticks short numbered from 0 to 3 by 1 /
% unlabeled short quantity 16 / /
 \put {{\sevenrm $\phi_1(w)$}} [l] at 0.04 3.7
 \put {{\sevenrm $w$}} [rb] at 7.1  0.05
 \put {$\beta=0.1$} [rt] at 7 3.5
%\setquadratic
\setdashes <3pt>
\Black{
\plot
  0.0000   0.0000
  0.2113   0.4350
  0.4225   1.0052
  0.6338   1.6733
  0.8450   2.3516
  1.0563   2.9187
  1.2675   3.2628
  1.4788   3.3279
  1.6901   3.1361
  1.9013   2.7742
  2.1126   2.3540
  2.3238   1.9700
  2.5351   1.6755
  2.7463   1.4818
  2.9576   1.3729
  3.1689   1.3225
  3.3801   1.3066
  3.5914   1.3079
  3.8026   1.3166
  4.0139   1.3275
  4.2252   1.3387
  4.4364   1.3492
  4.6477   1.3589
  4.8589   1.3678
  5.0702   1.3761
  5.2814   1.3837
  5.4927   1.3907
  5.7040   1.3972
  5.9152   1.4033
  6.1265   1.4090
  6.3377   1.4143
  6.5490   1.4193
  6.7602   1.4239
  6.9715   1.4283
  7.1828   1.4325
/
 }\relax
\setsolid
\Goldenrod{\relax 
\plot
  0.0000   0.0577
  0.2113   0.4899
  0.4225   1.1186
  0.6338   1.8067
  0.8450   2.4815
  1.0563   3.0589
  1.2675   3.4351
  1.4788   3.5417
  1.6901   3.3872
  1.9013   3.0515
  2.1126   2.6479
  2.3238   2.2783
  2.5351   2.0021
  2.7463   1.8305
  2.9576   1.7396
  3.1689   1.6938
  3.3801   1.6665
  3.5914   1.6490
  3.8026   1.6465
  4.0139   1.6662
  4.2252   1.7061
  4.4364   1.7514
  4.6477   1.7821
  4.8589   1.7864
  5.0702   1.7706
  5.2814   1.7579
  5.4927   1.7719
  5.7040   1.8150
  5.9152   1.8588
  6.1265   1.8651
  6.3377   1.8328
  6.5490   1.8235
  6.7602   1.8850
  6.9715   1.8969
  7.1828   1.9550
/
 }\relax
\Orange{\relax 
\plot
  0.0000   0.0577
  0.2113   0.4342
  0.4225   1.0372
  0.6338   1.7023
  0.8450   2.3541
  1.0563   2.9076
  1.2675   3.2593
  1.4788   3.3417
  1.6901   3.1644
  1.9013   2.8086
  2.1126   2.3884
  2.3238   2.0055
  2.5351   1.7191
  2.7463   1.5393
  2.9576   1.4417
  3.1689   1.3901
  3.3801   1.3578
  3.5914   1.3357
  3.8026   1.3291
  4.0139   1.3450
  4.2252   1.3809
  4.4364   1.4220
  4.6477   1.4483
  4.8589   1.4481
  5.0702   1.4283
  5.2814   1.4122
  5.4927   1.4231
  5.7040   1.4625
  5.9152   1.5015
  6.1265   1.5023
  6.3377   1.4651
  6.5490   1.4523
  6.7602   1.5082
  6.9715   1.5111
  7.1828   1.6306
/
 }\relax
\Red{\relax 
\plot
  0.0000   0.0577
  0.2113   0.4265
  0.4225   1.0264
  0.6338   1.6892
  0.8450   2.3391
  1.0563   2.8912
  1.2675   3.2416
  1.4788   3.3230
  1.6901   3.1449
  1.9013   2.7884
  2.1126   2.3676
  2.3238   1.9843
  2.5351   1.6974
  2.7463   1.5174
  2.9576   1.4195
  3.1689   1.3678
  3.3801   1.3353
  3.5914   1.3131
  3.8026   1.3064
  4.0139   1.3221
  4.2252   1.3580
  4.4364   1.3991
  4.6477   1.4254
  4.8589   1.4252
  5.0702   1.4054
  5.2814   1.3893
  5.4927   1.4003
  5.7040   1.4399
  5.9152   1.4791
  6.1265   1.4802
  6.3377   1.4434
  6.5490   1.4312
  6.7602   1.4883
  6.9715   1.4922
  7.1828   1.5862
/
 }\relax
%
%\LimeGreen{\relax 
%\plot
%/
% }\relax
%
%\PineGreen{\relax 
%\plot
%/
% }\relax
%
%\Blue{\relax 
%\plot
%/
% }\relax
%
\endpicture
}   % end of final time data for phi1:  $\betau=\betav=0.1$, $T=1$
\xfigdim=0.14\xfiglen
\yfigdim=0.33\yfiglen
\setbox\figuresixteen=\vbox{\hsize=\xfiglen
% This is for: final time data for phi2:  $\betau=\betav=0.1$, $T=1$
\beginpicture
\footnotesize
  \setcoordinatesystem units <\xfigdim,\yfigdim>  point at 0 0
  \setplotarea x from 0 to 7.2, y from 0 to 3
  \axis bottom shiftedto y=0 ticks short numbered from 0 to 7 by 1 /
% unlabeled short quantity 8 / /
  \axis left ticks short numbered from 0 to 3 by 1 /
% unlabeled short quantity 16 / /
 \put {{\sevenrm $\phi_2(w)$}} [l] at 0.04 3
 \put {{\sevenrm $w$}} [rb] at 7 0.05
%\setquadratic
\setdashes <3pt>
\Black{
\plot
  0.0000   0.0000
  0.2113   0.0383
  0.4225   0.1456
  0.6338   0.3106
  0.8450   0.5220
  1.0563   0.7685
  1.2675   1.0387
  1.4788   1.3214
  1.6901   1.6052
  1.9013   1.8789
  2.1126   2.1310
  2.3238   2.3503
  2.5351   2.5256
  2.7463   2.6454
  2.9576   2.6984
  3.1689   2.7000
  3.3801   2.7000
  3.5914   2.7000
  3.8026   2.7000
  4.0139   2.7000
  4.2252   2.7000
  4.4364   2.7000
  4.6477   2.7000
  4.8589   2.7000
  5.0702   2.7000
  5.2814   2.7000
  5.4927   2.7000
  5.7040   2.7000
  5.9152   2.7000
  6.1265   2.7000
  6.3377   2.7000
  6.5490   2.7000
  6.7602   2.7000
  6.9715   2.7000
  7.1828   2.7000
/
 }\relax
\setsolid
\Goldenrod{\relax 
\plot
  0.0000   0.2030
  0.2113  -0.0501
  0.4225   0.0699
  0.6338   0.1946
  0.8450   0.3494
  1.0563   0.5887
  1.2675   0.8885
  1.4788   1.1932
  1.6901   1.4681
  1.9013   1.7099
  2.1126   1.9304
  2.3238   2.1379
  2.5351   2.3271
  2.7463   2.4807
  2.9576   2.5792
  3.1689   2.6142
  3.3801   2.5974
  3.5914   2.5584
  3.8026   2.5322
  4.0139   2.5408
  4.2252   2.5820
  4.4364   2.6311
  4.6477   2.6579
  4.8589   2.6478
  5.0702   2.6132
  5.2814   2.5863
  5.4927   2.5965
  5.7040   2.6447
  5.9152   2.6941
  6.1265   2.6956
  6.3377   2.6463
  6.5490   2.6268
  6.7602   2.7022
  6.9715   2.7106
  7.1828   2.7497
/
 }\relax
\Orange{\relax 
\plot
  0.0000   0.2030
  0.2113   0.0266
  0.4225   0.1777
  0.6338   0.3256
  0.8450   0.4993
  1.0563   0.7541
  1.2675   1.0665
  1.4788   1.3810
  1.6901   1.6633
  1.9013   1.9103
  2.1126   2.1342
  2.3238   2.3434
  2.5351   2.5327
  2.7463   2.6848
  2.9576   2.7808
  3.1689   2.8127
  3.3801   2.7927
  3.5914   2.7509
  3.8026   2.7221
  4.0139   2.7284
  4.2252   2.7669
  4.4364   2.8128
  4.6477   2.8358
  4.8589   2.8219
  5.0702   2.7839
  5.2814   2.7544
  5.4927   2.7622
  5.7040   2.8072
  5.9152   2.8518
  6.1265   2.8474
  6.3377   2.7931
  6.5490   2.7705
  6.7602   2.8403
  6.9715   2.8388
  7.1828   2.9291
/
 }\relax
\Red{\relax 
\plot
  0.0000   0.2030
  0.2113   0.0213
  0.4225   0.1701
  0.6338   0.3160
  0.8450   0.4878
  1.0563   0.7409
  1.2675   1.0515
  1.4788   1.3644
  1.6901   1.6451
  1.9013   1.8905
  2.1126   2.1128
  2.3238   2.3204
  2.5351   2.5081
  2.7463   2.6588
  2.9576   2.7532
  3.1689   2.7837
  3.3801   2.7622
  3.5914   2.7189
  3.8026   2.6887
  4.0139   2.6936
  4.2252   2.7308
  4.4364   2.7753
  4.6477   2.7970
  4.8589   2.7818
  5.0702   2.7425
  5.2814   2.7117
  5.4927   2.7183
  5.7040   2.7622
  5.9152   2.8058
  6.1265   2.8004
  6.3377   2.7452
  6.5490   2.7219
  6.7602   2.7916
  6.9715   2.7898
  7.1828   2.8731
/
 }\relax
%
%\LimeGreen{\relax 
%\plot
%/
% }\relax
%
%\PineGreen{\relax 
%\plot
%/
% }\relax
%
%\Blue{\relax 
%\plot
%/
% }\relax
%
\endpicture
}   % end of final time data for phi2:  $\betau=\betav=0.1$, $T=1$
\xfigdim=0.1\xfiglen
\yfigdim=0.27\yfiglen
\setbox\figureseventeen=\vbox{\hsize=\xfiglen
% This is for: final time data for phi1:  $\betau=\betav=1$, $T=1$
\beginpicture
\footnotesize
  \setcoordinatesystem units <\xfigdim,\yfigdim>  point at 0 0
  \setplotarea x from 0 to 10, y from 0 to 3.7
  \axis bottom shiftedto y=0 ticks short numbered from 0 to 10 by 2 /
% unlabeled short quantity 8 / /
  \axis left ticks short numbered from 0 to 3 by 1 /
% unlabeled short quantity 16 / /
 \put {{\sevenrm $\phi_1(w)$}} [l] at 0.04 3.7
 \put {{\sevenrm $w$}} [rb] at 10 0.05
 \put {$\beta=1$} [rt] at 9.2 3.5
%\setquadratic
\setdashes <3pt>
\Black{
\plot
  0.0000   0.0000
  0.2755   0.5949
  0.5511   1.4046
  0.8266   2.2949
  1.1021   3.0153
  1.3776   3.3320
  1.6532   3.1848
  1.9287   2.7207
  2.2042   2.1788
  2.4798   1.7427
  2.7553   1.4756
  3.0308   1.3502
  3.3063   1.3095
  3.5819   1.3077
  3.8574   1.3193
  4.1329   1.3339
  4.4085   1.3478
  4.6840   1.3605
  4.9595   1.3718
  5.2350   1.3821
  5.5106   1.3913
  5.7861   1.3997
  6.0616   1.4073
  6.3372   1.4143
  6.6127   1.4207
  6.8882   1.4266
  7.1637   1.4321
  7.4393   1.4372
  7.7148   1.4419
  7.9903   1.4463
  8.2659   1.4504
  8.5414   1.4542
  8.8169   1.4579
  9.0924   1.4613
  9.3680   1.4645
/
 }\relax
\setsolid
\Goldenrod{\relax 
\plot
  0.0000  -0.0155
  0.2755   0.6493
  0.5511   1.5064
  0.8266   2.4059
  1.1021   3.0954
  1.3776   3.3945
  1.6532   3.2852
  1.9287   2.8959
  2.2042   2.4133
  2.4798   1.9920
  2.7553   1.7074
  3.0308   1.5603
  3.3063   1.5113
  3.5819   1.5147
  3.8574   1.5379
  4.1329   1.5637
  4.4085   1.5856
  4.6840   1.6025
  4.9595   1.6159
  5.2350   1.6288
  5.5106   1.6436
  5.7861   1.6601
  6.0616   1.6758
  6.3372   1.6877
  6.6127   1.6956
  6.8882   1.7023
  7.1637   1.7111
  7.4393   1.7230
  7.7148   1.7357
  7.9903   1.7455
  8.2659   1.7501
  8.5414   1.7541
  8.8169   1.7667
  9.0924   1.7759
  9.3680   1.7859
/
 }\relax
\Orange{\relax 
\plot
  0.0000  -0.0155
  0.2755   0.6043
  0.5511   1.4402
  0.8266   2.3202
  1.1021   2.9903
  1.3776   3.2700
  1.6532   3.1428
  1.9287   2.7379
  2.2042   2.2428
  2.4798   1.8118
  2.7553   1.5196
  3.0308   1.3663
  3.3063   1.3119
  3.5819   1.3103
  3.8574   1.3285
  4.1329   1.3494
  4.4085   1.3665
  4.6840   1.3787
  4.9595   1.3875
  5.2350   1.3959
  5.5106   1.4062
  5.7861   1.4183
  6.0616   1.4296
  6.3372   1.4372
  6.6127   1.4409
  6.8882   1.4435
  7.1637   1.4483
  7.4393   1.4561
  7.7148   1.4647
  7.9903   1.4705
  8.2659   1.4714
  8.5414   1.4717
  8.8169   1.4804
  9.0924   1.4862
  9.3680   1.4916
/
 }\relax
\Red{\relax 
\plot
  0.0000  -0.0155
  0.2755   0.6001
  0.5511   1.4344
  0.8266   2.3133
  1.1021   2.9824
  1.3776   3.2614
  1.6532   3.1333
  1.9287   2.7276
  2.2042   2.2316
  2.4798   1.7998
  2.7553   1.5070
  3.0308   1.3533
  3.3063   1.2985
  3.5819   1.2966
  3.8574   1.3146
  4.1329   1.3353
  4.4085   1.3522
  4.6840   1.3642
  4.9595   1.3729
  5.2350   1.3812
  5.5106   1.3913
  5.7861   1.4033
  6.0616   1.4145
  6.3372   1.4220
  6.6127   1.4256
  6.8882   1.4280
  7.1637   1.4327
  7.4393   1.4405
  7.7148   1.4490
  7.9903   1.4545
  8.2659   1.4550
  8.5414   1.4551
  8.8169   1.4636
  9.0924   1.4686
  9.3680   1.4795
/
 }\relax
%
%\LimeGreen{\relax 
%\plot
%/
% }\relax
%
%\PineGreen{\relax 
%\plot
%/
% }\relax
%
%\Blue{\relax 
%\plot
%/
% }\relax
%
\endpicture
}   % end of final time data for phi1:  $\betau=\betav=1$, $T=1$
\xfigdim=0.1\xfiglen
\yfigdim=0.33\yfiglen
\setbox\figureeighteen=\vbox{\hsize=\xfiglen
% This is for: final time data for phi2:  $\betau=\betav=1$, $T=1$
\beginpicture
\footnotesize
  \setcoordinatesystem units <\xfigdim,\yfigdim>  point at 0 0
  \setplotarea x from 0 to 10, y from 0 to 3
  \axis bottom shiftedto y=0 ticks short numbered from 0 to 10 by 2 /
% unlabeled short quantity 8 / /
  \axis left ticks short numbered from 0 to 3 by 1 /
% unlabeled short quantity 16 / /
 \put {{\sevenrm $\phi_2(w)$}} [l] at 0.04 3
 \put {{\sevenrm $w$}} [rb] at 10 0.05
%\setquadratic
\setdashes <3pt>
\Black{
\plot
  0.0000   0.0000
  0.2755   0.0641
  0.5511   0.2398
  0.8266   0.5020
  1.1021   0.8255
  1.3776   1.1852
  1.6532   1.5561
  1.9287   1.9130
  2.2042   2.2309
  2.4798   2.4846
  2.7553   2.6490
  3.0308   2.7000
  3.3063   2.7000
  3.5819   2.7000
  3.8574   2.7000
  4.1329   2.7000
  4.4085   2.7000
  4.6840   2.7000
  4.9595   2.7000
  5.2350   2.7000
  5.5106   2.7000
  5.7861   2.7000
  6.0616   2.7000
  6.3372   2.7000
  6.6127   2.7000
  6.8882   2.7000
  7.1637   2.7000
  7.4393   2.7000
  7.7148   2.7000
  7.9903   2.7000
  8.2659   2.7000
  8.5414   2.7000
  8.8169   2.7000
  9.0924   2.7000
  9.3680   2.7000
/
 }\relax
\setsolid
\Goldenrod{\relax 
\plot
  0.0000   0.0229
  0.2755   0.0419
  0.5511   0.2147
  0.8266   0.4728
  1.1021   0.7884
  1.3776   1.1438
  1.6532   1.5183
  1.9287   1.8824
  2.2042   2.2010
  2.4798   2.4441
  2.7553   2.5994
  3.0308   2.6769
  3.3063   2.7034
  3.5819   2.7082
  3.8574   2.7108
  4.1329   2.7171
  4.4085   2.7236
  4.6840   2.7264
  4.9595   2.7267
  5.2350   2.7293
  5.5106   2.7376
  5.7861   2.7497
  6.0616   2.7604
  6.3372   2.7657
  6.6127   2.7668
  6.8882   2.7680
  7.1637   2.7729
  7.4393   2.7815
  7.7148   2.7917
  7.9903   2.7993
  8.2659   2.8003
  8.5414   2.7998
  8.8169   2.8122
  9.0924   2.8197
  9.3680   2.8245
/
 }\relax
\Orange{\relax 
\plot
  0.0000   0.0229
  0.2755   0.0670
  0.5511   0.2496
  0.8266   0.5142
  1.1021   0.8343
  1.3776   1.1922
  1.6532   1.5674
  1.9287   1.9305
  2.2042   2.2465
  2.4798   2.4857
  2.7553   2.6362
  3.0308   2.7086
  3.3063   2.7298
  3.5819   2.7292
  3.8574   2.7267
  4.1329   2.7278
  4.4085   2.7291
  4.6840   2.7269
  4.9595   2.7222
  5.2350   2.7199
  5.5106   2.7233
  5.7861   2.7305
  6.0616   2.7362
  6.3372   2.7366
  6.6127   2.7328
  6.8882   2.7293
  7.1637   2.7294
  7.4393   2.7333
  7.7148   2.7387
  7.9903   2.7415
  8.2659   2.7379
  8.5414   2.7330
  8.8169   2.7405
  9.0924   2.7434
  9.3680   2.7429
/
 }\relax
\Red{\relax 
\plot
  0.0000   0.0229
  0.2755   0.0629
  0.5511   0.2437
  0.8266   0.5070
  1.1021   0.8259
  1.3776   1.1828
  1.6532   1.5571
  1.9287   1.9193
  2.2042   2.2345
  2.4798   2.4729
  2.7553   2.6225
  3.0308   2.6941
  3.3063   2.7144
  3.5819   2.7130
  3.8574   2.7096
  4.1329   2.7099
  4.4085   2.7105
  4.6840   2.7075
  4.9595   2.7021
  5.2350   2.6991
  5.5106   2.7018
  5.7861   2.7084
  6.0616   2.7135
  6.3372   2.7132
  6.6127   2.7087
  6.8882   2.7045
  7.1637   2.7041
  7.4393   2.7074
  7.7148   2.7122
  7.9903   2.7143
  8.2659   2.7098
  8.5414   2.7041
  8.8169   2.7113
  9.0924   2.7131
  9.3680   2.7143
/
 }\relax
%
%\LimeGreen{\relax 
%\plot
%/
% }\relax
%
%\PineGreen{\relax 
%\plot
%/
% }\relax
%
%\Blue{\relax 
%\plot
%/
% }\relax
%
\endpicture
}   % end of final time data for phi2:  $\betau=\betav=1$, $T=1$
\xfigdim=0.033\xfiglen
\yfigdim=0.27\yfiglen
\setbox\figurenineteen=\vbox{\hsize=\xfiglen
% This is for: final time data for phi1:  $\betau=\betav=10$, $T=1$
\beginpicture
\footnotesize
  \setcoordinatesystem units <\xfigdim,\yfigdim>  point at 0 0
  \setplotarea x from 0 to 31, y from 0 to 3.7
  \axis bottom shiftedto y=0 ticks short numbered from 0 to 30 by 5 /
% unlabeled short quantity 8 / /
  \axis left ticks short numbered from 0 to 3 by 1 /
% unlabeled short quantity 16 / /
 \put {{\sevenrm $\phi_1(w)$}} [l] at 0.04 3.7
 \put {{\sevenrm $w$}} [rb] at 31 0.05
 \put {$\beta=10$} [rt] at 29 3.5
\setquadratic
\setdashes <3pt>
\Black{
\plot
  0.0000   0.0000
  0.9139   2.5549
  1.8279   2.9123
  2.7418   1.4850
  3.6558   1.3101
  4.5697   1.3554
  5.4837   1.3904
  6.3976   1.4157
  7.3115   1.4349
  8.2255   1.4498
  9.1394   1.4618
 10.0534   1.4717
 10.9673   1.4799
 11.8813   1.4868
 12.7952   1.4928
 13.7091   1.4980
 14.6231   1.5025
 15.5370   1.5065
 16.4510   1.5101
 17.3649   1.5133
 18.2789   1.5161
 19.1928   1.5187
 20.1067   1.5211
 21.0207   1.5233
 21.9346   1.5252
 22.8486   1.5271
 23.7625   1.5287
 24.6765   1.5303
 25.5904   1.5317
 26.5043   1.5331
 27.4183   1.5343
 28.3322   1.5355
 29.2462   1.5366
 30.1601   1.5377
 31.0741   1.5386
/
 }\relax
\setsolid
\Goldenrod{\relax 
\plot
  0.0000  -0.3715
  0.9139   2.3400
  1.8279   2.6117
  2.7418   1.9558
  3.6558   1.4236
  4.5697   1.2959
  5.4837   1.3857
  6.3976   1.4727
  7.3115   1.4869
  8.2255   1.4737
  9.1394   1.4874
 10.0534   1.5297
 10.9673   1.5641
 11.8813   1.5650
 12.7952   1.5448
 13.7091   1.5373
 14.6231   1.5589
 15.5370   1.5915
 16.4510   1.6041
 17.3649   1.5903
 18.2789   1.5763
 19.1928   1.5883
 20.1067   1.6159
 21.0207   1.6248
 21.9346   1.6084
 22.8486   1.6034
 23.7625   1.6286
 24.6765   1.6418
 25.5904   1.6217
 26.5043   1.6280
 27.4183   1.6545
 28.3322   1.6319
 29.2462   1.6567
 30.1601   1.6384
 31.0741   1.6833
/
 }\relax
\Orange{\relax 
\plot
  0.0000  -0.3715
  0.9139   2.3261
  1.8279   2.5911
  2.7418   1.9301
  3.6558   1.3939
  4.5697   1.2626
  5.4837   1.3490
  6.3976   1.4328
  7.3115   1.4438
  8.2255   1.4276
  9.1394   1.4384
 10.0534   1.4779
 10.9673   1.5095
 11.8813   1.5075
 12.7952   1.4847
 13.7091   1.4746
 14.6231   1.4936
 15.5370   1.5234
 16.4510   1.5335
 17.3649   1.5171
 18.2789   1.5007
 19.1928   1.5101
 20.1067   1.5352
 21.0207   1.5415
 21.9346   1.5227
 22.8486   1.5153
 23.7625   1.5379
 24.6765   1.5487
 25.5904   1.5264
 26.5043   1.5301
 27.4183   1.5544
 28.3322   1.5295
 29.2462   1.5522
 30.1601   1.5315
 31.0741   1.5731
/
 }\relax
%
%\Red{\relax 
%\plot
%/
% }\relax
%
%\LimeGreen{\relax 
%\plot
%/
% }\relax
%
%\PineGreen{\relax 
%\plot
%/
% }\relax
%
%\Blue{\relax 
%\plot
%/
% }\relax
%
\endpicture
}   % end of final time data for phi1:  $\betau=\betav=10$, $T=1$
\xfigdim=0.033\xfiglen
\yfigdim=0.33\yfiglen
\setbox\figuretwenty=\vbox{\hsize=\xfiglen
% This is for: final time data for phi2:  $\betau=\betav=10$, $T=1$
\beginpicture
\footnotesize
  \setcoordinatesystem units <\xfigdim,\yfigdim>  point at 0 0
  \setplotarea x from 0 to 31, y from 0 to 3
  \axis bottom shiftedto y=0 ticks short numbered from 0 to 30 by 5 /
% unlabeled short quantity 8 / /
  \axis left ticks short numbered from 0 to 3 by 1 /
% unlabeled short quantity 16 / /
 \put {{\sevenrm $\phi_2(w)$}} [l] at 0.04 3
 \put {{\sevenrm $w$}} [rb] at 31 0.05
\setquadratic
\setdashes <3pt>
\Black{
\plot
  0.0000   0.0000
  0.9139   0.5991
  1.8279   1.7856
  2.7418   2.6435
  3.6558   2.7000
  4.5697   2.7000
  5.4837   2.7000
  6.3976   2.7000
  7.3115   2.7000
  8.2255   2.7000
  9.1394   2.7000
 10.0534   2.7000
 10.9673   2.7000
 11.8813   2.7000
 12.7952   2.7000
 13.7091   2.7000
 14.6231   2.7000
 15.5370   2.7000
 16.4510   2.7000
 17.3649   2.7000
 18.2789   2.7000
 19.1928   2.7000
 20.1067   2.7000
 21.0207   2.7000
 21.9346   2.7000
 22.8486   2.7000
 23.7625   2.7000
 24.6765   2.7000
 25.5904   2.7000
 26.5043   2.7000
 27.4183   2.7000
 28.3322   2.7000
 29.2462   2.7000
 30.1601   2.7000
 31.0741   2.7000
/
 }\relax
\setsolid
\Goldenrod{\relax 
\plot
  0.0000  -0.0977
  0.9139   0.7405
  1.8279   1.7675
  2.7418   2.4749
  3.6558   2.7700
  4.5697   2.8050
  5.4837   2.7623
  6.3976   2.7409
  7.3115   2.7522
  8.2255   2.7694
  9.1394   2.7726
 10.0534   2.7652
 10.9673   2.7618
 11.8813   2.7700
 12.7952   2.7832
 13.7091   2.7902
 14.6231   2.7875
 15.5370   2.7833
 16.4510   2.7873
 17.3649   2.7993
 18.2789   2.8082
 19.1928   2.8064
 20.1067   2.8006
 21.0207   2.8037
 21.9346   2.8163
 22.8486   2.8233
 23.7625   2.8184
 24.6765   2.8178
 25.5904   2.8291
 26.5043   2.8314
 27.4183   2.8280
 28.3322   2.8403
 29.2462   2.8342
 30.1601   2.8472
 31.0741   2.8335
/
 }\relax
\Orange{\relax 
\plot
  0.0000  -0.0977
  0.9139   0.7245
  1.8279   1.7441
  2.7418   2.4451
  3.6558   2.7344
  4.5697   2.7641
  5.4837   2.7165
  6.3976   2.6904
  7.3115   2.6974
  8.2255   2.7104
  9.1394   2.7095
 10.0534   2.6981
 10.9673   2.6909
 11.8813   2.6953
 12.7952   2.7048
 13.7091   2.7081
 14.6231   2.7018
 15.5370   2.6941
 16.4510   2.6947
 17.3649   2.7032
 18.2789   2.7087
 19.1928   2.7036
 20.1067   2.6944
 21.0207   2.6942
 21.9346   2.7035
 22.8486   2.7072
 23.7625   2.6990
 24.6765   2.6952
 25.5904   2.7031
 26.5043   2.7022
 27.4183   2.6955
 28.3322   2.7045
 29.2462   2.6949
 30.1601   2.7045
 31.0741   2.6878
/
 }\relax
%
%\Red{\relax 
%\plot
%/
% }\relax
%
%\LimeGreen{\relax 
%\plot
%/
% }\relax
%
%\PineGreen{\relax 
%\plot
%/
% }\relax
%
%\Blue{\relax 
%\plot
%/
% }\relax
%
\endpicture
}   % end of final time data for phi2:  $\betau=\betav=10$, $T=1$

\begin{figure}[!ht]
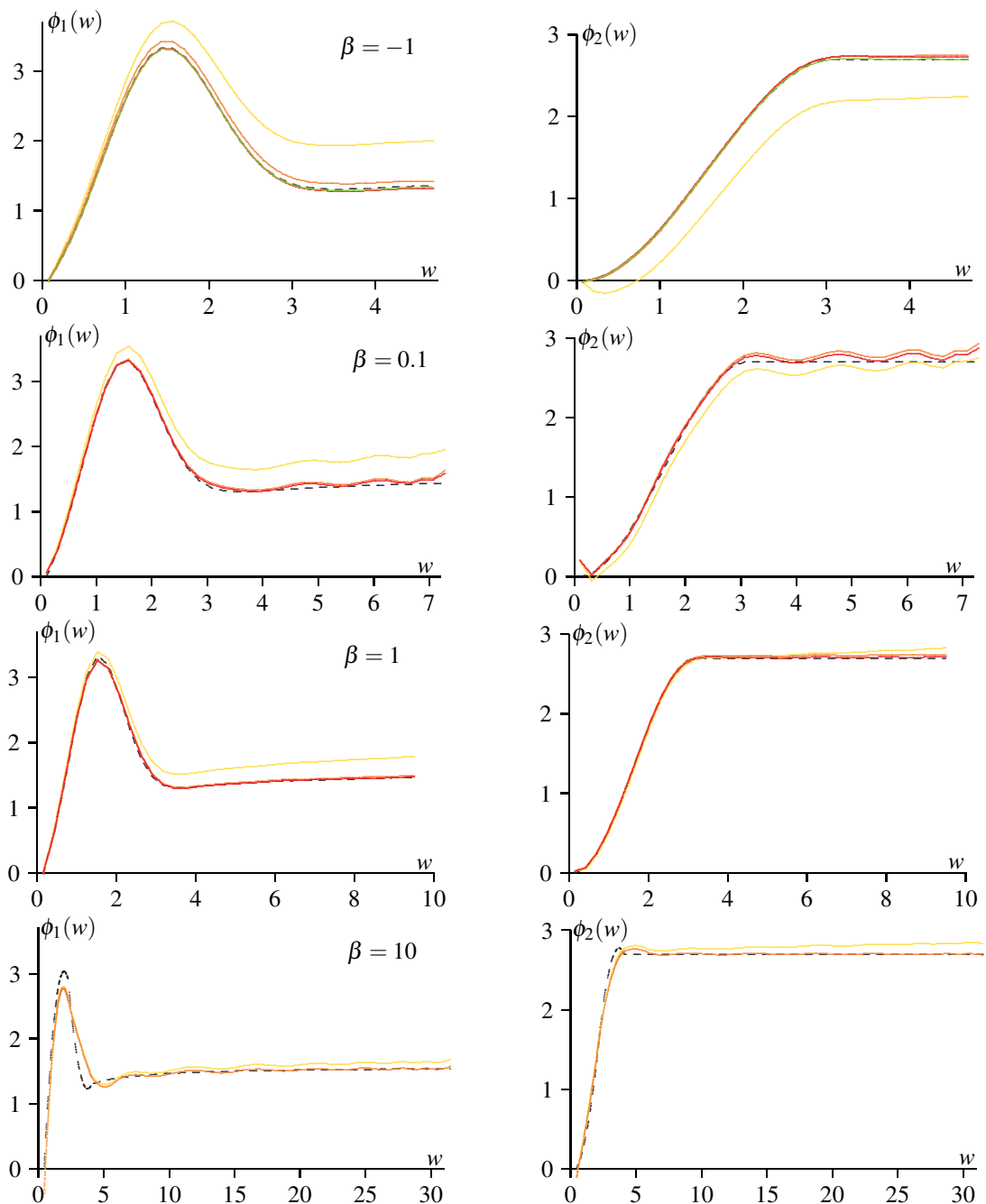

\hbox to\hsize{\hss\copy\figurethirteen\hss\hss\hss\copy\figurefourteen\hss}
\smallskip
\hbox to\hsize{\hss\copy\figurefifteen\hss\hss\hss\copy\figuresixteen\hss}
\smallskip
\hbox to\hsize{\hss\copy\figureseventeen\hss\hss\hss\copy\figureeighteen\hss}
\smallskip
\hbox to\hsize{\hss\copy\figurenineteen\hss\hss\hss\copy\figuretwenty\hss}
\caption{\small Reconstructions of $\phi_1$ and $\phi_2\;$ from time trace data. }
\label{fig:titr_reconstructions_phi}
\end{figure}

In Figure~\oldref{fig:cgt_rates_phi_beta} we plot 
the relative norms of
$\|\phi_i^{(n)} - \phi_{i,\mbox{\footnotesize actual}}\|$
(that is, the ratio of this quantity to
$\|\phi_i^{(0)} - \phi_{i,\mbox{\footnotesize actual}}\|$ as a 
function of the iteration index $n$, 
for each of these $\beta$-values.
%$\beta=-1,\;0.1,\;1,\;10$.
Here $\phi_i^{(0)}$  denotes the initial approximation to the 
$i^{\mbox{th}}$ function (in this case it was the zero function).

\begin{figure}[h]
\hbox to\hsize{\hss\copy\figurefive\hss\hss\copy\figuresix\hss}
\medskip
\hbox to\hsize{\hss\copy\figureseven\hss\hss\copy\figureeight\hss}
\smallskip
\caption{\small
Convergence rates of $\phi$ iterations as a function of $\beta$.
Top: time trace, Bottom: final time.
}
\label{fig:cgt_rates_phi_beta}
\end{figure}

These plots also show the case of $\beta_u=-\beta_v=1$ and setting instead
$w=u^2v$.
Such an  asymmetrical relationship occurs frequently in chemical reactions.
As an example, if $\phi_i(w)=w$, and $\beta_u$, $\beta_v$ have opposite
signs this is the classical ``Brusselator'' model for
a Belousov-Zhabotinsky auto-catalytic reaction.
It is the theoretical underpinning of the chemistry ``magic trick''
of a jar of liquid changing colour from red to blue and back again in a
repeating cycle.  Other colour variations are also possible.
The two chemicals used are often potassium bromate and cerium sulphate in an
acid base.
This is just a specific example of a wider class of periodic cycle
solutions of the Turing type, \cite{Turing:1952}.
See also \cite{WalgraefAifantis:1985,Pearson:1993}.
We remark here that from either an analytical or computational perspective
other choices of $w$ are possible.
For example, we could choose $w(u,v) = \sqrt{u^2 + v^2}$ or in the case
with more equations as $w(\vec u) = \|\vec u\|$.

In these references $\phi_1$ and $\phi_1$ have simple (low degree
polynomial) form. The ability to go beyond this and determine
a more complex form has clear physical applicability.

\section*{Appendix}
\paragraph{Proof of Lemma \oldref{lem:Kbar}}

We use a theorem from \cite{Friedman1964}, which we here quote for convenience of the reader.
\begin{quote}
Theorem 6 (Friedman \cite[Theorem 1, page 92]{Friedman1964})\\
Assume that $\mathbb{L}u=\sum_{i,j=1}^d a_{i,j} u_{x_i,x_j} + \sum_{i=1}^d b_{i} u_{x_i} + c u$ 
%(in non-divergence form!) 
with $\|a_{i,j}\|_{C^{0,\beta}(Q)}\leq K_1$,\\ $\|b_{i}\|_{C^{1,\beta}(Q)}\leq K_1$, $\|c\|_{C^{2,\beta}(Q)}\leq K_1$, that for any real vector $\xi$ and $(x,t)\in Q=\Omega\times(0,T)$ the inequality $\sum_{i,j=1}^d a_{i,j}(x,t)\xi_i\xi_j \geq K_2 |\xi|^2$ with $K_2>0$, and that $j\in C^\beta(Q)$.
\footnote{Note that in \cite[Theorem 1, page 92]{Friedman1964}, $|f|_{2,\alpha}\sim |f|_{\alpha}$ on a smooth domain $\Omega$}
\\
Then there exists a constant $K_0$ depending only on $K_1$, $K_2$, $d$ such that if $w$ is a solution of
$w_t-\mathbb{L} w =j$ in $Q$
with $w$, $w_t$, $w_{x_i}$, $w_{x_i,x_j}$ $\in C^{0,\beta}(Q)$ satisfies the estimate
\begin{equation}\label{Thm1p92Friedman1964}
\|w_{x_i,x_j}\|_{C^{0,\beta}(Q)}\leq K_0(\|w\|_{C(Q)}+\|j\|_{C^{0,\beta}(Q)})
\end{equation}
\end{quote}
Thus it suffices to estimate $\|w\|_{C(Q)}$ on the right hand side in terms of $\|j\|_{C^{0,\beta}(Q)}$, which we do by means of a maximum principle, similarly to Lemma \oldref{lem:expdecay}.
More precisely, since under the assumption \eqref{eqn:LA} $v:=\frac12 z^2$ solves   
\[
\begin{aligned}
v_t-\mathbb{L} v +\check{q} v&=(\check{q}-c) w^2 -\nabla w^T \underline{A} \nabla w + w\, j%\\ 
%&
 \leq  (\check{c}_1+\epsilon) w^2 +\frac{1}{4\epsilon} j^2
\mbox{ in }\Omega\times(0,T) \\
\frac{\partial v}{\partial\nu} + \gamma v&= 0\quad \mbox{ on }\partial \Omega\times(0,T) \\
v(x,0)&=0 \quad  x\in\Omega\,.
\end{aligned}
\]
where we have used Young's inequality with some $\epsilon>0$ to be chosen below, as well as nonnegativity of $w^2$ and \eqref{eqn:c1check_lem}.
Hence by the maximum principle $0\leq v(x,t)\leq\bar{v}(x,t)$ for all $(x,t)\in Q$, where $\bar{v}$ solves
\[
\bar{v}_t-\mathbb{L} \bar{v} +\check{q} \bar{v}= 2(\check{c}_1+\epsilon) v +\frac{1}{4\epsilon} j^2 \] 
in $Q$ with homogeneous initial and boundary conditions.
The values of $v,\bar{v}$ can be estimated analogously to \eqref{eqn:CA}, \eqref{eqn:expmuz} as follows
\[
\begin{aligned}
\|e^{\mu t} v(t)\|_{C(\Omega)}
&\leq\|e^{\mu t} \bar{v}(t)\|_{C^1(\Omega)}
\leq C
\Bigl(2(\check{c}_1+\epsilon)\|v\|_{L^p(\Omega\times(0,t))} + \tfrac{1}{4\epsilon}\|j^2\|_{L^p(\Omega\times(0,t))}\Bigr) \\
&\leq C |\Omega|^{1/p} t^{1/p}
\Bigl(2(\check{c}_1+\epsilon)\|v\|_{C(\Omega\times(0,t))} + \tfrac{1}{4\epsilon}\|j^2\|_{C(\Omega\times(0,t))}\Bigr) 
\end{aligned}
\]
with $C=C_{W^{\theta,p},C}^{\mathbb{R^+}} C_{W^{2-2\theta,p},C^1}^\Omega C^A_\mu$, i.e., 
\[
\|v(t)\|_{C(\Omega)}
\leq  
C |\Omega|^{1/p} t^{1/p} e^{-\mu t} 
\Bigl(2(\check{c}_1+\epsilon)\|v\|_{C(\Omega\times(0,t))} + \tfrac{1}{4\epsilon}\|j\|_{C(\Omega\times(0,t))}\Bigr)\,,
\]
hence by taking the supremum over $t\in[0,T]$ and using $\sup_{t\in\mathbb{R}} e^{-\mu t} t^{1/p} = (ep\mu)^{-1/p}$
\[
\|v\|_{C(Q)}
\leq  
C |\Omega|^{1/p} (ep\mu)^{-1/p}
\Bigl(2(\check{c}_1+\epsilon)\|v\|_{C(Q)} + \tfrac{1}{4\epsilon}\|j^2\|_{C(Q)}\Bigr)\,.
\]
Thus with $\check{c}_1+\epsilon$ sufficiently small, more precisely $\check{c}_1+\epsilon< \frac{1}{2C}|\Omega|^{-1/p} (ep\mu)^{1/p}$,

\[
\|w\|_{C(Q)}^2\leq
2\|v\|_{C(Q)}
\leq  \frac{C |\Omega|^{1/p} (ep\mu)^{-1/p}}{4\epsilon(1-2C(\check{c}_1+\epsilon)|\Omega|^{1/p} (ep\mu)^{-1/p})}\|j\|_{C(Q)}^2\,,
\]
which together with \eqref{Thm1p92Friedman1964} yields the assertion.

\section{Epilogue}

It is tempting for authors to wonder how a paper will be received and in
this case the answer to the question is likely to depend on the community
to which the reader belongs.

The practitioners might feel not enough attention was given for complete
answers to specific problems or the range of problems was insufficient;
``why was $\;\ldots\;$ not tackled?
Mathematicians might have liked to see theorems containing
``sufficently small'' conditions replaced by estimates with tangible
values.
The inverse problems community seeing the complexities arising from
the unknown ranges and nonlinearities themselves, might reflect,
``but linear inverse problems/equations behave even better.''

In some sense these are valid statements.
However, reaction diffusion systems are able to model an enormous range
of physical problems and coupled systems of nonlinear equations are
always going to impose mathematical difficulties.
There are indeed easier inverse problems, but it is the above ubiquity
and challenges that make them compelling.

We hope to continue this work by expanding the range of questions posed,
by looking for better analytic tools and superior computational
methods.
There is much, much more still to be said.

\section*{Acknowledgment}

\noindent
The work of the first author was supported by the Austrian Science Fund {\sc fwf}
under the grant P30054.

\noindent
The work of the second author was supported 
in part by the
National Science Foundation through award {\sc dms}-1620138.

%%%%%%%%%%%%%%%%%%%%%%%%%%%%%%%%%%%%%%%%%%%%

\bibliographystyle{plain}
\bibliography{bk-br}

\end{document}